\documentclass[12pt]{article}

\usepackage{amsmath,amscd,amssymb,amsthm,amstext,mathrsfs,dcolumn,
booktabs,enumerate,color}
 \usepackage[pdftex]{graphicx} 

\theoremstyle{plain}
   \newtheorem{theorem}{Theorem}[section]
    \newtheorem{conjecture}[theorem]{Conjecture}
   \newtheorem{proposition}[theorem]{Proposition}
   \newtheorem{lemma}[theorem]{Lemma}
   \newtheorem{corollary}[theorem]{Corollary}
   \theoremstyle{definition}

   \theoremstyle{remark}

\newcommand{\news}{\setcounter{equation}{0}}

\newcommand{\eps}{\varepsilon}
\newcommand{\RR}{{\mathbb R}}

\newcommand{\ZZ}{{\mathbb Z}}
\newcommand{\CC}{{\mathbb C}}
\newcommand{\PP}{{\mathbb P}}
\newcommand{\TT}{{\mathbb T}}
\newcommand{\DD}{{\mathbb D}}

\newcommand{\Vol}{{\rm Vol}}
\newcommand{\Vols}{{\rm Vol}(\Sigma)}

\newcommand{\Sc}{{\rm Scal}}
\newcommand{\Sym}{{\rm Sym}}
\newcommand{\sech}{{\rm sech}}
\newcommand{\tp}{{\scriptscriptstyle{+}}}
\newcommand{\tm}{{\scriptscriptstyle{-}}}
\newcommand{\tpm}{{\scriptscriptstyle{\pm}}}
\newcommand{\tmp}{{\scriptscriptstyle{\mp}}}

\renewcommand{\AA}{\mathcal{A}}
\renewcommand{\P}{{\mathbb P}}
\newcommand{\R}{{\mathbb{R}}}
\newcommand{\Z}{{\mathbb{Z}}}

\newcommand{\C}{{\mathbb{C}}}
\newcommand{\I}{{\mathbb{I}}}

\newcommand{\beq}{\begin{equation}}
\newcommand{\eeq}{\end{equation}}
\newcommand{\bea}{\begin{eqnarray}}
\newcommand{\eea}{\end{eqnarray}}
\newcommand{\ben}{\begin{eqnarray*}}
\newcommand{\een}{\end{eqnarray*}}
\newcommand{\bem}{\begin{enumerate}}
\newcommand{\eem}{\end{enumerate}}
\newcommand{\ii}{\item}
\newcommand{\ra}{\rightarrow}

\newcommand{\ds}{\displaystyle}
\newcommand{\cd}{\partial}

\newcommand{\less}{\backslash}
\newcommand{\M}{{\sf M}}

\newcommand{\ol}{\overline}
\def \d{\mathrm{d}}

\newcommand{\ip}[1]{\langle #1 \rangle}
\newcommand{\ignore}[1]{}

\newcommand{\swap}[2]{{\color{blue} #2}}

\begin{document}
\title{The geometry of the space of BPS vortex-antivortex pairs}
\author{N.M.~Rom\~ao \\ Email: {\tt n.m.romao@gmail.com}, 
ORCID 0000-0003-3234-6076\\
Institut f\"ur Mathematik, Universit\"at Augsburg, 86135 Augsburg, Germany
\vspace*{0.5cm}
\\
J.M.~Speight\thanks{Corresponding author.} \\
Email: {\tt speight@maths.leeds.ac.uk}, ORCID 0000-0002-6844-9539\\
School of Mathematics, University of Leeds, Leeds LS2 9JT, UK}
\date{20 April, 2020}

\maketitle

\begin{abstract} 
The gauged sigma model with target $\P^1$, defined on a Riemann surface $\Sigma$, supports static solutions in which $k_\tp$ vortices coexist in
stable equilibrium with $k_\tm$ antivortices. Their moduli space is a noncompact complex manifold $\M_{(k_\tp,k_\tm)}(\Sigma)$ of dimension $k_\tp+k_\tm$ which inherits a 
natural K\"ahler metric $g_{L^2}$ governing the model's low energy dynamics. This
paper presents the first detailed study of $g_{L^2}$, focussing on the
geometry close to the boundary divisor $D=\partial \, \M_{(k_\tp,k_\tm)}(\Sigma)$.

On $\Sigma=S^2$, rigorous estimates of $g_{L^2}$ close to $D$ are obtained which imply that $\M_{(1,1)}(S^2)$ has finite volume and is geodesically incomplete. On $\Sigma=\R^2$, careful numerical analysis and a point-vortex formalism are used to conjecture asymptotic formulae for $g_{L^2}$ in the limits of small and large separation. All these results make use of a localization formula, expressing $g_{L^2}$ in terms of data at the (anti)vortex positions, which is established for
general $\M_{(k_\tp,k_\tm)}(\Sigma)$.  

For arbitrary compact $\Sigma$, a natural compactification of the space ${\M}_{(k_\tp,k_\tm)}(\Sigma)$ is proposed in terms of a certain limit of gauged linear sigma models, leading to formulae for its volume and total scalar curvature. The volume formula agrees with the result established for ${\rm Vol}({\sf M}_{(1,1)}(S^2))$, and allows for a detailed study of the thermodynamics of vortex-antivortex gas mixtures. It is found that the equation of state is independent of the genus of $\Sigma$, and that the entropy of mixing is always positive.

\end{abstract}

\vspace{10pt}
MSC2010: {\tt 53C80; 70S15, 35Q70}

\newpage

\section{Introduction}\news

Gauged sigma models at critical coupling stand out among the most tractable classes of field theories in two spatial dimensions,
for they can be explored  with a secure mathematical scaffolding provided by self-duality. There is a topological lower bound on the energy of a field configuration, attained by solutions of a system of first order PDEs called the {\em vortex equations}. Such solutions are conventionally called BPS, in analogy with an early and famous example fitting into this general framework, the Bogomol'ny\u\i--Prasad--Sommerfield monopole~\cite[p.~4]{AtiHit}. The moduli space of solutions of the vortex equations carries a natural K\"ahler structure which, apart from its intrinsic interest, is a powerful instrument to probe the low-energy structure of the various field theories (both classical and quantum) that
one may associate to these models.

In order to define a gauged sigma model on an oriented Riemannian surface $\Sigma$, one needs to specify:
\bem
\ii a K\"ahler manifold $X$ (the target space);
\ii a compact Lie group $\mathbb T$ with bi-invariant metric (the structure group);
\ii a moment map $\mu:X\ra\mathfrak{t}^*$ for a holomorphic and Hamiltonian action $\rho$ of $\mathbb T$ on $X$.
\eem
Given these data, one fixes a principal $\mathbb T$-bundle $\pi_P: P \rightarrow \Sigma$, and associates to any connexion $A$ in $P$ and section
$\phi$ of the associated bundle $P^X_\rho:=P\times_\rho X\rightarrow \Sigma$ (equivalently, a $\mathbb T$-equivariant map $\phi: P\rightarrow X$),
the energy 
\begin{equation} \label{GSM}
\mathsf{E}(A,\phi) :=\frac{1}{2} \int_\Sigma  \left(  |F_{A}|^2 + |{\rm d^A \phi}|^2 +  |\mu \circ \phi |^2 \right),
\end{equation}
where, as usual, {$F_A$} denotes the curvature of $A$. Classical solutions of the model are critical
points of this functional. Those which minimize ${\sf E}$ within their homotopy class are of particular interest. 
One speaks of a gauged {\em linear} sigma model (GLSM) if $X$ is a Hermitian vector space and $\rho$ is a unitary representation of $\mathbb T$.

The vortex equations associated to (\ref{GSM}) are the PDEs on $\Sigma$
\begin{equation}\label{vort1}
\bar \partial^A \phi =0,
\end{equation}
expressing that $\phi$ is a holomorphic section (with respect to $A$ and the complex structures on $\Sigma$ and $X$), together with
\begin{equation} \label{vort2}
*F_A + \mu^\sharp  \circ \phi =0,
\end{equation}
where $*$ is the Hodge operator of the metric $g_\Sigma$ on $\Sigma$,  and $\mu^\sharp:X\rightarrow \mathfrak{t}$ denotes  the composition of $\mu$ with the isomorphism $\sharp:\mathfrak{t}^*\ra\mathfrak{t}$ defined by the metric on $\mathbb T$. 
When $\Sigma$ is noncompact or has a boundary, these two equations need to be supplemented by appropriate boundary conditions, but at this stage we will simplify our discussion
by assuming that $\Sigma$ is compact.

Equations (\ref{vort1}) and (\ref{vort2}) are invariant under the group $\mathcal{G}(P):={\rm Aut}_\Sigma(P)$ of automorphisms of $\pi_P$, which acts on pairs $(A,\phi)$ via
gauge transformations $(A,\phi) \mapsto ({\rm Ad}_{\gamma}A -  \gamma^{-1} {\rm d}\gamma, \rho(\gamma) \phi)$. The $\mathcal{G}(P)$-orbits of vortex solutions turn out to be more natural objects
than the solutions themselves, and their classification can be addressed as a moduli problem. The relevant topological invariant in this situation is a $\mathbb T$-equivariant 2-homology class (see e.g.\ \cite{CGMS})
\begin{equation}\label{theinvariant}
{[\phi]_2^\TT}  \in  H^\TT_2(X;\ZZ)
\end{equation}
that one can construct  from any section $\phi\in \Gamma(\Sigma,P^X_\rho)$ and any classifying map 
for $\pi_P:P\ra\Sigma$, as in Equation (\ref{phi2T}) below. To each $\mathbf{h} \in H_2^\TT(X;\ZZ)$ we assign a vortex moduli space
\begin{equation} \label{moduli}
\mathsf{M}_{\mathbf{h}}^X(\Sigma):= \left\{ (A,\phi) : \bar \partial^A \phi =0, *F_A + \mu^\sharp \circ\phi =0, [\phi]_2^\TT=\mathbf{h} \right\}/\mathcal{G}(P).
\end{equation}
A theorem in~\cite{CGS,Mun} generalising what is sometimes called the Bogomol'ny\u\i\ argument~\cite{Bog} yields the equality
\begin{equation} \label{Bogtrick}
{\mathsf{E}}(A,\phi) =  \langle [\omega_X+\mu]^2_\TT,[\phi]_2^\TT \rangle + \frac{1}{2}\int_\Sigma \left(  \left| *F_A + \mu^\sharp \circ\phi  \right|^2 + 2 \left| \bar\partial^A \phi \right|^2 \right).
\end{equation}
Here, $[\omega_X+\mu]^2_\TT \in H_\TT^2(X;\RR)$ denotes the class of the equivariant 2-form $1\otimes \omega_X+\mu$ in the Cartan complex of the $\TT$-action on $X$ (see~\cite[\S 7.1]{BeGeVe}), and the pairing between equivariant homology and cohomology is effected via the Chern--Weil homomorphism ${\rm CW}_A$ supplied by the connexion $A$ (see \cite[\S 7.6]{BeGeVe}):
\begin{equation} \label{CWpairing}
\langle [\omega_X+\mu]^2_\TT,[\phi]_2^\TT \rangle = \int_\Sigma \phi^* {\rm CW}_A(1\otimes \omega_X+\mu).
\end{equation}
By Chern--Weil theory, this integral is independent of $A$ and constant when $\phi$ varies within a section homotopy class, whereas the integrand in the second term of (\ref{Bogtrick}) is manifestly nonnegative.
A consequence is that, for fields rendering (\ref{CWpairing}) nonnegative, solutions of the vortex equations (\ref{vort1}) and (\ref{vort2}) are precisely the {minima} of the functional ${\sf E}$ in (\ref{GSM}) in their section homotopy class, if they exist.

In general, a moduli space (\ref{moduli}) may contain orbifold points, but its smooth part comes equipped with a natural Riemannian structure which we refer to as the {\em  $L^2$ metric}. This geometric structure is induced from a
formal metric on the infinite-dimensional space of fields defined from the ingredients entering (\ref{GSM}). 
Namely: regarding the space $\mathcal{A}(P)$ of connexions in $P$ as an affine space over $\Omega^1(\Sigma;P^{\mathfrak{t}}_{\rm Ad})$, and interpreting the tangent space ${\rm T}_\phi \Gamma(\Sigma;P^X_\rho)$ 
as $\Gamma(\Sigma,\phi^*{\rm T}X/\TT)$, one sets
\begin{equation} \label{L2innprod}
(\dot A_1 , \dot\phi_1) \bullet (\dot A_2, \dot \phi_2 ):= \int_\Sigma\left( \langle \dot A_1 \stackrel{\wedge}{,}*\dot A_2 \rangle_\mathfrak{t}  + (\phi^*g_X)(\dot\phi_1,\dot\phi_2) \omega_\Sigma\right);
\end{equation}
here $g_X$ is the K\"ahler metric of $X$, $\omega_\Sigma$ the area form on $\Sigma$, and $(\dot A_j,\dot\phi_j)  \in {\rm T}_A\mathcal{A}(P) \oplus {\rm T}_\phi\Gamma (\Sigma,P^X_\rho)$ denote tangent vectors at
$(A,\phi)$.

One salient feature of the  $L^2$ metric $g_{L^2}$ on $\mathsf{M}_{\mathbf{h}}^X(\Sigma)$ is that it is K\"ahler with respect to an underlying complex structure induced from
\begin{equation} \label{cxstr}
(\dot A, \dot \phi) \mapsto \left(*\dot A, (\phi^*j_X) \dot \phi\right),
\end{equation}
which preserves (\ref{L2innprod}).
There are several ways of seeing this --- for example, one can recast the definition (\ref{moduli}) as an infinite-dimensional K\"ahler quotient of the subspace of $\mathcal{A}(P) \times \Gamma (\Sigma,P^X_\rho)$ cut out by Equation (\ref{vort1}) (a complex submanifold with respect to (\ref{cxstr}), and hence K\"ahler --- see Section 2.3 of \cite{Mun}) by the gauge group $\mathcal{G}(P)$; the left-hand side of Equation (\ref{vort2}) plays the role of a moment map for this action,
so there is an immediate analogy with Kempf--Ness theory~\cite{MumFogKir}, in parallel to the symplectic approach to gauge theory originating with Atiyah--Bott~\cite{AtiBot} and Donaldson~\cite{DonNS,DonNM}. (For more details we refer the reader to \cite{MunHK}, where the notion of stability appropriate to set up a Hitchin--Kobayashi correspondence in this setting was developed, and the analysis required to make sense of the underlying Marsden--Weinstein quotient was carried out under some technical assumptions.)

In this paper, we will
only be dealing with {\em Abelian} gauged sigma models, i.e.\ the case where $\TT$ is an Abelian group.
The simplest possible example is when $\TT={\rm U}(1)$ with its usual action on $X=\CC$; then (\ref{vort1})--(\ref{vort2}) describe the most familiar `gauged' vortices~\cite{Bog,JafTau} that 
occur in condensed matter physics, where $\phi$ is a section of a complex line bundle, and (\ref{GSM}) is nothing more than the potential (Ginzburg--Landau) energy of the Abelian Higgs
model~\cite{ManSut}, itself a Lorentzian GLSM in $1+2$ dimensions. The generalisation to higher-dimensional representations, i.e.\ to Abelian
GLSMs where an $n$-dimensional torus $\TT={\rm U}(1)^n$ acts on $X=\CC^r$, was considered in \cite{MorPle,SchGLSM,Weh}. Under suitable assumptions, the corresponding moduli spaces are known to be smooth and compact, and
the K\"ahler class $[\omega_{L^2}]$ of the corresponding  $L^2$ metrics was first described in~\cite{BapL2}.

Our main focus, however, will be on the simplest type of {\em nonlinear} Abelian gauged sigma model with compact target, namely, the case where $X=\PP^1$ is given its round metric, and $\TT={\rm U}(1)$ acts by rotations
around a fixed axis. This corresponds to a natural ${\rm U}(1)$-gauging of the classical Heisenberg model of ferromagnetism, whose solutions (sometimes referred to as {\em lumps}) are holomorphic or antiholomorphic maps $\Sigma \rightarrow \PP^1$. The gauged model was first considered on the Euclidean plane $\Sigma=\RR^2$ in \cite{SchP1}. For the most natural choice of moment map, the existence
of vortex solutions was established in~\cite{YanCVA}, and then in \cite{Sib2Yan} for compact $\Sigma$. These vortices can be parametrised by arbitrary pairs of effective divisors on $\Sigma$ with fixed degrees and non-intersecting supports. Interpreting $\phi$ as a meromorphic section of a complex line bundle  $L\rightarrow \Sigma$ (that is, a holomorphic section of the projectivisation $\PP(L\oplus \underline{\CC}) \rightarrow \Sigma$),
the two divisors prescribe the positions of $k_\tp$ zeros and $k_\tm$ poles of $\phi$, including their algebraic multiplicities; and so, under a natural assumption (\ref{strictBradlow}) that we spell out in Section~\ref{sec:P1model},
\begin{equation} \label{moduliP1}
\mathsf{M}^{\PP^1}_{(k_\tp,k_\tm)}(\Sigma) = (\Sym^{k_+}(\Sigma) \times \Sym^{k_\tm}(\Sigma)) \setminus D_{(k_\tp,k_\tm)}.
\end{equation}
We write $ \Sym^{k}(\Sigma):=\Sigma^k/\mathfrak{S}_k$ for the $k$th symmetric product of $\Sigma$ (which is naturally a complex manifold~\cite{ACGH} since $\Sigma$ carries a complex structure), whereas $D_{(k_\tp,k_\tm)}$ is the large diagonal hypersurface in the product. To a certain extent, the zeros and poles can be thought of as locations of particles (the vortex {\em cores}) and their antiparticles, as we shall see, and this justifies the term `antivortex' (also adopted in reference~\cite{Yan}) in our title. 
One can check that the bijection (\ref{moduliP1}) identifies the natural complex structure on the right-hand side with the one induced from (\ref{cxstr}); but giving any concrete description of the  $L^2$ metric $g_{L^2}$ (or its associated symplectic structure $\omega_{L^2}$) is a much harder problem.

This article is intended as a first detailed study of these  $L^2$ metrics. In particular, we shall investigate aspects that directly relate to the noncompactness of the moduli spaces (\ref{moduliP1}) --- their completeness or incompleteness; whether they determine finite volume and finite total scalar curvature for a compact surface $\Sigma$; and most crucially,  their asymptotic behaviour close to the boundary $D_{(k_\tp,k_\tm)}$. 
These issues are pertinent to understanding the low-energy dynamics of the corresponding (1+2)-dimensional field theories (in the spirit of Manton's adiabatic approximation~\cite{StuA,SpeA}), but also
decisive to elucidate  the supersymmetric extensions of these field theories after an A-twist \cite{BapTSM} at the quantum level, in the semiclassical approximation provided by supersymmetric quantum mechanics
on the moduli spaces (\ref{moduli})~\cite{BokRomP,BokRomDB,BisRomNG}.

Let us summarize the contents of the paper. We start by reviewing in Section~\ref{sec:P1model} the most basic aspects of the gauged $\PP^1$ model. 
In Section~\ref{sec:meroStrSam}, we show how a localization argument in previous work by Strachan and Samols on the linear (Abelian Higgs) model extends to
meromorphic Higgs fields. The following section addresses the  $L^2$ geometry of the model on the Euclidean plane. We examine the case of vortex-antivortex pairs and present
striking evidence for a self-similarity property of the fields in relation to the pair separation; this is tested numerically and employed to derive an explicit conjecture for the asymptotic geometry at small separation. In contrast, the asymptotics at large separation can be treated in close analogy with the linear situation. 
Section~\ref{sec:P1S2} is dedicated to the study of the gauged $\PP^1$ model on $S^2_R$, the round two-sphere of arbitrary radius $R$; the main results are a formula for the volume of the moduli space of vortex-antivortex pairs, and a proof that this moduli space is geodesically incomplete. 
In Section~\ref{sec:GLSMs}, we employ an auxiliary GLSM to derive conjectural formulae for the volume and the total scalar curvature for the moduli spaces (\ref{moduliP1}) in
much greater generality, assuming only that $\Sigma$ is a compact surface; the volumes obtained reduce to our rigorous volume formula for $\mathsf{M}^{\PP^1}_{(1,1)}(S^2_R)$. Finally, the more general volume formulae are applied in Section~\ref{sec:thermo} to analyze the classical statistical mechanics of a mixture of gases of vortices and antivortices on a compact surface.

\section{The gauged $\PP^1$ model} \label{sec:P1model}\news

We need to describe the gauged $\PP^1$ model in more detail in order to set up our most basic notation. 
We shall think of $X=\PP^1$ concretely as the sphere $S^2$ of unit radius where $\TT={\rm U}(1)$ acts by rotations about an axis.

For this choice of target, and assuming for now that $\Sigma$ is compact, the topological invariant (\ref{theinvariant}) in the classification of vortex solutions lies in the group
\begin{equation} \label{HT2P1}
H^\TT_2(\PP^1;\ZZ) := H_2({\rm E}\TT \times_\TT \PP^1;\ZZ)  \cong \ZZ^2.
\end{equation}
In order to grasp (\ref{HT2P1}) and appreciate what this invariant encapsulates, it is useful to refer to toric geometry~\cite{CLS}. We recall that the group ${\rm Div}_{\TT}(\PP^1)$ of $\TT$-invariant divisors in $\PP^1$  identifies with
$H_\TT^2(\PP^1;\ZZ)$, which sits in a splitting short exact sequence  (see Lemma~1 in~\cite{BokRomDB})
\begin{equation} \label{splitSES}
0 \rightarrow H^2({\rm B}\TT;\ZZ) \stackrel{\alpha}{\longrightarrow} H^2_{\TT}(\PP^1;\ZZ)\stackrel{\beta}{\longrightarrow} H^2(\PP^1;\ZZ) \rightarrow 0.
\end{equation}
A basis of ${\rm Div}_{\TT}(\PP^1)\cong \ZZ^2$ is provided by the two fixed points ${\bf x}_\tpm$ of the $\TT$-action (the North and South poles on $S^2$), and they provide convenient coordinates that may be interpreted as components of the isomorphism (\ref{HT2P1}):
\begin{equation} \label{npmreally}
k_\tpm := \langle {\bf x}_\tpm, \mathbf{h} \rangle \qquad \text{ for }\mathbf{h} \in H_2^\TT(\PP^1;\ZZ).
\end{equation}
In our context, we are interested in equivariant 2-homology classes of the type
\begin{equation} \label{phi2T}
\mathbf{h}=[\phi]_2^\TT:= H_2((\tilde{f} \times \phi)/\TT;\ZZ)([\Sigma])
\end{equation}
for a lift $\tilde{f}: P \rightarrow {\rm E}\TT$ of a classifying map $f:\Sigma\rightarrow {\rm B}\TT$, $[\Sigma] \in H_2(\Sigma;\ZZ)$ being the fundamental class.
In this situation, the integers $k_\tpm$ can also be interpreted as
intersection numbers of the image $\phi(\Sigma) \subset P^{\PP^1}$ with the surfaces $\Sigma_\tpm \subset P^{\PP^1}$ swept out by the fixed points ${\bf x}_\tpm \in  \PP^1$ of the fibre: (\ref{npmreally}) can be
reinterpreted as
$$
k_\tpm = \int_\Sigma \phi^* {\rm CW}_A(1\otimes {\rm PD}({\bf x}_\tpm)),  
$$
where ${\rm PD}$ stands for Poincar\'e dual and we are using the Chern--Weil homomorphism of any connexion $A$ in $\pi_P$. Note that such a connexion will always be integrable, and it endows 
$P^{\PP^1}$ with the structure of complex surface where $\Sigma_\tpm$ sit as complex curves. If in addition we assume that (\ref{vort1}) is satisfied, then also $\phi(\Sigma)$ is a complex curve, and
it follows that $k_\tpm \ge  0$ will hold.

Let us consider the dual $\alpha^*$ of the map $\alpha$ in (\ref{splitSES}). We note that the linear map taking $$[\Sigma] \mapsto \alpha^*([\phi]_2^\TT) \in H_2({\rm B}\TT;\ZZ)$$ admits an interpretation as 
the degree ${\rm deg}\, (P)$ of the principal $\TT$-bundle $P\rightarrow \Sigma$.
But a
principal torus bundle 
is determined up to isomorphism by its degree, so we shall write as $P(\mathbf{h})$  the principal $\TT$-bundle with degree $[\Sigma] \mapsto \alpha^*({\bf h})$ for a given $\mathbf{h} \in H_2^\TT(\PP^1;\ZZ)$, in a slight abuse of notation. Referring back to the
description of the map $\alpha$ in terms of the normal fan of $\PP^1$, we may write down the map recovering the degree from the equivariant 2-homology class explicitly as 
\begin{equation} \label{degreetop}
\alpha^*: (t_\tp,t_\tm) \mapsto t_\tp - t_\tm.
\end{equation}
On the other hand, one can dualise the short exact sequence (\ref{splitSES}) to obtain an inclusion $\beta^*$ of $H_2(\PP^1;\ZZ) \cong \ZZ$ into $H_2^\TT(X;\ZZ)$ as the diagonal $\{ (k,k) | k \in \ZZ \}$ in terms of the coordinates (\ref{npmreally}). As we shall see later (in Equation (\ref{Ebound}) below), for a fixed moment map $\mu$ the total energy of a vortex, topologically quantised as the invariant (\ref{CWpairing}), supplements the information given by the degree (\ref{degreetop}) in a
way that completely determines the topological charge $[\phi]_2^\TT$, and vice-versa.

In our calculations, we shall describe the target by means of the isometric embedding $X=S^2\hookrightarrow \RR^3$,
mapping the North pole ${\bf x}_\tp$ to the vector $\mathbf{n}:=(0,0,1).$ The area form on $S^2$ can be written as
$$
\omega_{S^2}|_\mathbf{u}(\mathbf{v},\mathbf{w})= \mathbf{u} \cdot ( \mathbf{v}\times  \mathbf{w}) \qquad \text{ for } \mathbf{u}\in S^2 \text{ and } \mathbf{v}, \mathbf{w} \in {\rm T}_\mathbf{u}S^2 \subset {\rm T}_{\bf u}\RR^3=\RR^3,
$$ 
whereas the complex structure is $j_{S^2}|_{\bf u}={\bf u}\times$. 
Under the obvious identification $\mathfrak{u}(1)^*  \cong \RR$, the possible moment maps
\begin{equation}\label{momentmap}
\mu(\mathbf{u})=-\mathbf{n}\cdot \mathbf{u} +\tau
\end{equation}
 for the circle action are given by the height function (relative to $-\bf n$)
up to translation by a constant $\tau \in \RR$.  The most natural choice is $\tau = 0$, but other choices of $\tau$ will shift the vacuum circle of latitude from the equator and give rise to different versions of
the model, for which the roles of vortices
and antivortices are no longer interchanged through the antipodal map of $S^2$.

Often, it will be convenient to assume that a local trivialization or section $\sigma: U\rightarrow P$ of $\pi_P$ 
has been given over an open set $U \subset \Sigma$. The restriction $\phi|_{U}$ in the trivialization can also be interpreted as a smooth map $\mathbf{u}: U \rightarrow S^2 \subset \RR^3$,
and the connexion 1-form $A\in \Omega^1(P;\mathfrak{u}(1))$ pulls back as $a:=\sigma^*A \in \Omega^1(U;\RR)$ with the identification $\mathfrak{u}(1)^*\cong \RR$ made above. Any other choice of local
trivialization over $U$ is obtained from $\sigma$ through pointwise multiplication by a map ${\rm e}^{\rm i \chi}: U \rightarrow {\rm U}(1)$, and the two choices are related through a gauge transformation
\begin{equation} \label{gaugtra}
a \mapsto a + {\rm d}\chi,\qquad \mathbf{u} \mapsto 
\left(
\begin{array}{ccc}
\cos \chi  & -\sin \chi & 0\\
\sin \chi & \cos \chi & 0\\
0 & 0 & 1
\end{array}
\right)
\mathbf{u}.
\end{equation}
In these conventions, we can also think of ${\rm d}^A \phi|_U : {\rm T}U 	\rightarrow \Gamma(U,\phi^*{\rm T}S^2/\TT) $ somewhat more concretely as a map ${\rm T}U \rightarrow {\rm T}\RR^3$ that we write as
$$
{\rm d}^a \mathbf{u} = {\rm d}\mathbf{u} - a\,  \mathbf{n}\times \mathbf{u},
$$
whereas the curvature of the connexion restricts simply as $F_a := F_A|_U= {\rm d}a$. Sometimes, it will  be convenient to compose ${\bf u}=(u_1,u_2,u_3)$ with the stereographic projection from the South
pole to obtain a complex function $u= \frac{u_1+{\rm i}u_2}{1+u_3}: U\rightarrow \CC \cup \{ \infty \}$ that also represents $\phi|_U$ in the trivialization. Accordingly, we will refer to
preimages of ${\bf x}_{\tpm}$ as zeros and poles, respectively.

It is instructive to reformulate the Bogomol'ny\u\i\ argument (\ref{Bogtrick}) using a choice of local trivialization.
It suffices to consider the case in which $Z:=\phi^{-1}(\Sigma_\tp\cup\Sigma_\tm)$ is finite, and  every zero and pole of $\phi$ has multiplicity
$1$ or $-1$. Choose any ${\bf y}_0\in\Sigma\less Z$ and let $U=\Sigma\less\{{\bf y}_0\}$. Certainly $P^{\P^1}\rightarrow \Sigma$ trivializes over $U$. Choose a trivialization. Then, since $U$ is dense in $\Sigma$,
\begin{eqnarray}
{\sf E}(\phi,A)&=& \frac{1}{2} \int_U \left( |F_a|^2 + |{\rm d}^a \mathbf{u}|^2 + |\mathbf{n}\cdot \mathbf{u}-\tau|^2  \right) \nonumber \\
&=& \frac{1}{2}\int_U\left( |F_a|^2 + |({\rm d}^a {\bf u})({\bf e}_1)|^2 + |({\rm d}^a  {\bf u})({\bf e}_2)|^2 + |{\bf n}\cdot {\bf u} - \tau|^2 \right) \label{energyE1} \\
&=& \frac{1}{2}\int_U \left(  |F_a -*({\bf n}\cdot {\bf u} -\tau) |^2  + |({\rm d}^a {\bf u})({\bf e}_1) + {\bf u} \times  ({\rm d}^a {\bf u})({\bf e}_2)|^2  \nonumber \right)  +\int_U\Xi,\nonumber
\end{eqnarray}
where we defined
\begin{eqnarray} \label{topintegrand}
\Xi &:=&\mathbf{u}\cdot ({\rm d}^a \mathbf{u} \times {\rm d}^a {\bf u}) + ({\bf n} \cdot {\bf u} - \tau)F_a. \nonumber \\
&=& \mathbf{u}^*\omega_{S^2}+ {\rm d}(({\bf n}\cdot {\bf u} -\tau)a).
\end{eqnarray}
In the second step of (\ref{energyE1}) we evaluate at two orthonormal vector fields on $U$, say  ${\bf e}_1$ and ${\bf e}_2=j_\Sigma {\bf e}_1$, where $j_\Sigma$ is the complex structure on $\Sigma$. The last integral $\int_U\Xi$ matches with the topological term (\ref{CWpairing})
expected from Equation (\ref{Bogtrick}). 

It remains to compute $\int_U\Xi$. First, note that $\Xi$ is gauge invariant, so extends to a smooth two-form on $\Sigma$, which we also denote $\Xi$. On $U\less Z$, define the one-form
\begin{equation}\label{1formxi}
\xi := ({\bf n}\cdot {\bf u}) (a - {\bf u}^*{\rm d}\varphi ),
\end{equation}
where $\varphi$ denotes the azimuthal angle on $S^2$.
This form is also gauge invariant, so extends to a smooth one-form on $\Sigma\less Z$. A short computation
reveals that, on $\Sigma\less Z$,
\beq
\Xi=\d\xi -\tau F_a.
\eeq
Hence,
\ben
\int_U\Xi=\int_\Sigma\Xi&=&-2\pi\tau(k_\tp-k_\tm)+\int_{\Sigma\less D}\d\xi\\
&=&-2\pi\tau(k_\tp-k_\tm)-\sum_{{\bf z}\in D}\lim_{\eps\ra 0}\int_{\cd B_\eps({\bf z})}\xi\\
&=&-2\pi\tau(k_\tp-k_\tm)+\sum_{{\bf z}\in D}\lim_{\eps\ra 0}\int_{\cd B_\eps({\bf z})}({\bf n}\cdot {\bf u}){\bf u}^*{\rm d}\varphi,
\een
where $B_\eps({\bf z})$ denotes the disk of radius $\eps$ centred on ${\bf z}$. In the case where ${\bf z}\in\phi^{-1}(\Sigma_\tp)$ is a zero with mutiplicity $\pm 1$, 
${\bf n}\cdot {\bf u}({\bf z})=1$ and $\int_{\cd B_\eps({\bf z})}{\bf u}^*{\rm d}\varphi=\pm 2\pi$, so each such point contributes $\pm 2\pi$ to the sum.
Similarly, if ${\bf z}\in\phi^{-1}(\Sigma_\tm)$ is a pole with multiplicity $\pm1$, ${\bf n}\cdot {\bf u}({\bf z})=-1$ and $\int_{\cd B_\eps({\bf z})}{\bf u}^*{\rm d}\varphi=\mp 2\pi$ (note the angle $\varphi$ winds {\em clockwise} around the South pole), so each such point also contributes $\pm 2\pi$. Summing over all points in $D$, we find that
\bea
\int_\Sigma\Xi&=&-2\pi\tau(k_\tp-k_\tm)+2\pi(k_\tp+k_\tm)
=2\pi(1-\tau)k_\tp+2\pi(1+\tau)k_\tm.
\eea

We conclude that there is a bound on the energy
 \begin{equation} \label{Ebound}
 {\sf E}(A,\phi) \ge 2 \pi (1-\tau) k_\tp + 2 \pi(1+\tau) k_\tm,
 \end{equation}
 which is attained in each section homotopy class in which the 
vortex equations (\ref{vort1})  and (\ref{vort2}) admit solutions, corresponding to minimizers of ${\sf E}$.
In the trivialization over $U$, the equations take the form
\begin{equation}\label{vort1P1}
{\rm d}^a {\bf u} + {\bf u} \times  {\rm d}^a {\bf u} \circ j_\Sigma=0
\end{equation}
and
\begin{equation} \label{vort2P1}
*F_a -{\bf n}\cdot {\bf u} +\tau =0.
\end{equation}
We see that they are gauge-invariant and thus globalize over $\Sigma$, as they should.

Note that the function ${\bf n}\cdot {\bf u}$ on $U\subset \Sigma$ takes values in $[-1,1]$. Integrating Equation (\ref{vort2P1}) over the dense open trivializing set, one obtains the two bounds
\begin{equation}\label{Bradlow}
-(1+\tau) {\rm Vol}(\Sigma) \le 2 \pi (k_\tp-k_\tm) \le (1-\tau) {\rm Vol}(\Sigma)
\end{equation}
as necessary conditions for existence of vortex solutions; inequalities such as these are sometimes referred to as {\em Bradlow's bounds} \cite{Nog, Bra,ManSut}. More invariantly, one can reinterpret (\ref{Bradlow}) by saying that the degree of $P$, seen as a homomorphism $H_1(\Sigma;\ZZ) \rightarrow H_1({\rm B}\TT;\ZZ) \subset H_1({\rm B}\TT;\RR) \cong \mathfrak{u}(1)$, linearly extends to a map taking the {\em K\"ahler co-class} $[\omega_\Sigma]^\vee \in H_2(\Sigma;\RR)$ (where the K\"ahler class $[\omega_\Sigma]$ evaluates as unity) to the image $\mu^\sharp(S^2) \subset \mathfrak{u}(1) \cong \RR$ of the moment map. This image is a Delzant polytope for compact toric $X$; in our simple situation $X=S^2$, it corresponds to the bounded interval $[-1-\tau, 1-\tau]$ of shifted heights on the target.

Whenever the strict inequalities hold in (\ref{Bradlow}) (i.e.\ the K\"ahler co-class is taken to the {\em interior} of the interval), we have the following result
(see also \cite{Mun,BapVA, BokRomTFB}, and \cite{Sib2Yan} for the case $\tau=0$).

\begin{theorem} \label{thmmodulispace}
Suppose  that $0 \ne \mathbf{h} \in {H_2^\TT(\PP^1;\ZZ)}$ lies in the cone
$${\rm span}_{\RR_{\ge 0}} \, \{ {\bf x}_\tp, {\bf x}_\tm \} \; \subset {\rm Div}_\TT(\PP^1)\otimes_\ZZ \RR =H^\TT_2(\PP^1;\RR) \cong \RR^2$$ 
(in other words, that the quantities $k_\tpm$ defined from $\bf h$ in (\ref{npmreally}) are nonnegative, not both zero), and that the inequalities
\begin{equation}\label{strictBradlow}
-(1+\tau) {\rm Vol}(\Sigma) < 2 \pi (k_\tp-k_\tm) < (1-\tau) {\rm Vol}(\Sigma)
\end{equation}
are satisfied for some $\tau \in \RR$. Then the moduli space
\begin{equation} \label{moduliP1def}
\mathsf{M}_{\mathbf{h}}^{\PP^1}(\Sigma):= \left\{ (A,\phi) : \bar \partial^A \phi =0, *F_A + \mu^\sharp \circ\phi =0, [\phi]_2^\TT=\mathbf{h} \right\}/\mathcal{G}(P(\mathbf{h}))
\end{equation}
of the vortex equations for the gauged $\PP^1$ model with moment map (\ref{momentmap}) admits the description
\begin{equation} \label{moduliP1descr}
\mathsf{M}^{\PP^1}_{\mathbf{h}}(\Sigma) \cong (\Sym^{k_+}(\Sigma) \times \Sym^{k_\tm}(\Sigma)) \setminus D_{(k_\tp,k_\tm)}.
\end{equation}
\end{theorem}

\begin{proof}[(Sketch of) Proof.]
We take the symplectic viewpoint on the quotient (\ref{moduliP1def}), as explained in \cite{Mun} and already mentioned in our Introduction,
interpreting Equation (\ref{vort2}) as giving the zero-set of the moment map for the Hamiltonian
$\mathcal{G}(P(\mathbf{h}))$-action on the infinite-dimensional  manifold $$\mathcal{A}(P(\mathbf{h})) \times \Gamma(\Sigma,P(\mathbf{h})^{\PP^1}),$$ 
which is equipped with a K\"ahler structure by the assignments (\ref{cxstr}) and (\ref{L2innprod}). 

The assumption (\ref{strictBradlow}) ensures that, for a pair $(A, \phi)$ with  $\bar\partial^A \phi=0$, the function ${\bf n} \cdot {\bf u}$  is neither of the constants $\pm1$ in an open $U\subset \Sigma$; and it is easy to check that this is equivalent to  $(A,\phi)$ being {\em simple}, in the sense of Definition~2.17 of \cite{MunHK}. But the same assumption also implies that $(A,\phi)$ is {\em stable}  in the sense of the very general Definition~2.16 of \cite{MunHK}, which in this Abelian setting reduces to the condition (29) in \cite{BapVA}; Baptista shows that this condition holds true in the subsequent paragraph, within the proof of Proposition~4.7 ibidem.

Now we can resort to Mundet's Hitchin--Kobayashi correspondence (Theorem~2.19 in \cite{MunHK}). Consider the  complexification $\mathcal{G}(P({\bf h}))^\CC$, which also acts on pairs $(A,\phi)$
with $[\phi]_2^\TT={\bf h}$ and preserves the Equation (\ref{vort1}). The correspondence establishes that, for a simple and stable solution of this equation, there is a unique
$\mathcal{G}(P({\bf h}))$-orbit inside its $\mathcal{G}(P({\bf h}))^\CC$-orbit where (\ref{vort2}) also holds.
This gives a bijection between (\ref{moduliP1def}) and the formal GIT quotient
\begin{equation}\label{GITquotient}
\left\{ (A,\phi) : \bar \partial^A \phi =0, \phi(\Sigma) \ne \Sigma_\tpm,  [\phi]_2^\TT=\mathbf{h} \right\}/\mathcal{G}(P(\mathbf{h}))^\CC.
\end{equation}
Finally, it is clear that (\ref{GITquotient}) identifies with the right-hand side of (\ref{moduliP1descr})
through the bijection 
\begin{equation*} \label{thebijection}
[(A,\phi)] \mapsto (\phi)=(\phi)_0-(\phi)_\infty,
\end{equation*} 
which assigns to a solution of (\ref{vort1}) up to complex gauge transformations the divisor (of zeros and poles) of the meromorphic section $\phi$.
\end{proof}

Condition (\ref{strictBradlow}) being granted, this theorem establishes that, up to gauge equivalence, BPS configurations in the gauged $\PP^1$ model are determined by the two effective
(but possibly zero, though not simultaneously) divisors  $(\phi)_0$ and $(\phi)_\infty$
on $\Sigma$ of 
degrees $k_\tp$ and $k_\tm$, respectively,  specifying isolated {\em cores} of vortices and antivortices on the surface that may coalesce {\em separately}. These divisors are only subject to the obvious constraint that
their supports 
$$
{\rm supp} \, (\phi)_0 = \phi^{-1} (\Sigma_\tp) \quad \text{and} \quad {\rm supp} \, (\phi)_\infty = \phi^{-1} (\Sigma_\tm)
$$
must not intersect.

In references~\cite{YanCVA, YanSOMC, Yan}, Yang argued that the result (\ref{moduliP1descr}) also holds for the noncompact base surface $\Sigma=\RR^2$ endowed with the Euclidean metric,
for which the assumption (\ref{strictBradlow}) does not apply. In this situation, the definitions and derivations we have presented above have to be adjusted, and sometimes even completely reworked. Though
$\Sigma$ is then contractible and hence $P\rightarrow \Sigma$ a trivial bundle, one incorporates nontrivial topology via the assumption that the section $\phi$ converges along a boundary circle $S^1_\infty$ at infinity with a certain (exponential) regularity. One may introduce a global trivialization ($U=\Sigma$), and the role played by ${\rm deg}\,P$ in the arguments above is delegated to the degree of the map 
$$
\lim_{|{\bf y}|\ra\infty}{\bf u}({\bf y}):S^1_\infty\ra \mu^{-1}(0)\cong S^1
$$
from spatial infinity to the circle of vacua of the Higgs potential. Strictly speaking, Yang (following Schroers in \cite{SchP1}) only considered the `easy plane' case $\tau=0$ (see e.g.~Theorem~11.3.1 in \cite{Yan}), but for $\tau\ne 0$ the existence of vortices in this model has also been established, following from the results of \cite{Han} for more general potentials. In Sections~\ref{sec:meroStrSam},~\ref{sec:P1Eucl}~and~\ref{sec:P1S2} we shall also restrict our attention to $\tau=0$ for the sake of simplicity.

\subsection{The  $L^2$ metric on moduli of BPS configurations} \label{sec:L2metric}

Recall that the  $L^2$ metric is induced from the formal metric on the space of fields $\mathcal{A}(P) \times \Gamma(\Sigma,P^{\PP^1})$ given by the inner product (\ref{L2innprod}) on tangent vectors
\begin{equation} \label{tangentvectors}
(\dot A,\dot\phi)  \in {\rm T}_A\mathcal{A}(P) \oplus {\rm T}_\phi\Gamma (\Sigma,P^{\PP^1})= \Omega^1(\Sigma,P^{\mathfrak{u}(1)}_{\rm Ad}) \oplus \Gamma(\Sigma,\phi^*{\rm T}\PP^1/\rm{U}(1))
\end{equation}
at some pair $(A,\phi)$.
In the following, we will fall back on various descriptions of tangent vectors on the moduli space itself. The first one (see e.g.~\cite{StuAHV}) consists in identifying the elements in the (finite-dimensional) vector space
${\rm T}_{[(A,\phi)]} {\sf M}_{(k_\tp,k\tm)}^{\PP^1}(\Sigma)$ with those $(\dot A,\dot \phi)$ in (\ref{tangentvectors}) that satisfy the {\em linearisation} of the vortex equations (\ref{vort1}) and (\ref{vort2})
at $(A,\phi)$, and are {\em orthogonal} to the orbit $\mathcal{G}(P) (A,\phi)$ with respect to the inner product (\ref{L2innprod}). 

For concreteness, we make use of representatives $(a,{\bf u})\in \Omega^1(U;\RR)\times C^\infty(U;S^2)$ of each pair
$(A, \phi)$ in a trivialization of $P\rightarrow \Sigma$ over some open subset $U\subset \Sigma$, as above, and denote by $(\dot a,\dot{\bf u})$ solutions to the linearisation of 
equations (\ref{vort1P1}) and (\ref{vort2P1}) about $(a,{\bf u})$. Since the local form of the action of ${\rm Lie}\, \mathcal{G}(P)$ tangent to (\ref{gaugtra}) is obtained by linearisation as
$$
(a,{\bf u}) \mapsto (a+{\rm d}\chi, {\bf u}+ \chi {\bf n} \times {\bf u}) \qquad \text{ for }\chi \in {\rm Lie}\, \mathcal{G}(P|_U) \cong C^\infty(U,\RR),
$$
the extra condition ensuring othogonality to gauge orbits with respect to (\ref{L2innprod}) can be reexpressed by trivializing over a dense open $U$ and imposing
$$
0=(\dot a,\dot{\bf u}) \bullet ({\rm d} \chi, \chi {\bf n} \times {\bf u}) = \langle {\rm d}^* \dot a + \dot {\bf u} \cdot ({\bf n} \times {\bf u}), \chi \rangle_{L^2(U)} 
$$
for all $\chi \in C^\infty(U;\RR)$, where we are employing the usual  $L^2$ inner product on differential forms with respect to the metric $g_{\Sigma}|_U$, writing ${\rm d}^*$ for the adjoint to $\rm d$. This is
also expressed by the differential equation over such $U$
\begin{equation} \label{Gauss}
{\rm div} \, \dot a = - {\rm d}^* \dot a = \dot {\bf u} \cdot ({\bf n} \times {\bf u}), 
\end{equation}
which corresponds to Gau\ss 's law in Maxwell's theory (if we interpret $\dot a$ as an electric field), or the choice of {\em Coulomb gauge}.

Given a representative $(\dot a,\dot {\bf u}) \in \Omega^1(U)\oplus C^\infty(U,\RR^3)$ of a tangent vector  on the moduli space, the squared  $L^2$ norm induced from (\ref{L2innprod}) is
\begin{equation} \label{kinetic}
\| \dot a \|_{L^2(U)}^2  +  {\|\dot{\bf u}\|_{L^2(U)}^2}.
\end{equation}
In this normalization, the quantity (\ref{kinetic}) can also be given the physical interpretation of (twice) the {\em kinetic energy} ${\sf E}_{\rm kin}$ for fields $(a(t),{\bf u}(t))$ in the dynamical gauged $\PP^1$ model depending on a time parameter $t$, provided that one further imposes a {\em temporal gauge} where the 1-form $a(t)$ on spacetime
has only spatial components, at any instant $t$.

\section{Meromorphic Strachan--Samols localization} \label{sec:meroStrSam}\news

In order to describe the  $L^2$ metrics concretely in terms of coordinates on the moduli space, we will show how a localization technique for vortices in line bundles, introduced by Strachan~\cite{Str} on the hyperbolic plane and later applied by Samols~\cite{Sam} to the Euclidean plane, extends to our situation where the Higgs field is a meromorphic section.

We begin with the case $\Sigma=\RR^2$ and, following the notation in Section~\ref{sec:P1model}, work with pairs $(a,u)$ in a global trivialization. It is convenient to recast the vortex equations (\ref{vort1P1}) and (\ref{vort2P1}) as a second-order PDE for the gauge-invariant function
\begin{equation} \label{functionh}
h:= \log |u|^2= \log \left( \frac{1-{\bf n} \cdot {\bf u}}{1+{\bf n}\cdot {\bf u}}\right) : \RR^2 \rightarrow \RR \cup \{ -\infty, +\infty \}
\end{equation}
taking finite values outside the {\em $(\pm)$-vortex} cores (a shorthand we will use for vortices and antivortices, respectively), and $\mp \infty$ at the cores. For this, we solve  (\ref{vort1P1})
for $a=a_z{\rm d}z + a_{\bar z}{\rm d}\bar z$ with respect to $u$, yielding
\begin{equation} \label{solvevort1}
a_{\bar z} = -{\rm i} \frac{\partial_{\bar z} u}{u}
\end{equation}
outside the cores. Then we use this to eliminate $a$ from (\ref{vort2P1}), obtaining
$$
*F_a=*{\rm d}a=-\frac{1}{2} \nabla^2  \log |u|^2
$$
and then, after solving (\ref{functionh}) for ${\bf n}\cdot {\bf u}$,
\begin{equation}\label{cheapTaubesEucl}
\nabla^2 h - 2 \tanh \frac{h}{2}=0
\end{equation}
again away from the vortex cores; we use $\nabla^2\equiv \nabla^2_z:= 4 \partial_z \partial_{\bar z}$. 
Henceforth, we assume the $(\pm)$-vortex cores to be located at the points ${{\bf z}_r^\tpm} \in \RR^2 \cong \CC$ with complex coordinates $z=z^\tpm_r$, where $r=1,\ldots, k_\tpm$ (these locations are repeated if coalescence occurs). 
Equation (\ref{cheapTaubesEucl}) extends to the whole plane as
\begin{equation} \label{TaubesEucl}
\nabla^2 h - 2 \tanh \frac{h}{2} = 4 \pi \left( \sum_{r=1}^{k_\tp} \delta_{{\bf z}_r^{\tp}} -  \sum_{r=1}^{k_\tm} \delta_{{\bf z}_r^{\tm}}\right),
\end{equation}
which we refer to as (the Euclidean) {\em Taubes' equation} in analogy with \cite{JafTau,ManSut}; this can be got directly by plugging  the
Poincar\'e--Lelong formula (see e.g.~\cite{SABK}, p.~42) 
$$
\frac{{\rm i}}{2\pi}F_a+2{\rm i}\, {\partial \bar \partial} h =  \delta_{(\phi)} 
$$
for the meromorphic section $\phi$ into the second vortex equation (\ref{vort2P1}). Though we have fixed the divisor 
\begin{equation} \label{thedivisor}
(\phi)=\sum_{r=1}^{k_\tp} {\bf z}_r^\tp- \sum_{r=1}^{k_\tm} {\bf z}_r^\tm
\end{equation}
so far, we observe that (\ref{TaubesEucl}) also makes sense as an equation on the product ${\sf M}_{(k_\tp,k_\tm)}^{\PP^1}(\RR^2) \times \RR^2$.

The gist of Theorem 11.3.1 in \cite{Yan} is that (\ref{TaubesEucl}) has a unique solution $h$ with $h(z) \rightarrow 0$ exponentially
fast as $|z|\rightarrow\infty$, for arbitrarily fixed configurations of $({\pm})$-vortex cores on $\Sigma=\RR^2$ with disjoint supports --- in agreement with (\ref{moduliP1}). Note that in this case the global coordinate $z$
provides global coordinates on the moduli space, for we have global identifications
${\rm Sym}^{k_\tpm}(\RR^2)\equiv \CC^{k_\tpm}$ assigning to $k_\tpm$ unordered complex numbers $z^\tpm_r\in \CC$  the $k_\tpm$ complex coefficients of the monic polynomial $\prod_{r=1}^k(z-z^\tpm_r)$. If $k_\tpm>1$, the $z^\tpm_r$ themselves only  define local coordinates in the open dense subset where no coalescence of ($\pm$)-vortices occur, but for many purposes they are less cumbersome to work with.

Consider a curve of solutions to (\ref{vort1P1}) and (\ref{vort2P1}) along which there are $k_\tpm$ distinct vortex cores which move along trajectories $t\mapsto z_r^{\tpm}(t)$. Following \cite{Sam,ManSut}, we 
define $\chi$ (up to addition of integral multiples of $2\pi$) such that $u={\rm e}^{\frac{1}{2} h + {\rm i}\chi}$, and then $\eta$ through
\begin{equation} \label{eta}
\dot u =: \eta u,
\end{equation}
so 
\begin{equation} \label{etaexpl}
\eta = \frac{1}{2} \dot h + {\rm i} \dot \chi.
\end{equation}
Note that (\ref{etaexpl}) determines $\eta$ uniquely, since $\dot \chi$ is unambiguous. In global terms, the quantity $\eta$ should be understood as a 1-{\em current} (see \cite{GriHar}, Chapter 3) on ${\sf M}^{\PP^1}_{(k_\tp,k_\tm)}(\Sigma)\times \Sigma$ which evaluates on vector fields over ${\sf M}^{\PP^1}_{(k_\tp,k_\tm)}(\Sigma)$ to yield 0-currents with singularities along the {\em discriminant locus}
$$
\mathcal{D}_{(k_\tp,k_\tm)} := \left\{ ((\phi),{\bf y}) \in  {\sf M}^{\PP^1}_{(k_\tp,k_\tm)}(\Sigma) \times \Sigma \, | \, {\bf y} \in {\rm supp}\, (\phi) \right\}.
$$
This yields a second description of tangent vectors to the moduli space, alternative to the pairs
$(\dot a, \dot {\bf u})$ of Section~\ref{sec:L2metric}, which is more convenient for our present purposes. 
To be more concrete, we will describe the evaluation of $\eta$  at each tangent vector via a differential equation that it satisfies, and then how to
determine it directly from a solution of Taubes' equation (\ref{TaubesEucl}).

Fixing the divisor (\ref{thedivisor}), the pairs $(\dot h, \dot \chi)$ are on the same footing as the previous
$(\dot a, \dot u)$, that is: $\dot h$ satisfies the linearised Taubes' equation at $(\phi)$,
$$
\nabla^2 \dot h - \sech^2 \frac{h}{2}\, \dot h =0
$$
away from the cores, together with (\ref{Gauss}), which in view of (\ref{solvevort1}) becomes 
$$
{\rm div}\, \dot a = \nabla^2 \dot \chi = \frac{4|u|^2}{(1+|u|^2)^2} \dot \chi = \sech^2 \frac{h}{2} \, \dot \chi. 
$$
Hence away from the cores (or more accurately, the discriminant locus), $\eta$ evaluated at $(\dot h,\dot \chi)$ satisfies
\begin{equation}\label{PDEetaaway}
\nabla^2 \eta - \sech^2 \frac{h}{2}\, \eta = 0.
\end{equation}
In complete analogy with Proposition 5.1 in Chapter III of \cite{JafTau}, we can show that there is a representation of the form
\begin{equation}\label{TaubesFact}
u(z)=(z-z_r^\tpm)^{\pm 1}{v(z)}
\end{equation}
in some neighbourhood of $z=z_r^\tpm$, with $v$ smooth and satisfying $v(z^\tpm_r)\ne 0$. 
This leads to the asymptotics
\begin{equation} \label{asympteta}
\eta(z) = \frac{\mp \dot z_r^\tpm}{z-z_r^\tpm} + O(1)\quad \text { as } z\rightarrow z^\tpm_r.
\end{equation}
In this equation we are using $\dot z_r^\tpm$ to denote  complex coordinates for tangent vectors induced by the $z_r^\tpm$ over their local dense coordinate chart. This is our third (and most direct) representation of the tangent vectors, simply in terms of moduli coordinates.

Recalling that
$$
\nabla^2 \log |z-z_r^\tpm|^2 = 4 \pi \delta_{{\bf  z}_r^\tpm}(z),
$$
leading to
$$
\nabla^2 \left( \frac{\mp 1}{z-z_r^\tpm}\right) = \mp 4 \pi \left( \partial_z \delta_{{\bf z}_r^\tpm}\right) (z),
$$
we deduce that the global version of (\ref{PDEetaaway}) is
\begin{equation}\label{Taubeseta}
\nabla^2 \eta - \sech^2 \frac{h}{2} \, \eta = 4 \pi \left( \sum_{r=1}^{k_\tp} \dot z_r^\tp  \frac{\partial }{ \partial {z_r^\tp} } \delta_{{\bf z}_r^\tp} -  \sum_{r=1}^{k_\tm} \dot z_r^\tm \frac{\partial} {\partial {z_r^\tm}} \delta_{{\bf z}_r^\tm} \right).
\end{equation}
Comparing this with Taubes' equation (\ref{TaubesEucl}), we conclude that a solution $\eta$ to (\ref{Taubeseta}) can be calculated from the unique (exponentially decaying) solution $h$ to (\ref{TaubesEucl}) as
\begin{equation} \label{soleta}
\eta = \sum_{r=1}^{k_\tp} \dot z_r^\tp \frac{\partial h}{\partial z_r^\tp} +  \sum_{r=1}^{k_\tm} \dot z_r^\tm \frac{\partial h}{\partial z_r^\tm}.
\end{equation}

Differentiating (\ref{solvevort1}) with respect to the parameter $t$ yields
$$
\dot a_{\bar z} = - {\rm i} \,\partial_{\bar z} \eta
$$
away from the cores, whence the kinetic energy in (\ref{kinetic}) can be written as
\begin{eqnarray*}
{\sf E}_{\rm kin}&=& \frac{1}{2} \int_{\RR^2}\left( 4 |\partial_{\bar z} \eta|^2 + \frac{4|u|^2|\eta|^2}{(1+|u|^2)^2}\right) \\
&=& \frac{1}{2}\int_{\RR^2} \left( 4 \partial_z \bar \eta \, \partial_{\bar z}\eta + \sech^2 \frac{h}{2}\, |\eta|^2\right).
\end{eqnarray*}

We split $\RR^2$ into the union of disks $$D_\eps:=\bigcup_{\sigma=\pm}\bigcup_{r=1}^{k_\sigma} B_\eps({\bf z}_r^\sigma),$$ which are disjoint for $\eps$ sufficiently small, and its complement.
Since the integrand in ${\sf E}_{\rm kin}$ is smooth everywhere,
\begin{eqnarray}
{\sf E}_{\rm kin}&=&\lim_{\eps \rightarrow 0} \frac{1}{2}\left( \int_{D_\eps} + \int_{\RR^2\setminus D_\eps} \right) \left( 4 \partial_z \bar \eta \, \partial_{\bar z}\eta + \sech^2 \frac{h}{2}\, |\eta|^2\right) \nonumber \\
&=& \lim_{\eps \rightarrow 0} \frac{1}{2}  \int_{\RR^2\setminus D_\eps}  \left( 4 \partial_z \bar \eta \, \partial_{\bar z}\eta + \sech^2 \frac{h}{2}\, \bar \eta \eta \right)  \nonumber \\
&=&  \lim_{\eps \rightarrow 0}  \left[ 2\int_{\RR^2\setminus D_\eps} \partial_z(\bar \eta\, \partial_{\bar z} \eta) -\frac{1}{2} \int_{\RR^2\setminus D_\eps} \bar\eta\left( \nabla^2 \eta - \sech^2 \frac{h}{2}\, \eta\right)\right]  \nonumber \\
&=& 2 \lim_{\eps \rightarrow 0}  \int_{\RR^2\setminus D_\eps} \partial_z(\bar \eta\, \partial_{\bar z} \eta)  \nonumber \\
&=& -{\rm i} \lim_{\eps \rightarrow 0} \sum_{\sigma=\pm} \sum_{r=1}^{k_\sigma}\oint_{\partial B_\eps({\bf z}_r^\sigma)} \bar\eta\, \bar \partial \eta, \label{neatformula}
\end{eqnarray}
where the contour integrals are taken anticlockwise, and $\bar \partial$ is the Cauchy--Riemann operator. Essentially, in the last step we employed Stokes's theorem to express the squared norm (in  $g_{L^2}$) of tangent vectors represented by $\eta$ as a residue of the $(0,1)$-current  $\bar\eta\, \bar \partial \eta$ at the support of the divisor $(\phi)$
that specifies the fibre of ${\rm T}{\sf M}^{\PP^1}_{(k_\tp,k_\tm)}(\Sigma)$ on which $\eta$ is evaluated. 

The neat formula (\ref{neatformula}) for the squared norm in the  $L^2$ metric can be recast rather more explicitly in terms of coordinates on the moduli space. For this, we expand $h$ in a neighbourhood of $z=z_s^\tpm$ for fixed moduli, in analogy with \cite{ManSut}, as
\begin{eqnarray}\label{marcar}
\pm h(z) &= &\log |z-z_s^\tpm|^2+a_s^\tpm + \frac{1}{2}\bar b^\tpm_s (z-z_s^\tpm) + \frac{1}{2} b^\tpm_s (\bar z-\bar z_s^\tpm)  \nonumber \\
&& + \bar c_s^\tpm (z-z_s^\tpm)^2 + d_s^\tpm |z-z_s^\tpm|^2 + c_s^\tpm (\bar z-\bar z_s^\tpm)^2 + O(|z-z_s^\tpm|^3) \label{expandh} \\
&=&\pm [h_{\rm head} + h_{\rm quad}] +  O(|z-z_s^\tpm|^3),\nonumber
\end{eqnarray}
where $\pm h_{\rm head}$ and $\pm h_{\rm quad}$ refer to the terms shown explicitly on the first and second lines respectively; $a_s^\tpm, b^\tpm_s, c_s^\tpm$ and $d_s^\tpm$ are coefficients which depend, in some
undetermined way, on the location of the vortex and antivortex cores. Now
\begin{eqnarray*}
\nabla^2 h_{\rm head} &=& \pm 4 \pi \delta_{z_s^\tpm}, \\
\nabla^2 h_{\rm quad} &=& \pm 4 d_s^\tpm
\end{eqnarray*}
and
$$
\lim_{z\rightarrow z_s^\tpm} 2 \tanh \frac{h}{2} = \mp 2,
$$
so we deduce that
$$
d_s^\tpm = - \frac{1}{2}.
$$
Observe that, as $z\rightarrow z_s^\tpm,$
\begin{eqnarray*}
\partial_{\bar z} \left( \frac{\partial h}{\partial z_r^\tpm}\right) &= & \pm \frac{1}{2} \frac{\partial b_s^\tpm}{\partial z_r^\tpm} \pm \frac{1}{2}\delta_{r,s}+O(|z-z_s^\tpm|),\\
\partial_{\bar z} \left( \frac{\partial h}{\partial z_r^\tmp}\right) &= &\pm \frac{1}{2} \frac{\partial b_s^\tpm}{\partial z_r^\tmp} +O(|z-z_s^\tpm|);\\
\end{eqnarray*}
hence we get from (\ref{soleta}), at fixed core positions,
\begin{eqnarray*}
\partial_{\bar z}\eta &=& \sum_{r=1}^{k_\tp} \dot z^\tp_r \partial_{\bar z} \left( \frac{\partial h}{\partial z_r^\tp}\right) +  \sum_{r=1}^{k_\tm} \dot z^\tm_r \partial_{\bar z} \left( \frac{\partial h}{\partial z_r^\tm}\right)\\
&=&	\pm \frac{1}{2} \sum_{r=1}^{k_\tpm} \dot z_r^{\tpm} \left(\frac{\partial b_s^\tpm}{\partial z_r^\tpm} + \delta_{r,s} \right) \pm \frac{1}{2} \sum_{r=1}^{k_\tmp} \dot z_r^\tmp \frac{\partial b_s^\tpm}{\partial z_r^\tmp} + O(|z-z_s^\tpm|).
\end{eqnarray*}
Using this together with (\ref{asympteta}), we obtain
$$
\lim_{\eps \rightarrow 0}  \oint_{\partial B_\eps(z_s^\pm)} \bar\eta \, \bar\partial \eta= \pi {\rm i} \dot{\bar z}^\tpm_s
\left[ \sum_{r=1}^{k_\tpm} \dot z_r^\tpm \left( \frac{\partial b_s^\tpm}{\partial z_r^\tpm} + \delta_{r,s}\right) + \sum_{r=1}^{k_\tmp} \dot z_r^\tmp \frac{\partial b_s^\tpm}{\partial z_r^\tmp} \right],
$$
and substitution in (\ref{neatformula}) yields the result
\begin{eqnarray}
{\sf E}_{\rm kin} &=& \pi \left(  \sum_{r=1}^{k_\tp}|\dot z_r^\tp|^2 +  \sum_{r=1}^{k_\tm}|\dot z_r^\tm|^2 
+ \sum_{r,s=1}^{k_\tp} \frac{\partial b_s^\tp}{\partial z_r^\tp} \dot z_r^\tp \dot{\bar z}_s^\tp
+ \sum_{r,s=1}^{k_\tm} \frac{\partial b_s^\tm}{\partial z_r^\tm} \dot z_r^\tm \dot{\bar z}_s^\tm \right.  \nonumber \\
&& \left. \qquad+ \sum_{r=1}^{k_\tp} \sum_{s=1}^{k_\tm} \frac{\partial b_s^\tm}{\partial z_r^\tp} \dot z_r^\tp \dot{\bar z}_s^\tm
+\sum_{r=1}^{k_\tm} \sum_{s=1}^{k_\tp} \frac{\partial b_s^\tp}{\partial z_r^\tm} \dot z_r^\tm \dot{\bar z}_s^\tp \right).  \label{EkinStrSam}
\end{eqnarray}

Since ${\sf E}_{\rm kin}$ is a real quantity,  the following equalities must hold for all $r,s$:
\begin{equation} \label{reality}
\frac{\partial b_s^\tpm}{\partial z_r^\tpm}= \frac{\partial \bar b_r^\tpm}{\partial \bar z_s^\tpm} 
\qquad \text{and}\qquad
\frac{\partial b_s^\tmp}{\partial z_r^\tpm}= \frac{\partial \bar b_r^\tpm}{\partial \bar z_s^\tmp}. 
\end{equation}
This means that the Riemannian metric on ${\sf M}^{\PP^1}_{(k_\tp,k_\tm)}(\RR^2)$ corresponding to (\ref{EkinStrSam}), which can be written in the dense open stratum  where no cores coalesce as
\begin{eqnarray}
g_{L^2} &=&2 \pi \left(  \sum_{r=1}^{k_\tp}|{\rm d} z_r^\tp|^2 +  \sum_{r=1}^{k_\tm}|{\rm d} z_r^\tm|^2 
+ \sum_{r,s=1}^{k_\tp} \frac{\partial b_s^\tp}{\partial z_r^\tp} {\rm d} z_r^\tp {\rm d}{\bar z}_s^\tp
+ \sum_{r,s=1}^{k_\tm} \frac{\partial b_s^\tm}{\partial z_r^\tm} {\rm d} z_r^\tm {\rm d}{\bar z}_s^\tm \right. \nonumber \\
&& \left. \qquad+ \sum_{r=1}^{k_\tp} \sum_{s=1}^{k_\tm} \frac{\partial b_s^\tm}{\partial z_r^\tp} {\rm d} z_r^\tp {\rm d}{\bar z}_s^\tm+
\sum_{r=1}^{k_\tm} \sum_{s=1}^{k_\tp} \frac{\partial b_s^\tp}{\partial z_r^\tm} {\rm d} z_r^\tm {\rm d}{\bar z}_s^\tp \right), \label{StrSamgL2}
\label{metrichard}
\end{eqnarray}
is Hermitian; its $(1,1)$-form \cite{GriHar} can be cast locally as
$$
\omega_{L^2} = {\rm i}\pi \left( \sum_{r=1}^{k_\tp} {\rm d}z_r^\tp \wedge {\rm d}\bar z_r^\tp +  \sum_{r=1}^{k_\tm} {\rm d}z_r^\tm \wedge {\rm d}\bar z_r^\tm + \partial b  \right),
$$
where
$$
b:= \sum_{s=1}^{k_\tp} b_s^\tp {\rm d}\bar z_s^\tp +  \sum_{s=1}^{k_\tm} b_s^\tm {\rm d}\bar z_s^\tm
$$
is a $(0,1)$-form that is also defined over that stratum. We now have
$$
{\rm d}\omega_{L^2} = {\rm i} \pi {\rm d}\partial b = - {\rm i}\pi\partial \bar \partial b\,;
$$
but
$$
\bar \partial b = \sum_{r=1}^{k_\tp} \sum_{s=1}^{k_\tp}\frac{\cd b_s^\tp}{\cd \bar z_r^\tp}{\rm d}\bar z_r^\tp \wedge {\rm d}\bar z_s^\tp  +
\sum_{r=1}^{k_\tp} \sum_{s=1}^{k_\tm} \left( \frac{\partial b_s^\tp}{\partial \bar z_r^\tm} - \frac{\partial b_r^\tm}{\partial \bar z_s^\tp}\right) 
{\rm d}\bar z_r^\tm \wedge {\rm d}\bar z_s^\tp  +
\sum_{r=1}^{k_\tm} \sum_{s=1}^{k_\tm}\frac{\cd b_s^\tm}{\cd \bar z_r^\tm}{\rm d}\bar z_r^\tm \wedge {\rm d}\bar z_s^\tm 
$$
and, by virtue of (\ref{reality}),
$$
\frac{\cd^2 b_s^\tpm}{\cd z_q^\tpm\cd \bar z_r^\tpm}=\frac{\cd^2 \bar b_q^\tpm}{\cd \bar z_s^\tpm\cd \bar z_r^\tpm},\qquad
\frac{\partial}{\partial z_q^\tpm} \left( \frac{\partial b_s^\tp}{\partial \bar z_r^\tm} - \frac{\partial b_r^\tm}{\partial \bar z_s^\tm} \right) = 
\frac{\partial^2 \bar b_q^\tpm}{\partial \bar z_r^\tm \partial \bar z_s^\tp} - \frac{\partial^2 \bar b_q^\tpm}{\partial \bar z_s^\tp \partial \bar z_r^\tm} =0.
$$
We deduce that $\partial \bar \partial b =0$ on the stratum of no coalescence; since this stratum is an open dense subset of the moduli space, and we know that $\omega_{L^2}$ is globally regular, we conclude that ${\rm d}\omega_{L^2}=0$ everywhere. Thus we have a consistency check that the metric $g_{L^2}$ is K\"ahler.

We remark that the localization argument presented above readily extends to the case where $\Sigma$ is compact, or is given a non-Euclidean metric. One works over a dense open subset $U\subseteq\Sigma$, homeomorphic to a disk, on which $g_\Sigma=\Omega(z) \d z \d\bar{z}$, and repeats the argument, introducing factors of $\Omega$ where required. The final result may be economically expressed as follows: on the stratum where none of the $(\pm)$-vortex positions $z_r^{\tpm}$ coincide (and all lie in $U$), define $(Z_1,Z_2,\ldots,Z_{k_\tp+k_\tm})=(z_1^\tp,\ldots,z_{k_\tp}^\tp,z_1^\tm,\ldots,z_{k_\tm}^\tm)$ and
$(B_1,B_2,\ldots,B_{k_\tp+k_\tm})=(b_1^\tp,\ldots,b_{k_\tp}^\tp,b_1^\tm,\ldots,b_{k_\tm}^\tm)$, defined as in (\ref{marcar}). Then the $L^2$ metric on
${\sf M}_{(k_\tp,k\tm)}^{\PP^1}(\Sigma)$ is
\beq
\label{metriceasy}
g_{L^2}=2\pi\sum_{r,s=1}^{k_\tp+k_\tm}\left(\Omega(Z_r)\delta_{rs}+\frac{\cd B_s}{\cd Z_r}\right)\d Z_r\d\bar{Z}_s.
\eeq
Note that this coincides with the rather bulkier expression (\ref{metrichard}) in the case $\Omega\equiv 1$.

\section{ $L^2$ geometry of the  Euclidean gauged $\PP^1$ model} \label{sec:P1Eucl}\news

This section is dedicated to the study of the  $L^2$ metric on the moduli space ${\sf M}_{(1,1)}^{\PP^1}(\R^2)$ in the situation where $\Sigma= \RR^2$ is given the Euclidean metric. We fix $\tau = 0$ throughout.

\subsection{The  $L^2$ metric on the space of vortex-antivortex pairs} \label{sec:pairsC}

We take $k_\tp = k_\tm =1$, which is the case of BPS vortex-antivortex pairs that the title of our paper refers to. In this situation we can simplify notation by
omitting the $r,s$ subscripts that appeared throughout the discussion in Section~\ref{sec:meroStrSam}; we shall also write the former $\pm$ superscripts distinguishing vortices from antivortices as {sub}scripts.

To determine the  $L^2$ metric on ${\sf M}^{\PP^1}_{(1,1)}(\RR^2)$ after the meromorphic Strachan--Samols localization discussed in the previous section, the coefficients $b_\tp (z_\tp, z_\tm)$ and $b_\tm (z_\tp, z_\tm)$ are required. We define a function $b: (0,\infty) \rightarrow \RR$ by
$$
b(\eps) := b_\tp (\eps, -\eps).
$$
As transpires from this equation, we will be using $\eps$ to denote half  the  distance separating the vortex and antivortex cores on the plane. It is clear from (\ref{TaubesEucl}) and (\ref{expandh}) that 
\begin{equation} \label{simp1}
b_\tm (z_\tp, z_\tm) = - b_\tp(z_\tp, z_\tm)
\end{equation}
and
\begin{equation} \label{simp2}
b_\tp (z_\tp,z_\tm) = \frac{z_\tp-z_\tm}{|z_\tp - z_\tm|} \, b\left( \frac{1}{2}|z_\tp - z_\tm| \right);
\end{equation}
hence
\begin{equation}\label{Lambda}
\Lambda(\eps) := \frac{\partial b_\tp}{\partial z_\tp} = \frac{1}{4} \left( b'(\eps)+ \frac{b(\eps)}{\eps}  \right) = \frac{1}{4\eps} \left. \frac{{\rm d}}{{\rm d}\eps} (\eps b(\eps)) \right|_{\eps = \frac{1}{2}|z_\tp-z_\tm|}.
\end{equation}

According to (\ref{StrSamgL2}), the  $L^2$ metric on ${\sf M}^{\PP^1}_{(1,1)}(\RR^2)$ can be written in the (global) coordinates $z_\tp, z_\tm$ as
\begin{eqnarray*}
g_{L^2} &=& 2\pi \left( |{\rm d}z_\tp|^2 +  |{\rm d}z_\tm|^2
+ \frac{\partial b_\tp}{\partial z_\tp}|{\rm d}z_\tp|^2 + \frac{\partial b_\tm}{\partial z_\tm}|{\rm d}z_\tm|^2  \right. \\
&&\qquad\left. 
+\frac{\partial b_\tm}{\partial z_\tp} {\rm d}z_\tp {\rm d}\bar z_\tm + \frac{\partial b_\tp}{\partial z_\tm} {\rm d}z_\tm {\rm d}\bar z_\tp \right).
\end{eqnarray*}
Observe that we have from (\ref{simp1}) and (\ref{simp2})
$$
\frac{\partial b_\tm}{\partial z_\tp} = - \frac{\partial b_\tp}{\partial z_\tp}, \quad 
\frac{\partial b_\tp}{\partial z_\tm} = - \frac{\partial b_\tp}{\partial z_\tp}\quad \text{and} \quad
\frac{\partial b_\tm}{\partial z_\tm} =  \frac{\partial b_\tp}{\partial z_\tp}, \quad 
$$
and hence
\beq\label{ranhas}
g_{L^2} = 2 \pi \left[  |{\rm d}z_\tp|^2 +  |{\rm d}z_\tm|^2 + \Lambda \left( \frac{1}{2}|z_\tp - z_\tm| \right) |{\rm d}z_\tp - {\rm d}z_\tm|^2  \right].
\eeq

The locus ${\sf M}^{\PP^1 (0)}_{(1,1)}(\RR^2)$ of {\em centred} pairs, for which $z_\tp = -z_\tm = \eps {\rm e}^{{\rm i} \psi}$, consists of the fixed points of the
isometry induced by the involution
$$
h(z) \mapsto -h(-z),
$$
and is therefore a geodesic submanifold (see e.g. Lemma 4.2 in \cite{RomP1}) --- it is a surface of revolution with induced metric
\begin{equation} \label{confFactor}
g_{L^2}^{(0)} = 2 \pi \left( 2+ b'(\eps) + \frac{b(\eps)}{\eps}\right)\left({\rm d}\eps^2 + \eps^2 {\rm d}\psi^2 \right)=:
F(\eps)(\d\eps^2+\eps^2\d\psi^2).
\end{equation}

In the rest of Section~\ref{sec:P1Eucl}, we will be studying the conformal factor $F(\eps)$, which completely determines the  $L^2$ geometry of $M_{(1,1)}^{\PP^1}(\RR^2)$ --- compare ({\ref{Lambda}}) and (\ref{confFactor}). The most pressing questions concern its asymptotic behaviour. Though both limiting cases $\eps \rightarrow 0$ and $\eps \rightarrow \infty$ are interesting, the former is more novel: it corresponds to the boundary of the moduli space arising from target nonlinearity that we alluded to in the Introduction.

At least in principle, all information about $F(\eps)$ can be extracted from Taubes's equation (\ref{TaubesEucl}). The following regularisation will be convenient. First of all, by rotational symmetry we can restrict to the situation where $z_\tpm=\pm \eps$ are real. Let us define $\hat h: \CC \rightarrow \RR$ via the equation
$$
h(z) =: \log \left( \frac{|z-\eps|^2}{|z+\eps|^2}\right) + \hat h(z/\eps).
$$
It follows from (\ref{TaubesFact}) that this $\hat h$ is regular everywhere on $\CC$, and it satisfies the PDE
\begin{equation}\label{normTaubes}
\nabla^2_w \hat h - 2 \eps^2 \frac{|w-1|^2 {\rm e}^{\hat h} - |w+1|^2} {|w-1|^2 {\rm e}^{\hat h} + |w+1|^2} =0,
\end{equation}
where we introduced the normalised complex coordinate $w:= z/\eps$.

It is clear that the symmetries 
$$\hat h(\bar w) = \hat h(w) \quad \text{ and }\quad\hat h(-w) = -\hat h(w)$$
hold, so $\hat h$ vanishes on the imaginary axis. 
This allows us to concentrate on solving the one-parameter family of nonlinear elliptic PDEs (\ref{normTaubes}) as a boundary value problem in the right half-plane (conformally a disk) 
\begin{equation} \label{RHP}
\mathcal{H}:= \{w\in \CC: {\rm Re}(w) >0 \}
\end{equation}
with Dirichlet boundary conditions. Moreover, since the solutions also have reflexion symmetry in the real axis, it suffices to solve the equation  in the upper right quadrant of the complex plane,
adopting suitable boundary conditions (when discretising on a grid, say) along the positive real axis. Recall that the results of Yang \cite{YanSOMC} guarantee existence and uniqueness of a smooth solution to this problem for any $\eps \in (0,\infty)$.
 
\subsection{Self-similarity and asymptotic geometry at small separation} \label{sec:Fsmalleps}

Here, we report on our study of the conformal factor $F(\eps)$ in (\ref{confFactor}) as $\eps \rightarrow 0$.

First-hand information on the family of boundary value problems we have described for the nonlinear elliptic equation (\ref{normTaubes}) can be obtained from numerical methods.
We use a Newton--Raphson scheme to solve the equation on the upper right quadrant (at fixed $\eps$) on a $100 \times 100$ grid with the standard 5-point Laplacian, adopting a variety of lattice spacings
in the range 0.01 (which works for moderate to large $\eps$) to 2 (required for very small $\eps$). Once we have a solution $\hat h_\eps(x+{\rm i}y)$, where the subscript is used to
emphasise the $\eps$-dependence, we extract the coefficient
\begin{equation}\label{bfromhhat}
b(\eps) = \frac{1}{\eps} \left[\left. \frac{\partial \hat{h}_\eps} {\partial x}\right|_{x+{\rm i}y=1} - 1\right].
\end{equation}
Figure~\ref{fig:hprofiles} shows plots of $x\mapsto \hat h_\eps (x+{\rm i}0)$, for a decreasing sequence of values of $\eps$. Careful inspection of these led us to conjecture that $\hat h_\eps$ is {\em asymptotically self-similar} for small $\eps$, in
the 
{following sense. For each $\eps>0$ define $f_\eps:\R\ra\R$  by
$$
f_\eps(x):= \frac{1}{\eps} \hat h_\eps((x+{\rm i}0)/\eps).
$$
We conjecture that, as $\eps\ra 0$, $f_\eps$ converges uniformly to some fixed profile
$f_*$. Numerical evidence for this conjecture is presented in Figure~\ref{fig:selfsim}, which plots 
$f_\eps$ for a decreasing sequence of values of $\eps$: it can be seen that the curves approach the graph of
a particular fixed function $f_*$.}

\begin{figure}[htb]
\begin{center}
\includegraphics[scale=0.5]{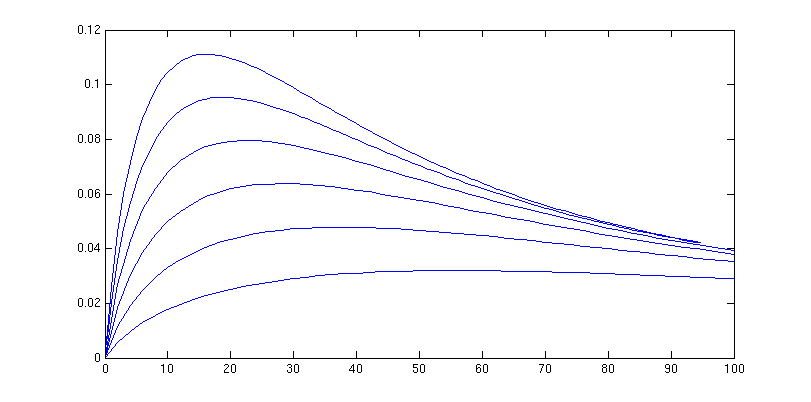}
\end{center}
\vspace{-10pt}
\caption{Numerical solutions of the regularized Taubes equation, $\hat{h}_\eps(x+{\rm i}0)$, for small values of $\eps$ ($\eps=0.07$ to $0.02$ in steps of $-0.01$ from top to bottom).}
\vspace{-195pt} \hspace{70pt} $\hat h_\eps (x)$ \\[125pt]
\vspace{0pt} \hspace{355pt} $x$
\vspace{40pt}
\label{fig:hprofiles}
\end{figure}

\begin{figure}[htb]
\begin{center}
\includegraphics[scale=0.5]{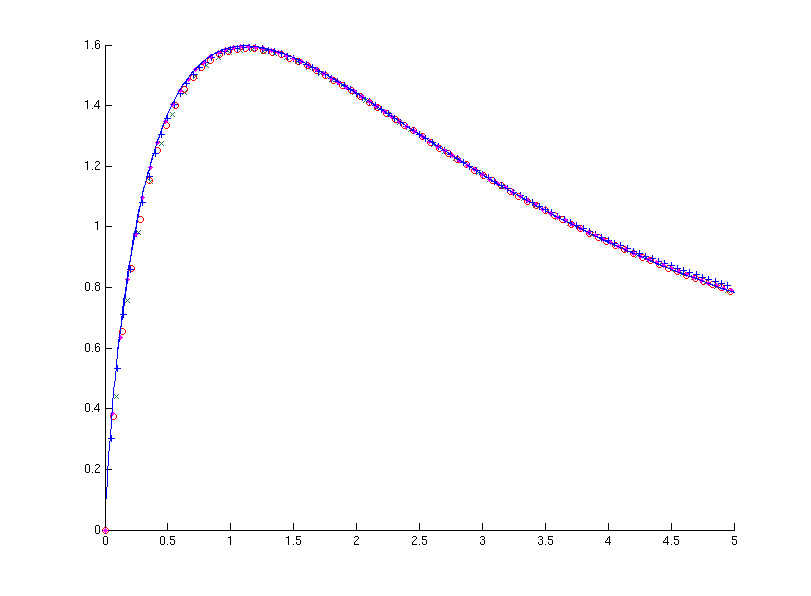}
\end{center}
\vspace{-235pt} \hspace{70pt} $f_\eps(x)$ \\[190pt]
\vspace{0pt} \hspace{360pt} $x$
\vspace{5pt}
\caption{Self similarity of $\hat{h}_\eps$: plots of $f_\eps(x):=\eps^{-1}\hat{h}_\eps(\eps^{-1}x)$ for $\eps=0.09$ ($\times$),
$\eps=0.07$ ($\circ$), $\eps=0.05$ (+) and $\eps=0.03$ ($\bullet$). The solid curve depicts $f_*$, the conjectured limit of $f_\eps$ as $\eps\ra 0$ (see Equation (\ref{fstarconj})).}
\label{fig:selfsim}
\end{figure}

We now explain how 
{an explicit formula for the conjectured limit $f_*$ can be obtained.
We first extend the defnition of $f_\eps$ to the whole complex plane,
$$
f_\eps(z):=\eps^{-1}\hat h_\eps(\eps^{-1}z),
$$
and note that the regularized Taubes equation (\ref{normTaubes}), when rewritten in terms of $f_\eps$
and $z=\eps w$, becomes
\beq\label{soraha}
\nabla^2f_\eps(z) -\frac{2}{\eps} \frac{|z-\eps|^2{\rm e}^{ \eps f_\eps( z)}  -|z+\eps|^2}{|z-\eps|^2 {\rm e}^{ \eps f_\eps( z)} +|z+\eps|^2}=0.
\eeq
We now assume that $f_\eps=f_*+\eps f_{(1)}+\eps^2 f_{(2)}+\cdots$, and that (\ref{soraha}) may be solved order by order in $\eps$. The leading order equation (order $\eps^0$) is
\bea
\nabla^2 f_*&=&2 \left. \frac{\rm d}{\rm d\eps} \left(   \frac{|z-\eps|^2{\rm e}^{ \eps f_*( z)}  -|z+\eps|^2}{|z-\eps|^2 {\rm e}^{ \eps f_*( z)} +|z+\eps|^2} \right)  \right|_{\eps=0} \nonumber \\
&=&
 f_*(z) -\frac{2 (z+\bar z)}{|z|^2}, \label{screenedPoisson}
\eea
which is an equation for $f_*$ alone.}

At this point, we want to solve the screened Poisson equation (\ref{screenedPoisson}) for $f_*$.
Since $f_*$ is a pointwise limit of functions odd under reflexion in the $y$ axis, it too has this symmetry, so we may solve (\ref{screenedPoisson}) on the right half-plane (\ref{RHP}) with Dirichlet boundary conditions over the imaginary axis. 
{This} can be done straightforwardly by separating variables in polar coordinates. 
We write
$z=:r{\rm e}^{{\rm i}\theta}$ and make the {ansatz } $f_*(r,\theta)=R(r)\Theta(\theta)$. Substitution in (\ref{screenedPoisson}) leads to
\begin{eqnarray*}
\nabla^2 f_* - f_* &=&\frac{\partial^2 f_*}{\partial r^2} + \frac{1}{r}\frac{\partial f_*}{\partial r} + \frac{1}{r^2}\frac{\partial^2 f_*}{\partial \theta^2}\\
&=&R''(r)\Theta(\theta) + \frac{1}{r} R'(r)\Theta(\theta)+\frac{1}{r^2} R(r) \Theta''(\theta) \\
&=& -\frac{4 \cos \theta}{r}.
\end{eqnarray*}
{Clearly $\Theta=\cos$, whence $R$ satisfies
$$
r^2 R'' + r R' - (r^2+1) R = -4r 
$$
an inhomogeneous second-order ODE  of modified Bessel type. The unique solution to this ODE with
$R(0)=0$ and $\lim_{r\ra\infty}R(r)=0$}
is
$$
R(r)=\frac{4}{r}\left(1-r K_1(r)\right),
$$
where $K_1$ denotes a modified Bessel function of the second kind~(see~\cite{WhiWat}, p.~373).
Thus we obtain the explicit solution for the profile $f_*$
\beq\label{fstarconj}
f_*(z)=\frac{2(z+\bar z)}{|z|^2}(1-|z| K_1(|z|)).
\eeq
Changing back to the coordinate $w=z/\eps$, we 
{obtain the more precise self-similarity conjecture that}
\begin{equation} \label{strongerselfsim}
\hat h_\eps(x) \ra \hat h_*(x):=\frac{4 }{x} (1-\eps x K_1(\eps x)) \qquad \text{uniformly as }\eps \rightarrow 0.
\end{equation}
The right-hand side of  (\ref{strongerselfsim}) is the function that we plotted on top of our numerical data in Figure~\ref{fig:selfsim}. It provides a very good fit for the values of $\eps$ in the interval $[0.03,0.09]$, with errors of the order of 0.1\% estimated at the single maximum. 

Equipped with (\ref{strongerselfsim}), 
{ and assuming that the convergence $\hat h_\eps\ra \hat h_*$ is actually uniform in $C^1$ in some neighbourhood of $1$,} we may infer {\em analytic} asymptotics (at small $\eps$) for the crucial function $\eps b(\eps)$ determining the metric on ${\sf M}^{\PP^1}_{(1,1)}(\RR^2)$ using (\ref{bfromhhat}), namely
$$
b(\eps) \approx b_*(\eps)= - \frac{5}{\eps}  +4 \eps K_0 (\eps) + 4 K_1(\eps) \qquad \text{as } \eps \rightarrow 0.
$$
An approximation for the conformal factor $F(\eps)$ can now be calculated as
\begin{equation} \label{asympFsmall}
F(\eps)\approx F_*(\eps) =  4 \pi  \left( 1 +2 K_0(\eps) - 2\eps K_1(\eps)   \right). 
\end{equation}
This function is positive and monotonically decreasing for  $\eps \in (0,2]$, and it blows up as $\eps\rightarrow 0$ with the asymptotics
\beq\label{lilise}
F_*(\eps)=8\pi \left( \log\frac{2}{\eps} -\frac{1}{2} -\gamma \right)+2\pi\left(  3 \log \frac{2}{\eps} + 2-3\gamma\right)\eps^2 + o(\eps^2),
\eeq
where $\gamma$ is the Euler--Mascheroni constant (\cite{WhiWat}, p.~235). This behaviour  leads to bounded lengths of radial geodesics as 
$\int_0^\delta \sqrt{F_*(\eps)}\, {\rm d}\eps < \infty$, approximating geodesics which represent
head-on vortex-antivortex pair collisions in the true  $L^2$ metric for small $\delta >0$. Thus we are led to conjecture that ${\sf M}^{\PP^1}_{(1,1)}(\RR^2)$ is incomplete. We plot this conformal factor in Figure~\ref{fig:conffactor}.

\begin{figure}[htb]
\begin{center}
\includegraphics[scale=0.5]{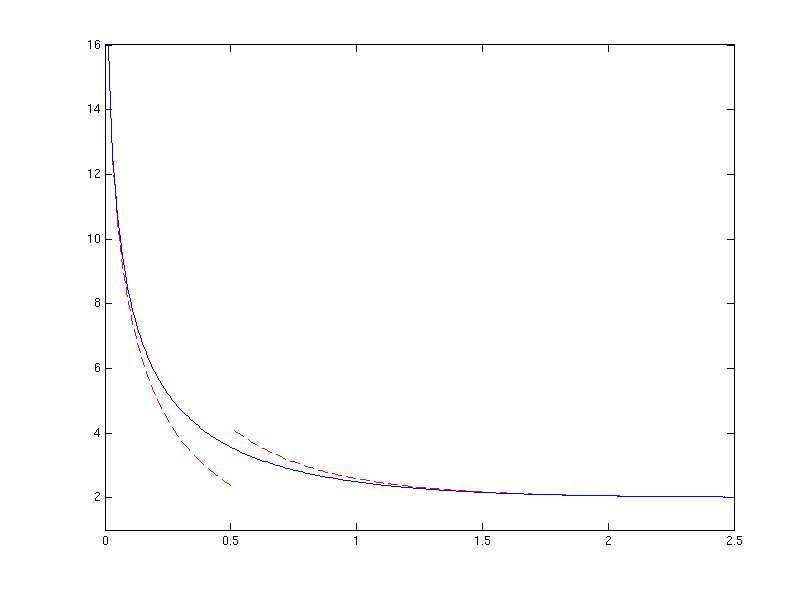}
\end{center}
\vspace{-230pt} \hspace{70pt} $\displaystyle \frac{F(\eps)}{2\pi}$\\[180pt]
\vspace{0pt} \hspace{355pt} $\eps$
\vspace{5pt}
\caption{The conformal factor $F(\eps)/2\pi$ for the metric on the space of centred $(1,1)$ vortex 
pairs. The dashed lines show the conjectured asymptotic forms for small and large $\eps$ (see equations (\ref{lilise}) and (\ref{asymptFlarge}) respectively).}
\label{fig:conffactor}
\end{figure}

{The space $(\M^{\P^1(0)}_{(1,1)}(\C),g^{(0)}_{L^2})$ of centred vortex-antivortex pairs can be isometrically embedded in $\R^3$ as a surface of revolution: see Figure~\ref{fig:cymbal}. We see that
its Gau\ss\ curvature 
$$
{\mathcal K}(\eps) = - \frac{1}{2\eps F(\eps)} \frac{{\rm d}}{{\rm d}\eps} \left( \eps \frac{{\rm d}}{{\rm d} \eps} \log F(\eps)\right)
$$
is positive in the core region, where the vortices are close to one another, but negative where they are more widely separated, becoming asymptotically flat as $\eps\ra \infty$. The small $\eps$ asymptotic formula $F_*$ for $F$ suggests that
${\mathcal K}(\eps)\ra\infty$ like $1/(16\pi\eps^2|\log(\eps)|^3)$ as $\eps\ra 0$, and that the total Gau\ss\ curvature of the surface is
$$
\int_{\C^\times}{\mathcal K}(\eps)\, F(\eps)\eps\, \d\eps\, \d\psi=\pi\lim_{\eps\ra 0}\frac{\eps F'(\eps)}{F(\eps)}=0.
$$}

\begin{figure}[htb]
\begin{center}
\includegraphics[scale=0.4]{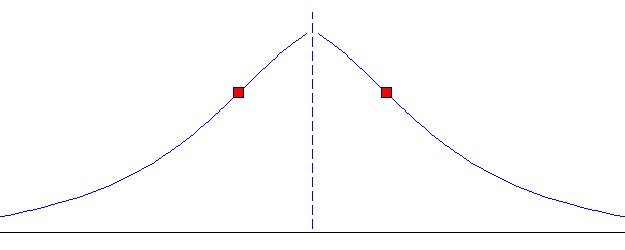}
\end{center}
\caption{Generating curve for the space of centred $(1,1)$ vortices isometrically embedded as a surface of revolution in $\R^3$: the full surface is obtained by rotation about the symmetry axis indicated by the dashed line. The red squares indicate inflexion points of the curve, where the Gauss curvature of the surface changes sign.}
\label{fig:cymbal}
\end{figure}

\subsection{Asymptotics at large separation}

A conjectural asymptotic formula for the conformal factor (\ref{confFactor}) at large $\eps$ can also be obtained, by adapting the point-vortex model described in Section~3 of reference~\cite{ManSpe} to the
gauged $\PP^1$ model.

In this framework, a ($\pm$)-vortex is modelled by a point particle carrying a scalar monopole charge $\pm q$ and magnetic dipole moment $\pm q\textit{\textbf{k}}$ orthogonal to the physical plane $\Sigma=\RR^2$.
Two such point particles, when at rest, exert no net force on one another, because the repulsive force due to the opposite scalar charges is exactly balanced by the attractive force between the opposite magnetic
dipoles~\cite{SpeSIF}. When in relative motion, such point particles do exert velocity-dependent forces on one another, however. The Lagrangian governing the motion of such point particles, of rest mass $2\pi$
({\em not} $\pi$ as in the Abelian Higgs model, see (\ref{Ebound}) with $\tau =0$), moving along trajectories $t\mapsto z_\tpm (t)$, is
\begin{equation} \label{Lagrpoints}
L=\pi (|\dot z_\tp|^2 + |\dot z_\tm|^2) + \frac{q^2}{4\pi} K_0(|z_\tp-z_\tm|)|\dot z_\tp-\dot z_\tm|^2
\end{equation}
up to quadratic order in the velocities. The constant $q$ may be deduced from the large $r$ behaviour of a single vortex, obtained, for example, by solving the
Bogomol'ny\u\i\ equations in a radially symmetric ansatz. 
We find, numerically, that
\begin{equation} \label{qP1}
q\approx -7.1388.
\end{equation}
Using this, we deduce from (\ref{Lagrpoints}) an asymptotic formula for the conformal factor (\ref{confFactor}) at large $\eps$:
\begin{equation} \label{asymptFlarge}
F(\eps) \approx F_{\infty} (\eps) = 2\pi \left( 2+\frac{q^2}{\pi^2} K_0 (2\eps)\right).
\end{equation}
In Figure~\ref{fig:conffactor} we plot these asymptotics against the result for $F(\eps)$ obtained by solving (\ref{normTaubes}) numerically, as well as the asymptotics $F_*(\eps)$ at small $\eps$ in Section~\ref{sec:Fsmalleps}. This provides a satisfactory check of our asymptotics at both ends,
and also indicates where our approximations break down.

Both the argument and the result (\ref{asymptFlarge}) in this section replicate almost verbatim the asymptotics of a  pair of vortices in the Abelian Higgs model (see Equation (3.51) in~\cite{ManSpe}). The only discrepancies
are the factor of two (already mentioned) in the rest mass, the pre-sign of the second terms in brackets in (\ref{asymptFlarge}), and the numerical value (\ref{qP1}) of $q$. We recall that in the Abelian Higgs model one obtains
$q_{\scriptscriptstyle{\rm AH}} \approx -10.6$ instead~\cite{SpeSIF}; in addition, there is an ingenious argument due to Tong~\cite{Ton}, invoking T-duality in configurations of D-branes, which proposes $q_{\scriptscriptstyle{\rm AH}}=-2^{{7}/{4}} \pi \approx -10.57$. One may ask whether such techniques can be adjusted to deal with the gauged $\PP^1$ model, producing an `analytic' version of $q$ that reproduces our estimate (\ref{qP1}) with comparable accuracy.

\section{ $L^2$ geometry of the gauged  $\PP^1$ model on a sphere} \news \label{sec:P1S2}

This section is dedicated to the study of the  $L^2$ geometry of the moduli space ${\sf M}_{(1,1)}^{\PP^1}(S^2_R)$, where $S_R^2$ is the round two-sphere of radius $R$, in the symmetric case $\tau=0$.
We write the metric on $S^2_R$ as
\begin{equation} \label{roundmetric}
g_{S^2_R}= \frac{4 R^2}{(1+|z|^2)^2}{\rm d}z{\rm d}\bar z
\end{equation}
where $z\in\C$ is a stereographic coordinate obtained by projection from the South pole.

We study the vortex equations (\ref{vort1}) and (\ref{vort2}) in the guise of the gauge-invariant Taubes equation
\begin{equation}\label{TaubespairsS2}
\nabla^2_z h - \frac{8 R^2}{(1+|z|^2)^2} \tanh \frac{h}{2} = 4 \pi \left( \delta(z-z_\tp) - \delta(z-z_\tm) \right)
\end{equation}
for $h:\CC \rightarrow \RR \cup\{ \pm \infty\} $ defined  as in (\ref{functionh}), with the vortex placed at $z=z_\tp$ and the antivortex at $z=z_\tm$. As in Section~\ref{sec:pairsC}, it will be convenient to place the $(\pm)$-vortices at $z_\tpm=\pm \eps$ with $\eps >0$ and  regularise (\ref{TaubespairsS2})
by introducing $\tilde h: \CC \rightarrow \RR$ defined by
\begin{equation} \label{tildeh}
h(z) := \tilde{h}(z) + h_{\rm sing} (z) \quad \text{ with } \quad  h_{\rm sing}(z):=  \log \left|  \frac{z - \eps}{z+\eps}\right|^2.
\end{equation}
This function extends smoothly over the $(\pm)$-vortex positions, and the South pole $z=\infty$.
Composing $\tilde{h}$ with the dilation map $\mathcal{R}_\eps:S^2\ra S^2$, $z\mapsto\eps z$, we obtain a smooth function
$$
\hat h:S^2\ra\R,\quad  \hat h:=\tilde h\circ \mathcal{R}_\eps,
$$
which, on our coordinate patch, satisfies the PDE
\begin{equation} \label{PDES2ish}
\nabla^2_w \hat h - \frac{8R^2 \eps^2}{(1+\eps^2 |w|^2)^2} F(w,\hat h) =0
\end{equation}
where
\begin{equation} \label{functionF}
\quad F(w,v):= \frac{|w-1|^2 {\rm e}^v - |w+1|^2}{|w-1|^2 {\rm e}^v + |w+1|^2},
\end{equation}
and we have, as a notational convenience, introduced a rescaled coordinate $w=z/\eps$ (so $\hat h(w)=\tilde h(z)$), to remind us that, for $\hat h$, the vortex positions are $w=\pm 1$.  

In terms of the coordinate $\tilde w := 1/\bar w$,  obtained from stereographic projection from the North pole of $S^2$, $\hat h$ satisfies
a PDE almost identical to (\ref{PDES2ish}), but replacing $\eps$ by $1/\eps$ and $w$ by ${\tilde w}$. So in order to
solve the vortex equations on $S^2$, it suffices to find  a solution $\hat h_{\rm  upper}$ to (\ref{PDES2ish}) on the closed unit disk $|w|\le 1$,
as well as a solution $\hat h_{\rm lower}$ of (\ref{PDES2ish}) with $\eps \mapsto 1/\eps$ on the disk $|w|\le 1$, while imposing the matching condition
$$
\hat h_{\rm upper}( w) = \hat h_{\rm lower}( w)
$$
for all $|w|=1$. We have found numerical solutions for various values of $R$ and $\eps$, using a discretisation on a regular $50 \times 50$ grid for polar coordinates $w=: r {\rm e}^{{\rm i} \theta}$ on the unit disk.
The results for $R=1$ are shown in Figure~\ref{fig:hslices_sphere}, where plots of $\hat h$ over the real $w$-axis against the angle of declination $\vartheta = \arcsin  r$  are superposed for several values of $\eps$.

\begin{figure}[htb]
\begin{center}
\includegraphics[scale=0.5]{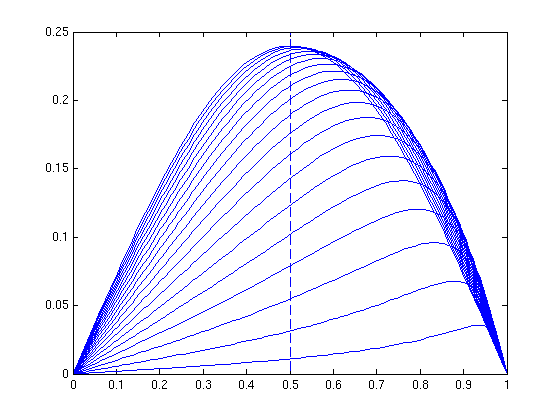}
\end{center}
\caption{Numerical solutions of the regularized Taubes equation on the sphere of radius $R=1$:
$\hat{h}$ along the real axis plotted against angle of declination $\vartheta$, for $\eps=1$ (top curve) to $\eps=0.05$ (bottom curve) in steps of $-0.05$. The metric on the moduli space of $(1,1)$ vortices
can be deduced from the gradient of these curves at $\vartheta=\pi/2$. Note that this gradient
vanishes as $\eps\ra 0$.}
\vspace{-255pt} \hspace{120pt} $\hat h$ \\[135pt]
\vspace{0pt} \hspace{310pt} $\vartheta/\pi$
\vspace{100pt}
\label{fig:hslices_sphere}
\end{figure}

The localization formula (\ref{metriceasy}), and a repeat of the argument in Section \ref{sec:pairsC}, imply that the  $L^2$ metric on the submanifold of centred $(\pm)$-pairs on $S^2_R$ is
\begin{equation}\label{g0pairsS2}
g_{L^2}^{(0)} = 2\pi \left(  \frac{8R^2}{(1+\eps^2)^2} + b'(\eps) + \frac{b(\eps)}{\eps} \right)({\rm d}\eps^2+\eps^2 {\rm d}\psi^2)
\end{equation}
where
\begin{equation}\label{thequantity}
\eps b(\eps)= \left. \frac{\partial \hat h}{\partial w_1}\right|_{w=1}-1,
\end{equation}
$w=:w_1 + {\rm i}w_2$ and $z_\tp=-z_\tm=\eps {\rm e}^{{\rm i}\psi}$. It is clear on symmetry grounds that (\ref{thequantity}) takes the value $-1$ for $\eps=1$, since this corresponds to the case where the $(\pm)$-vortices are antipodal, so that $\hat h$ must have a critical point at $w_1=1$.
The numerics strongly suggest that also
\begin{equation} \label{suggestion}
\lim_{\eps \rightarrow 0} (\eps b(\eps)) = -1,
\end{equation}
since $\hat h(w_1+{\rm i}0)$ seems to be flattening completely in this limit --- this is illustrated in Figure~\ref{fig:hslices_sphere} for $R=1$.  

In Section \ref{sec:vol11}, we will use strictly analytical methods to prove that (\ref{suggestion}) does in fact hold; see Theorem~\ref{thmvol}.
The proof exploits  properties of $g_{L^2}$ such as its large isometry group and K\"ahler property. In the next section we analyze the class of K\"ahler metrics on ${\sf M}_{(1,1)}^{\PP^1}(S^2_R)$ having the symmetries of $g_{L^2}$, obtaining a structure result (Proposition~\ref{invKaehler}) which is of independent interest (since, for example, it can be applied to the two-vortex moduli space of the usual {\em linear} Abelian Higgs model on
domain $S^2_R$).

It is of great interest to decide whether  ${\sf M}_{(k_\tp,k_\tm)}^{\PP^1}(\Sigma)$ is geodesically complete for $\Sigma$ compact
(see \cite{BokRomDB} for some motivation). The special case studied here ($k_\tp=k_\tm=1$, $\Sigma=S^2_R$, $\tau=0$) is the simplest in which the
question is nontrivial, since this is the simplest moduli space that is noncompact. In this situation, the question amounts to understanding the small $\eps$ behaviour of the {\em derivative} of $\eps b(\eps)$. In Section \ref{sec:inco} we present a proof that ${\sf M}_{(1,1)}^{\PP^1}(S^2_R)$
is {\em incomplete}, Theorem~\ref{incompleteness}. This is another hint that gauged sigma models are qualitatively similar to ungauged models, whose soliton moduli spaces are generically
incomplete \cite{SadSpe}.

\subsection{K\"ahler metrics on ${\sf M}_{(1,1)}^{\PP^1}(S^2_R)$}\label{sec:structure}

We know that the metric $g_{L^2}$ on
$${\sf M}_{(1,1)}^{\PP^1}(S^2_R)\cong S^2\times S^2 \setminus D_{S^2}$$
(where $D_{S^2}$ denotes the diagonal copy of $S^2$ in the product) is K\"ahler with respect to the obvious (product) complex structure $J$, and also {\em invariant}, in the sense that ${\rm SO}(3)$ rotations, acting diagonally on the product, and the holomorphic involution $I$ 
swapping the two factors are isometries. These properties alone almost determine it. 

To see this, we note that the ${\rm SO}(3)$-action preserves the (spherical or chordal) distance between the points $z=z_\tpm$ on $S^2$, and that on every orbit there is precisely one point of the form
$$
(z_\tp,z_\tm)=(\eps, -\eps)=: q(\eps) \qquad \text{ with }0<\eps \le 1.
$$
Further, the isotropy group ${\rm SO}(3)_{q(\eps)}$ is trivial for $\eps \in (0,1)$ and ${\rm SO}(2)$ for $\eps=1$. Hence $S^2\times S^2\setminus D_{S^2}$ is diffeomorphic to $(0,1)\times {\rm SO}(3)$ with
a copy of $S^2$ glued to the right end. It follows that, away from the exceptional orbit at $\eps=1$, we can write any ${\rm SO}(3)$-invariant metric $g$ on $S^2\times S^2\setminus D_{S^2}$ as
$$
g=\sum_{i,j=0}^3 A_{ij}(\eps) \sigma_i \sigma_j,
$$
for appropriate smooth functions $A_{ij}$, where $\sigma_0={\rm d}\eps$ and $\{\sigma_1,\sigma_2,\sigma_3\}$ is some left invariant basis of one-forms on ${\rm SO}(3)$. A convenient choice is the basis dual to the basis
\begin{equation} \label{vectorsEj}
E_1=\left( \begin{array}{ccc}
0 & 0& 0\\
0& 0& -1\\
0& 1& 0
\end{array}\right),\;\;
E_2=\left( \begin{array}{ccc}
0 & 0& 1\\
0& 0& 0\\
-1 & 0& 0
\end{array}\right),\;\;
E_3=\left( \begin{array}{ccc}
0 & -1& 0\\
1& 0& 0\\
0 & 0& 0
\end{array}\right)
\end{equation}
of $\mathfrak{so}(3)$. 

With these conventions, we have the following:

\begin{proposition} \label{invKaehler}
Any invariant K\"ahler metric on $S^2\times S^2 \setminus D_{S^2}$ has the form
\begin{equation} \label{gA}
g=A(\eps) \left( \frac{1-\eps^2}{1+\eps^2}\sigma_1^2 + \frac{1+\eps^2}{1-\eps^2} \sigma_2^2 \right)  - 
\frac{A'(\eps)}{\eps} \left({\rm d}\eps^2 + \eps^2 \sigma_3^2 \right)
\end{equation}
for some smooth, strictly decreasing function $A: (0,1]\rightarrow \RR$ with $A(1)=0$.
Such a metric has total volume
\begin{equation}\label{totvolgA}
{\rm Vol}(S^2\times S^2 \setminus D_{S^2},g)= 4 \pi^2 \left( \lim_{\eps\rightarrow 0} A(\eps)\right)^2.
\end{equation}
\end{proposition}

\begin{proof}
For the first statement, we follow closely the strategy of proof of Proposition~5 in reference~\cite{AlqSpe}.
The stereographic coordinate $z$ on $S^2$ induces complex coordinates $z_\pm$ on each copy of $S^2$ in the product, and $I$ acts as
\begin{equation}\label{involution}
I: (z_\tp,z_\tm) \mapsto (z_\tm, z_\tp).
\end{equation}
A short calculation shows that, at $q(\eps)$, the infinitesimal generators representing the vectors (\ref{vectorsEj}) and $E_0=\frac{{\rm d}}{{\rm d} \eps}$ can be
matched with a basis of ${\rm T}_{q(\eps)}(S^2\times S^2\setminus D_{S^2})$ as follows:
\begin{eqnarray*}
E_0 &=& \frac{\partial}{\partial z_\tp} -  \frac{\partial}{\partial z_\tm}, \\
E_1 &=& -\frac{{\rm i}}{2}(1-\eps^2) \left(  \frac{\partial}{\partial z_\tp} +  \frac{\partial}{\partial z_\tm}\right), \\
E_2 &=& \frac{1}{2}(1+\eps^2) \left(  \frac{\partial}{\partial z_\tp} +  \frac{\partial}{\partial z_\tm}\right), \\
E_3 &=& {\rm i}\eps \left(  \frac{\partial}{\partial z_\tp} -  \frac{\partial}{\partial z_\tm}\right). \\
\end{eqnarray*}
Hence we find that, at $q(\eps)$, 
 \begin{equation}  \label{JonEj}
 J E_0=\frac{1}{\eps} E_3, \quad JE_1=\frac{1-\eps^2}{1+\eps^2}E_2, \quad JE_2 = -\frac{1+\eps^2}{1-\eps^2} E_1,\quad JE_3=-\eps E_0.
 \end{equation}
 Now $g$ is assumed Hermitian, whence we conclude from (\ref{JonEj}) that
 $$
 A_{03}=A_{12}=0,\qquad A_{33}= \eps^2 A_{00}, \qquad A_{11}=\left( \frac{1-\eps^2}{1+\eps^2}\right)^2 A_{22}.
 $$
 Further, the involution (\ref{involution}) acts by pullback as follows:
 $$
 I^*{\rm d}\eps={\rm d}\eps,\qquad I^*\sigma_1=-\sigma_1,\qquad I^*\sigma_2=-\sigma_2,\qquad I^*\sigma_3=\sigma_3.
 $$
 This implies
 $$
 A_{01}=A_{02}=A_{13}=A_{23}=0.
 $$

So far, we have shown that every $J$-Hermitian ${\rm SO}(3)\times \{{\rm id},I\}$-invariant metric on \mbox{$S^2\times S^2\setminus D_{S^2}$} must have the form
\begin{equation} \label{shapeg}
g=A_0(\eps)({\rm d}\eps^2+\eps^2)+ A_2(\eps) \left(\left(\frac{1-\eps^2}{1+\eps^2}\right)\sigma_1^2 + \sigma_2^2\right),
\end{equation}
where $A_0=A_{00}$ and $A_2=A_{22}$. The corresponding $J$-$(1,1)$-form, $\omega(\cdot, \cdot)=g(J \cdot, \cdot)$, is
$$
\omega=\eps A_0(\eps) {\rm d}\eps \wedge \sigma_3+\frac{1-\eps^2}{1+\eps^2}A_2(\eps) \sigma_1\wedge \sigma_2.
$$
Now we have ${\rm d}\sigma_1=-\sigma_2\wedge \sigma_3$ and cyclic permutations, so $g$ is K\"ahler (i.e.\ $\omega$ is closed) if and only if
\begin{equation}\label{A0A2}
\eps A_0(\eps)=-\frac{{\rm d}}{{\rm d}\eps}\left( \frac{1-\eps^2}{1+\eps^2}A_2(\eps)\right).
\end{equation}
Defining, for convenience,
$$
A(\eps):=\frac{1-\eps^2}{1+\eps^2} A_2(\eps),
$$
and noting that regularity of $g$ as $\eps\rightarrow 1$ implies that $A_2(1)$ must be finite, and hence $A(1)=0$, we obtain the first assertion of the proposition.

The volume form associated with the metric (\ref{gA}) is
$$
{\rm vol}_{g}=-A'(\eps) A(\eps)\, {\rm d}\eps \wedge \sigma_{123}
$$
with $\sigma_{123}:=\sigma_1 \wedge \sigma_2\wedge \sigma_3$.
This leads to a total volume
\begin{equation}\label{Volwithc}
{\rm Vol}(S^2\times S^2 \setminus D_{S^2},g)=\frac{1}{2}c\left( \left(\lim_{\eps\rightarrow 0}A(\eps)\right)^2-A(1)^2\right) = \frac{c}{2} \left(\lim_{\eps\rightarrow 0}A(\eps)\right)^2
\end{equation}
with $c:= \int_{{\rm SO}(3)}\sigma_{123}$. To calculate $c$, we specialize the formula (\ref{Volwithc}) to the particular case where each factor $S^2=S^2_1$ is given its usual round metric of unit radius,
for which
$$
A(\eps)=A_{(S^2_1)^2}(\eps) := -2+\frac{4}{1+\eps^2};
$$
then (\ref{Volwithc}) should also yield $({\rm Vol}(S^2_1))^2=(4\pi)^2$. Thus $c=8\pi^2$, and we have proved the formula (\ref{totvolgA}).
\end{proof}

Let us  now consider the case of the  $L^2$ metric $g_{L^2}$ on ${\sf M}_{(1,1)}^{\PP^1}(S^2_R)$. Referring to the notation in equations (\ref{g0pairsS2}) and (\ref{shapeg}), one has that
$$
A_0(\eps)=g_{L^2}\left(\frac{\partial}{\partial \eps}, \frac{\partial}{\partial \eps}\right)=2\pi\left( \frac{8R^2}{(1+\eps^2)^2} + \frac{1}{\eps} \frac{{\rm d}}{{\rm d}\eps}(\eps b(\eps))\right),
$$
hence the ODE (\ref{A0A2}) takes the form
$$
\eps A_0(\eps) = 2\pi \frac{{\rm d}}{{\rm d}\eps} \left( \eps b(\eps) - \frac{4R^2}{1+\eps^2} \right)
$$
and integrates to
$$
A(\eps)=-2\pi \left( \eps b(\eps) - \frac{4R^2}{1+\eps^2} +C\right),
$$
where $C$ is some constant. Recall that, by regularity, $A(1)=0$, and that $b(1)=-1$, from the symmetry of $\hat h$ for antipodal pairs. Hence we find that
$$
C=1+2R^2,
$$
i.e.\ $g_{L^2}$ takes the form (\ref{gA}) with
$$
A(\eps)=2\pi \left( 2R^2 \frac{1-\eps^2}{1+\eps^2} - \eps b(\eps)  - 1\right).
$$
Proposition~\ref{invKaehler} now yields
\begin{equation}\label{voloutofb}
{{\rm Vol}(S^2\times S^2 \setminus D_{S^2},g_{L^2})= (2\pi)^2\left( 4\pi R^2- 2\pi \left(\lim_{\eps\rightarrow 0} (\eps b(\eps))+1 \right) \right)^2. }
\end{equation}

\subsection{The volume of ${\sf M}_{(1,1)}^{\PP^1}(S^2_R)$}\label{sec:vol11}

In this section, we prove

\begin{theorem} \label{thmvol}
For $\tau=0$, 
\begin{equation} \label{volumeformula}
{
{\rm Vol}\left({\sf M}_{(1,1)}^{\PP^1}(S^2_R)\right) = (2\pi\times4\pi R^2)^2.
}
\end{equation} 
\end{theorem}

This result is very satisfactory: it  coincides with the natural volume of the configuration space of a pair of point particles of mass
$2\pi$ moving on $S^2_R$. In light of Equation (\ref{voloutofb}), proving it reduces to establishing
the limit (\ref{suggestion}), which may be achieved via a direct analysis of the family of
nonlinear elliptic PDEs (\ref{PDES2ish}).

Throughout this section and the next, we assume that $h$, $\hat h$ satisfy Taubes' equation (\ref{TaubespairsS2}), (\ref{PDES2ish}) for vortex-antivortex pairs on the two-sphere. It follows from the main theorem of \cite{Sib2Yan} that $\hat h$ is, for each $\eps>0$, unique, and smooth as a mapping $S^2\ra\R$. It is helpful to recast (\ref{PDES2ish}) as a global PDE on $S^2$,
\beq\label{PDES2}
-\Delta_{S^2}\hat h+2R^2\eps^2 f_\eps(w)^2F(w,\hat{h})=0,
\eeq
where $\Delta_{S^2}$ denotes the usual Laplacian on the unit two-sphere (with the analysts' sign convention), $f_\eps(w)=(1+|w|^2)/(1+\eps^2|w|^2)$,
$F$ is defined in (\ref{functionF}), and we note that both $f_\eps$ and $F$ extend smoothly over $S^2$. 
We denote by $\Delta_{\RR^2}$ the usual Laplacian on (subsets of) the Euclidean plane. For a given Riemannian manifold $\Sigma$, we denote by $H^k(\Sigma)$ the completion with respect to the norm
$$
\|u\|_{H^k(\Sigma)}^2=\sum_{l=0}^k\|\nabla^lu\|_{L^2(\Sigma)}^2
$$
of the space of smooth maps $\Sigma\ra\R$. $H^k(\Sigma)$ is a Banach space, see \cite{Aub} for details.
Similarly, we denote by $C^0(\Sigma)$ the Banach space of bounded continuous real functions on $\Sigma$ with the norm $\|u\|_{C^0(\Sigma)}=\sup\{|u(p)|:p\in\Sigma\}$.  We will make use of these spaces for
$\Sigma=S^2$, the unit sphere, and $\Sigma=\DD$, an open Euclidean disk. 

Our argument makes heavy use of the following standard elliptic estimates for the Laplacians on $S^2$ and $\DD$.

\begin{proposition}\label{SEE} Let $\DD$, $\DD'$ 
be open disks in $\R^2$ such that $\ol{\DD'}\subset\DD$. Then there exists $C>0$, depending only on
$\DD,\DD'$ such that for all smooth $v:\DD\ra\R$, and  all smooth $u:S^2\ra\R$ with $\ip{1,u}_{L^2(S^2)}=0$,
\renewcommand{\labelenumi}{(\roman{enumi})}
\begin{itemize}
\ii[{\rm (i)}] 
$\ds{
\|v\|_{H^2(\DD')}\leq C(\|\Delta_{\R^2}v\|_{L^2(\DD)}+\|v\|_{L^2(\DD')}),
}$
\ii[{\rm (ii)}] 
$\ds{
\|u\|_{H^2(S^2)}\leq C\|\Delta_{S^2}u\|_{L^2(S^2)},
}$
\ii[{\rm (iii)}] 
$\ds{\|u\|_{H^1(S^2)}^2\leq C\ip{u,-\Delta_{S^2}u}_{L^2(S^2)}.}$
\end{itemize}
\end{proposition}

\begin{proof} Parts (i) and (ii) follow from more general results presented in \cite[p.\ 423]{DonKro}. 
To prove part (iii), we note that all $u$ with $\ip{1,u}_{L^2}=0$ are $L^2$ orthogonal to the kernel of $-\Delta_{S^2}$, and so
$$
\ip{u,-\Delta_{S^2} u}_{L^2}\geq \lambda_1\|u\|_{L^2}^2
$$
by Rayleigh's theorem \cite[p.\ 16]{Chavel}, where $\lambda_1>0$ is the lowest nonzero eigenvalue of $-\Delta_{S^2}$. Hence
$$
\|u\|_{H^1}^2=\ip{u,-\Delta_{S^2} u}_{L^2}+\|u\|_{L^2}^2\leq \left(1+\frac{1}{\lambda_1}\right)\ip{u,-\Delta_{S^2} u}_{L^2}.
$$
\end{proof}

We will also use two standard Sobolev imbeddings \cite[pp.~44, 51]{Aub}:

\begin{proposition}\label{SOB} Let $\DD\subset\R^2$ be an open disk. There exist constants $C(S^2)>0$ and $C(\DD)>0$ such that, for
all $u\in H^2(S^2)$ and all $v\in H^2(\DD)$,
$$
\|u\|_{C^0(S^2)}\leq C(S^2)\|u\|_{H^2(S^2)}\quad\mbox{and}\quad
\|v\|_{C^0(\ol{\DD})}\leq C(\DD)\|v\|_{H^2(\DD)}.
$$
\end{proposition}

We start with a very crude estimate of the $H^2(S^2)$-norm of the regularized, but undilated, Taubes function $\tilde{h}$; see (\ref{tildeh}).
Here, and henceforth, the symbol $C$ is used for a variable positive constant, independent of the parameter $\eps$
(but possibly dependent on $R$); the actual value of $C$ may vary from line to line.

\begin{lemma} \label{lem:H2normestimate}
 $\| \tilde h\|_{H^2(S^2)} < C$. 
\end{lemma}
\begin{proof}
Since $\tilde h$ is odd under reflexion across the imaginary $z$-axis, $\ip{1,\tilde{h}}_{L^2(S^2)}=0$, so applying Proposition~\ref{SEE}(ii) we
get
\begin{equation}
\| \tilde h \|_{H^2(S^2)} \le C\, \|\Delta_{S^2} \tilde h \|_{L^2(S^2)}  .
\end{equation}
Note that, away from the vortex positions, the magnetic field can be written
$$
F_a={\rm d}a=-\frac{{\rm i}}{4} \nabla^2_zh\, {\rm d}z\wedge {\rm d}\bar z = - \frac{1}{2}(\Delta_{S^2}h) {\rm vol}_{S^2},
$$
whence
$$
*F_a=- \frac{1}{2}\Delta_{S^2} h.
$$
Now 
$$
\Delta_{S^2} h_{\rm sing}(z) = \frac{(1+|z|^2)^2}{4} \Delta_{\RR^2} h_{\rm sing}(z)=0 
$$
away from the vortex positions, since $h_{\rm sing}$ is harmonic there. So the equation
\begin{equation} \label{actuallyglobal}
*F_a = - \frac{1}{2}\Delta_{S^2} \tilde h
\end{equation}
holds away from the vortex positions; but since both sides in (\ref{actuallyglobal}) are  smooth everywhere, this equation holds globally.

Since $(a,{\bf u})$
is a $(1,1)$ vortex, its total energy is ${\sf E}(a,{\bf u})=8\pi$. Hence
$$
\|\!*\!F_a\|^2_{L^2(S^2)} \le 2 \,{\sf E}(a,{\bf u}) =8\pi;
$$
using (\ref{actuallyglobal}), the result follows immediately.
\end{proof}

\begin{corollary} \label{cor:C0est}
$\| \hat h\|_{C^0(S^2)} = \| \tilde h\|_{C^0(S^2)} < C.$
\end{corollary}
\begin{proof}
This is a consequence of Lemma~\ref{lem:H2normestimate}, by Proposition~\ref{SOB}.
\end{proof}

The next step involves two pointwise sign estimates. {Let $\mathcal{H}$ denote the right half plane $\mathcal{H}=\{w\in\C\: :\: {\rm Re}(w)>0\}$.
In a slight abuse of notation, we will also use
$\mathcal{H}$ to denote the open hemisphere centred on $(0,1,0)$, that is, the open set in $S^2$ covered by the coordinate patch $\mathcal{H}$, and $\ol{\mathcal{H}}$ to denote its closure.}

\begin{proposition} \label{signestimates}
For all $\eps > 0$ and all $w\in{\mathcal{H}}$, one has
\begin{enumerate}
\item[{\rm (i)}] $\quad F(w,\hat h(w)) \le 0$;
\item[{\rm (ii)}] $\quad \hat h (w) \ge 0$.
\end{enumerate}
\end{proposition}

\begin{proof}
To prove (i), we shall argue directly in terms of the fields $(a,{\bf u})$, assumed to satisfy the local form (\ref{vort1P1}) and (\ref{vort2P1}) of the vortex equations. By (\ref{PDES2}) and (\ref{actuallyglobal}), it suffices to prove that $*F_a$ is a nonnegative function on the closed hemisphere $\overline{\mathcal H}$. We know that $\bf u$ restricted to $\partial \mathcal{H}$ is a map from a great circle in $S^2$ to the equator of $S^2$ of unit winding; that it takes the value $\bf n$ at exactly one point (the vortex core); and that it never assumes the value $-\bf n$ (since the antivortex core lies in $S^2\setminus \overline{\mathcal{H}}$). 
Observe that $\bf u$ wraps $\mathcal H$ exactly once around the Northern hemisphere in the target; that is,
$$
\int_{\mathcal{H}} \mathbf{u}^*\omega_{S^2}= \frac{1}{2} \int_{S^2}\omega_{S^2}= 2\pi.
$$

Assume now, towards a contradiction, that there is a maximal region $\mathcal{S} \subset \mathcal{H}$ of positive measure on which $*F_a$ is negative. By (\ref{vort2P1}) with $\tau = 0$ we have ${\bf n} \cdot {\bf u} <0$ on $\mathcal S$; so $\bf u$ is mapping $\mathcal{S}$ to the interior of the Southern hemisphere {\em punctured at the South pole} (since $-{\bf n}$ is never reached) and, by continuity, the boundary $\partial \mathcal{S}$ to the equator. Since the punctured (closed) Southern hemisphere retracts to the equator, we conclude that
\begin{equation}\label{rubber}
\int_{\mathcal{S}}{\bf u}^*\omega_{S^2}=0.
\end{equation}

By Equation (\ref{topintegrand}),
\begin{eqnarray}
\int_{\mathcal{S}} {\bf u}^*\omega_{S^2} &=& \int_{\mathcal{S}} \left( {\bf u}\cdot ({\rm d}^a{\bf u}\times {\rm d}^a {\bf u} )  + ({\bf n}\cdot {\bf u}) F_a- {\rm d}(({\bf n}\cdot {\bf u}) a)  \right)\nonumber \\
&=&  \int_{\mathcal{S}} \left( {\bf u}\cdot ({\rm d}^a{\bf u}\times {\rm d}^a {\bf u} )  +  * F_a \wedge F_a  \right) - \int_{\partial \mathcal{S}} ({\bf n}\cdot{\bf u})a \nonumber \\
& = & \int_{\mathcal{S}}  {\bf u}\cdot ({\rm d}^a{\bf u}\times {\rm d}^a {\bf u} ) + \| \!*\!F_a\|^2_{L^2(\mathcal{S})} \label{allpositive}
\end{eqnarray}
using (\ref{vort2P1}) in the second step, and the fact that ${\bf u}$ maps $\partial \mathcal S$ to the equator in the third. But (\ref{vort1P1}) also implies that
$$
 \int_{\mathcal{S}}  {\bf u}\cdot ({\rm d}^a{\bf u}\times {\rm d}^a {\bf u} )= \int_{\mathcal{S}} {\rm d}^a{\bf u} \cdot ({\bf u}\times {\rm d}^a{\bf u})=\|{\rm d}^a{\bf u} \|^2_{L^2(\mathcal{S})} \ge 0.
$$
Hence, by (\ref{allpositive}), 
$$
\int_{\mathcal{S}} {\bf u}^*\omega_{S^2} \ge \|\!*\!F_a \|^2_{L^2(\mathcal{S})}=\int_{\mathcal{S}}(*F_a)^2 >0,
$$
which contradicts (\ref{rubber}).

Turning now to (ii), assume,
towards a contradiction, that there exists $w\in \mathcal{H}$ such that $\hat{h}(w)<0$. Then, since
$\hat{h}$ attains a global minimum on the compact set $\ol{\mathcal{H}}$, and vanishes on the boundary
circle $\cd\ol{\mathcal{H}}$ on symmetry grounds, 
$\hat{h}$ attains a negative global minimum at some interior point $w_*\in\mathcal{H}$. Consider the Hessian
matrix $(\hat{h}_{ij})=(\cd^2 \hat{h}/\cd w_i\cd w_j)|_{w_*}$ where $w=w_1+{\rm i}w_2$. Both eigenvalues of this matrix must be non-negative
(since $w_*$ is a minimum of $\hat{h}$), so ${\rm tr}(\hat{h}_{ij})=\nabla_w^2\hat{h}(w_*)\geq 0$. But, by part (i), $\nabla_w^2\hat{h}\leq 0$ on $\mathcal{H}$, so $\nabla_w^2 \hat{h}(w_*)=0$. 
Hence, by (\ref{PDES2}), 
$$
|w_*-1|^2 {\rm e}^{h(w_*)} = |w_*+1|^2.
$$
But $|w+1|>|w-1|$ on $\mathcal{H}$, so $\hat{h}(w_*)>0$, contradicting the assumption that $\hat{h}$ attains a
{\em negative} minimum at $w_*$.
\end{proof}

Next, we provide a pointwise estimate for $w \mapsto F(w,\hat h(w))$ on ${\mathcal{H}}$ from below.

\begin{lemma} \label{lem:Ffrombelow}
For all $\eps>0$ and all $w=: |w|{\rm e}^{{\rm i}\theta} \in {\mathcal{H}}$,
$$
F(w,\hat h(w)) \ge- \frac{2|w|\cos \theta}{1+|w|^2}.
$$
\end{lemma}

\begin{proof}
It is clear that the rational function
$$
f(a,b,c):= \frac{ac-b}{ac+b}=1-\frac{2b}{ac+b}
$$
increases monotonically with $c$ when $a,b>0$. Hence
$$
f(a,b,c)\ge f(a,b,1)=-\frac{b-a}{b+a}
$$
for all $c\ge 1$. Now we apply this in the case $a=|w-1|^2$, $b=|w+1|^2$ and $c={\rm e}^{\hat h(w)}\ge 1$; the latter inequality is guaranteed by Proposition~\ref{signestimates} (ii).
\end{proof}

The next estimate, of the squared norm of $\hat h$ in the bilinear form associated to $-\Delta_{S^2}$ \ignore{(see \cite{Joh}, p.~191)}, is just enough for our present purposes.

\begin{lemma} \label{lem:justenough}
$\langle \hat h, -\Delta_{S^2}\hat h \rangle_{L^2(S^2)} \le C \eps.$
\end{lemma}
\begin{proof}
By reflexion symmetry, (\ref{PDES2}) and Proposition~\ref{signestimates},
\begin{eqnarray*}
\langle \hat h, -\Delta_{S^2}\hat h \rangle_{L^2(S^2)}&  = & 4\eps^2 R^2 \int_{\mathcal{H}} \left( \frac{1+|w|^2}{1+\eps^2 |w|^2}\right)^2 (-F(w,\hat h (w)))\hat h(w) \frac{{\rm d}^2 w}{(1+|w|^2)^2} \\
&\le & C^2 \eps^2 \int_\mathcal{H} \frac{-F(w,\hat h (w))}{(1+\eps^2 |w|^2)^2}	\, {\rm d}^2w   \\
& \le &  C \eps^2 \int_0^\infty \frac{r^2\, {\rm d}r}{(1+\eps^2 r^2)^2(1+r^2)}  \\
& \le & C \eps^2 \, \frac{1}{\eps(\eps+1)^2} \\
& \le & C\eps.
\end{eqnarray*}
We made use of  Corollary~\ref{cor:C0est} in the second step, and of Lemma~\ref{lem:Ffrombelow} in the third.
\end{proof}

\begin{corollary} \label{cor:H1est}
$\| \hat h \|_{H^1(S^2)} \le C\eps^{1/2}.$
\end{corollary}
\begin{proof}
This follows from Lemma~\ref{lem:justenough} by applying Proposition~\ref{SEE}(iii), noting that $\ip{1,\hat h}_{S^2}=0$ {by} reflexion symmetry.
\end{proof}

Let us choose and fix a pair of open disks $\DD'\subset\DD\subset \mathcal{H}$ centred at $w=1$, with $\ol{\DD'}\subset\DD$, for instance $\DD'=B_{1/4}(1)$, $\DD=B_{1/2}(1)$. Recall our notation $w=:w_1+{\rm i}w_2$, and let $\partial_1\hat h (w):=  \frac{\partial \hat h}{\partial w_1}(w)$. We  perform our last estimates on $\DD$, on which $\hat h$ satisfies (\ref{PDES2ish}).

\begin{lemma} 	\label{lem:H1der}
$ \| \partial_1 \hat h \|_{H^2(\DD')} \le C\eps^{1/2}.$
\end{lemma}

\begin{proof}
Differentiating (\ref{PDES2ish}) with respect to $w_1$, we obtain
\begin{equation*}
\Delta_{\RR^2}\partial_1 \hat h = - \frac{32\eps^4 R^2 w_1}{(1+\eps^2|w|^2)^3} F(w,\hat h(w)) 
+ \frac{8\eps^2 R^2}{(1+\eps^2 |w|^2)^2}\left(F_1(w,\hat h) + F_3(w,\hat h)\partial_1 \hat h  \right),
\end{equation*}
where $F_1$ and $F_3$ denote partial derivatives of the function of three real variables $(w_1,w_2,u)\mapsto F(w_1+{\rm i}w_2,u)$ defined in (\ref{functionF}).

By Corollary~\ref{cor:C0est} (and elementary estimates that we omit), the functions $F(w,\hat h(w))$, $F_1(w,\hat h(w))$ and $F_3(w,\hat h(w))$ on $\DD$
are bounded independent of $\eps$. Hence
$$
\| \Delta_{\RR^2}\partial_1 \hat h \|^2_{L^2(\DD)} \le C\eps^4 \left(1+ \|  \partial_1 \hat h\|^2_{L^2(\DD)} \right).
$$
Hence, by Proposition~\ref{SEE}(i),
\begin{eqnarray*}
 \| \partial_1 \hat h\|^2_{H^2(\DD')}& \le& C \left(  \| \Delta_{\RR^2}\partial_1 \hat h \|^2_{L^2(\DD)} + \| \partial_1 \hat h \|^2_{L^2(\DD')}  \right)  \\
 &\le & C\left( \eps^4 + \| \partial_1 \hat h \|^2_{L^2(\DD)}\right) \\
 &\le & C\left( \eps^4 + \|  \hat h \|^2_{H^1(\DD)}\right)   \\
 &\le & C\left( \eps^4 + \|  \hat h \|^2_{H^1(S^2)}\right)   \\
 & \le & C(\eps^4 + \eps),
\end{eqnarray*}
where we used Corollary~\ref{cor:H1est} in the last step. This proves the lemma.
\end{proof}

\begin{corollary} \label{cor:C0der}
$ \| \partial_1 \hat h \|_{C^0(\DD')} \le C\eps^{1/2}.$
\end{corollary}

\begin{proof}
This follows immediately from Lemma~\ref{lem:H1der} and Proposition~\ref{SOB}.
\end{proof}

Corollary~\ref{cor:C0der} implies that $\lim_{\eps\ra 0} \cd_1\hat h|_{w=1}=0$, which, by 
Equation (\ref{thequantity}), establishes the limit (\ref{suggestion}). This completes the proof of Theorem~\ref{thmvol}.

\subsection{Incompleteness of ${\sf M}_{(1,1)}^{\PP^1}(S^2_R)$}\label{sec:inco}

In this section we will prove

\begin{theorem}\label{incompleteness}
${\sf M}_{(1,1)}^{\PP^1}(S^2_R)$ is incomplete with respect to the $L^2$ metric.
\end{theorem}

It is helpful to denote the unique solution of (\ref{PDES2}) by
$\hat h_\eps$, to remind ourselves that this function depends parametrically on $\eps$. By the implicit function theorem
\cite[p.\ 420]{DonKro} applied to the smooth mapping
$$
\R\oplus H^{k+2}(S^2)\ra H^k(S^2),\qquad (\eps,\hat h)\mapsto -\Delta_{S^2}\hat h+2 R^2\eps^2f_\eps(w)^2F(w,\hat h),
$$
(where $k\geq 0$ and $f_\eps(w)=(1+|w|^2)/(1+|\eps w|^2)$)
the function $u_\eps:=\cd \hat h_\eps/\cd \eps:S^2\ra\R$ is smooth and satisfies the linear inhomogeneous PDE
\beq\label{uPDE}
-\Delta_{S^2}u_\eps+4R^2\eps \left(f_\eps+\eps\frac{\cd f_\eps}{\cd\eps}\right)f_\eps F(w,\hat{h}_\eps)
+2R^2\eps^2 f_\eps^2F_3(w,\hat h_\eps)u_\eps=0,
\eeq
where, as before,
$$
F_3(w,v):=\frac{\cd F}{\cd v}(w,v)=\frac{2|w^2-1|^2e^v}{(|w-1|^2e^v+|w+1|^2)^2}\geq 0.
$$
Theorem \ref{incompleteness} will follow from a suitable estimate of $\cd_1 u_\eps|_{w=1}$. 

As before, we use $C$ to denote a variable bounding constant, independent of $\eps$.

\begin{lemma} \label{demall}
$\|u_\eps\|_{H^1(S^2)}\leq C/\sqrt{\eps}$. 
\end{lemma}

\begin{proof}
Since $\hat{h}_\eps$ is odd under $w\mapsto-w$ for all $\eps\in(0,1)$, $u_\eps$ also has this symmetry. Hence
$$
\ip{u_\eps,-\Delta_{S^2}u_\eps}_{L^2(S^2)}=2\ip{u_\eps,-\Delta_{S^2}u_\eps}_{L^2(\mathcal{H})},
$$
where, once again, $\mathcal{H}$ denotes the open hemisphere $w_1>0$ (with its round metric). Now, (\ref{uPDE}) and (\ref{PDES2}) imply that
$$
-\Delta_{S^2}u_\eps+\frac{2}{\eps}\left(\frac{1-\eps^2|w|^2}{1+\eps^2|w|^2}\right)\Delta_{S^2}\hat{h}_\eps+2R^2\eps^2f_\eps(w)^2 F_3(w,\hat{h}_\eps)u_\eps=0,
$$
and hence, having noted that
 $F_3(w,\hat h_\eps)\geq 0$,
\ben
\ip{u_\eps,-\Delta_{S^2}u_\eps}_{L^2(S^2)}&\leq&\frac4\eps\ip{u_\eps,\frac{1-\eps^2|w|^2}{1+\eps^2|w|^2}(-\Delta_{S^2}\hat{h}_\eps)}_{L^2(\mathcal{H})}\\
\ignore{&\leq&\frac4\eps\ip{\left|\frac{1-\eps^2|w|^2}{1+\eps^2|w|^2}u_\eps\right|,(-\Delta_{S^2}\hat h_\eps)}_{L^2(\mathcal{H})}\\}
&\leq&\frac4\eps\ip{|u_\eps|,(-\Delta_{S^2}\hat h_\eps)}_{L^2(\mathcal{H})}
\een
since $-\Delta_{S^2}\hat{h}_\eps\geq 0$ on $\mathcal{H}$ by Proposition \ref{signestimates}, and $|(1-\eps^2|w|^2)/(1+\eps^2|w|^2)|\leq 1$. Now,
by its odd reflexion symmetry, ${u}_\eps$ vanishes on $\cd\mathcal{H}$, so
$$
\ip{|u_\eps|,-\Delta_{S^2}\hat h_\eps}_{L^2(\mathcal{H})}=\ip{\d|u_\eps|,\d\hat{h}_\eps}_{L^2(\mathcal{H})}
\leq\|\d|u_\eps|\|_{L^2(\mathcal{H})}\|\d\hat{h}_\eps\|_{L^2(\mathcal{H})}.
$$
Hence, again exploiting the symmetry of $u_\eps$ and $\hat{h}_\eps$,
$$
\ip{u_\eps,-\Delta_{S^2}u_\eps}_{L^2(S^2)}\leq\frac2\eps\|\d|u_\eps|\|_{L^2(S^2)}\|\d\hat{h}\|_{L^2(S^2)}
\leq\frac2\eps\|u_\eps\|_{H^1(S^2)}\|\hat{h}\|_{H^1(S^2)}
\leq\frac{C}{\sqrt{\eps}}\|u_\eps\|_{H^1(S^2)}
$$
by Corollary \ref{cor:H1est}. Since $\ip{1,u_\eps}_{L^2(S^2)}=0$ by reflexion symmetry, the Lemma now follows from Proposition \ref{SEE}(iii).
\end{proof}

As in Section \ref{sec:vol11}, let $\DD'\subset\DD\subset\R^2$ be two open disks centred on $w=1$, with $\ol{\DD'}\subset\DD$.

\begin{lemma}\label{pardol}
$\|\cd_1u_\eps\|_{H^2(\DD')}\leq C/\sqrt{\eps}$.
\end{lemma}

\begin{proof} Differentiating the PDE (\ref{uPDE}) for $u_\eps$ with respect to $w_1$, and using Corollary \ref{cor:C0est}
we see that $\cd_1u_\eps$ satisfies an equation of the form
$$
-\Delta_{\R^2}u_\eps+\eps G^{(1)}_\eps+\eps^2 G^{(2)}_\eps u_\eps +\eps^2G^{(3)}_\eps \cd_1 u_\eps=0,
$$
where $G^{(k)}_\eps$ are continuous functions on $\DD$ satisfying uniform bounds $\|G^{(k)}_\eps\|_{C^0(\DD)}\leq C$. Hence, by
Proposition \ref{SEE}(i),
\ben
\|\cd_1u_\eps\|_{H^2(\DD')}&\leq&C(\|\eps G^{(1)}_\eps+\eps^2 G^{(2)}_\eps u_\eps+\eps^2 G^{(3)}_\eps\cd_1u_\eps\|_{L^2(\DD)}
+\|\cd_1 u_\eps\|_{L^2(\DD')}\|)\\
&\leq& C(\eps+\eps^2\|u_\eps\|_{L^2(\DD)}+\|\cd_1u_\eps\|_{L^2(\DD)})\\
&\leq& C(\eps+\|u_\eps\|_{H^1(\DD)})\\
&\leq& C(\eps+\|u_\eps\|_{H^1(S^2)})\\
&\leq& \frac{C}{\sqrt{\eps}}
\een
by Lemma \ref{demall}.
\end{proof}

We may now complete the proof of Theorem \ref{incompleteness}. It suffices to exhibit a curve of finite length that escapes every compact subset
of $M_{1,1}^{\PP^1}(S^2_R)$. We claim that the curve $q(\eps)$ of centred pairs, $0<\eps\leq 1$, is such a curve. Certainly, it escapes every compact subset. Its length, by (\ref{g0pairsS2}), is
$$
L=\int_0^1\sqrt{2\pi \left(  \frac{8R^2}{(1+\eps^2)^2} + \frac1\eps\frac{ {\d}\:}{{\rm d}\eps}(\eps b(\eps))\right)}\, {\rm d}\eps,
$$
where $\eps b(\eps)=\cd_1\hat h_\eps|_{w=1}-1$. Hence,
$$
L\leq 4R\sqrt{\pi}+\sqrt{2\pi}\int_0^1 \eps^{-1/2}|\cd_1 u_\eps(1)|^{1/2}\,{\rm d}\eps.
$$
By Lemma \ref{pardol} and the Sobolev inequality (Proposition \ref{SOB}),
$$
|\cd_1 u_\eps(1)|\leq \|\cd_1u_\eps\|_{C^0(\DD')}\leq C\|\cd_1u_\eps\|_{H^2(\DD')}\leq \frac{C}{\sqrt{\eps}}.
$$
Hence
$$
L\leq C+C\int_0^1\eps^{-3/4}\, {\rm d}\eps\leq C.
$$

\section{Geometry of the moduli spaces from GLSMs} \label{sec:GLSMs}\news

From now on, we will again suppose that  $\Sigma$ is any compact oriented surface, and its genus will be denoted by $g$.

Our present aim is to derive conjectural formulae  for the  $L^2$ volume and the total scalar curvature of the moduli spaces $\mathsf{M}_{(k_\tp,k_\tm)}^{\PP^1}(\Sigma)$ in all generality, by
employing an auxiliary Abelian GLSM with toric target $\CC^2$. This {idea was} briefly discussed in Section~4 of \cite{SchGLSM} for $\Sigma=\RR^2$, but one needs to be more careful when $\Sigma$ is
a compact surface.

The construction in~\cite{SchGLSM} extends a familiar setup~\cite{MorPle} to study (ungauged) sigma models
targeted in a K\"ahler quotient, relying on the physical interpretation of such a sigma model as a {low energy} effective theory for a parent GLSM.\, 
Taking this a little further, one may expect that certain quantities attached to the sigma model 
{can be} recovered as limits of corresponding quantities for the GLSM in the limit of large coupling constant
$e^2 \rightarrow \infty$. Elaborating on this argument, Baptista derived conjectural {formulae} for the $L^2$ volume and total scalar curvature of moduli spaces of compact holomorphic curves $\Sigma \rightarrow \PP^n$ in~\cite{BapL2}. His formulae have been rigorously justified in an infinite family of cases with $\Sigma=S^2$ in \cite{SpeS2Pk,Alq}, and {are consistent with
rigorous results on $L^2$ metrics for general $\Sigma$ \cite{Liu}}.

\subsection{The parent GLSM} \label{sec:GLSM}

In our context, the novelty with respect to Baptista's argument in~\cite{BapL2} is that one needs to arrange for some of the  gauge symmetry present in the GLSM to survive the large $e^2$ limit of the type mentioned above.
We implement this by introducing {two} independent coupling constants $e_1, e_2$ in a parent GLSM with structure group $\TT^2:={\rm U}(1)_1 \times {\rm U}(1)_2$;
the parameters $e_1, e_2$ will be thought of as (inverse) lengths of two fixed vectors of an orthogonal basis of ${\rm Lie}(T^2)^*\cong \RR^2$ with respect to a given invariant metric on $\TT^2$.
The gauged $\PP^1$-model will be obtained as a limit $e_1^2 \rightarrow \infty$ at constant $e^2_2=1$. In fact, the argument generalises to gauged nonlinear sigma-models with higher-dimensional and more complicated quotient targets~\cite{SchGLSM}, but we prefer to make our discussion as concrete as possible.

The parent model is constructed from a $\TT^2$-principal bundle $P\rightarrow \Sigma$, to which one associates a rank-2 vector bundle $P^{\CC^2}_{\rho} \rightarrow \Sigma$  via the representation $\rho$ with weights $\mathbf{Q}_\tpm=(Q_\tpm^1,Q_\tpm^2) \in {\rm Lie}(\TT^2)_\ZZ^*$
$$
\rho: \left((\lambda_1,\lambda_2),  (X_\tp,X_\tm) \right) \mapsto \left( \lambda_1^{Q_\tp^1} \lambda_2^{Q_\tp^2} X_\tp, \lambda_1^{Q_\tm^1} \lambda_2^{Q_\tm^2} X_\tm\right).
$$ 
Let us parametrise the moment maps $\mu: \CC^2 \rightarrow  {\rm Lie}(\TT^2)^*$ for this action with respect to the canonical symplectic structure of $\CC^2$ as
$$
\mu_j(X_\tp,X_\tm)= \frac{1}{2} \left(  Q^j_\tp |X_\tp|^2 + Q^j_\tm |X_\tm |^2-\tau_j \right)\qquad j=1,2
$$
for some $(\tau_1,\tau_2)\in \RR^2$. 
We equip the gauge group with a deformed Riemannian metric 
$$
g_{\TT^2}=e_1^{-2}\d\vartheta_1^2+e_2^{-2}\d\vartheta_2^2
$$
where $\lambda_i=:{\rm e}^{{\rm i}\vartheta_i}$.

{These choices determine a GLSM with energy functional (\ref{GSM}) on $\mathbb{T}^2$-connexions $A=(A_1,A_2)$ in $P\rightarrow \Sigma$ and sections $\phi=(\phi_\tp,\phi_\tm)$ of $P^{\C^2}_\rho \rightarrow \Sigma$,
$$
E_1=\frac12\int_{\Sigma}\left(\frac{|F_{A_1}|^2}{e_1^2}+\frac{|F_{A_2}|^2}{e_2^2}+|\d^{A}\phi_\tp|^2+
|\d^{A}\phi_\tm|^2+e_1^2\mu_1(\phi_\tp,\phi_\tm)^2+e_2^2\mu_2(\phi_\tp,\phi_\tm)^2\right)
$$
where $\d^A\phi_\tpm=(\d-{\rm i}Q_\tpm^1A_1-{\rm i}Q_\tpm^2A_2)\phi_\tpm$. 
The associated vortex equations are
\bea\label{vortlin1}
\bar\partial^A \phi_{\tpm}&=&0,\\
\label{vortlin2}
*F_{A_j} =-\mu_j^\sharp \circ \phi &=& \frac{e_j^2}{2}\left(\tau_j-Q^j_\tp|\phi_\tp|^2-Q_\tm^j |\phi_\tm|^2\right),  \qquad j=1,2.
\eea
}

Let us denote by $(k_{1},k_2)$ the components of ${\rm deg}\, P \in H^2(\Sigma;\ZZ)^{\oplus 2} \cong \ZZ^{\oplus 2}$ with respect to the top degree class ${\rm PD}[\{{\rm pt}\}]\in H^2(\Sigma;\ZZ)$.
We also set 
\begin{equation}\label{npm}
k_\tpm:=\sum_{j=1}^2Q_\tpm^j k_j,
\end{equation}
assumed from now on to be nonnegative. As the notation suggests, these integers will end up being matched with the topological invariants (\ref{npmreally}) of the gauged $\PP^1$-model on $\Sigma$.

Taking the Hodge dual to (\ref{vortlin2}) and integrating over $\Sigma$, we obtain the equations
\begin{equation*}\label{linBradlow}
\frac{2\pi k_j}{e_j^2} = \frac{\tau_j}{2}{\rm Vol}(\Sigma)  -  \frac{1}{2} \left(Q_\tp^j\|\phi_\tp \|_{L^2}^2 + Q_\tm^j \|\phi_\tm \|_{L^2}^2  \right),
\end{equation*}
from which the Bradlow bounds~\cite{Bra} appropriate to the parent GLSM can be inferred: they amount to the (necessary) demand that both square norms $\|\phi_\tpm\|^2_{L^2}$ be nonnegative.
Geometrically, this is equivalent to the following statement, whose proof now reduces to straightforward computation:

\begin{lemma} \label{Lem:lin}
A necessary condition for the vortex
equations (\ref{vortlin2}) to admit solutions
is that the vector $\mathbf{\tau}=(\tau_1,\tau_2)$, regarded as an element of ${\rm Lie} (\TT^2)^* \cong \RR^2$, be contained in the closed affine cone with vertex at
the point $\left( \frac{4\pi k_1}{ e_1^2{\rm Vol(\Sigma)}}, \frac{4\pi k_2}{ e_2^2{\rm Vol(\Sigma)}}\right)$ and generated by the two  independent vectors $\mathbf{Q}_+$ and $\mathbf{Q}_-$ in the weight
lattice.
\end{lemma}

Provided $k_\tpm$ are sufficiently large in comparison with the genus of $\Sigma$ and the inequalities of Lemma \ref{Lem:lin} are satisfied strictly, the moduli space of solutions of the vortex
equations of this model is nonempty, and it admits a simple description in terms of the zero sets
of $\phi_\tpm$.

\begin{theorem} \label{prop:modulispaceC2}
Suppose that 
\beq\label{nposstable}
k_\tp\geq k_\tm>\max\{2g-2,0\}
\eeq
and that $(e_1,e_2)$ and $(\Sigma,g_\Sigma)$ are such that the point $(\tau_1,\tau_2)$ is in the interior of the cone with vertex at 
$\left( \frac{4\pi k_1}{ e_1^2{\rm Vol(\Sigma)}}, \frac{4\pi k_2}{ e_2^2{\rm Vol(\Sigma)}}\right)$ and generated by the
vectors $\mathbf{Q}_+$ and $\mathbf{Q}_-$. 
Then for each pair of effective divisors $(D_\tp,D_\tm)$ of degrees $k_\tp$, $k_\tm$ respectively,
there exists a unique gauge equivalence class of solutions of (\ref{vortlin1}) and (\ref{vortlin2}) with $(\phi_{\tpm})=D_\tpm$. Hence the moduli space
of vortices associated to the GLSM introduced above admits the description
\begin{equation} \label{productsyms}
\mathsf{M}_{(k_\tp,k_\tm)}^{\CC^2, e}(\Sigma)\cong{\rm Sym}^{k_\tp}(\Sigma) \times {\rm Sym}^{k_\tm}(\Sigma).
\end{equation}
\end{theorem}

\begin{proof}
We will start by applying Theorem 4.3 of reference~\cite{BapL2}, which encapsulates results from \cite{MorPle} and \cite{Weh}. First note that
the condition imposed on $(\tau_1,\tau_2)$ and the inequalities (\ref{nposstable})
imply that Assumptions (i) 4.1 and (ii) 4.2 of \cite{BapL2}, respectively, are satisfied.
The theorem then establishes that the moduli space admits a GIT-quotient description
\begin{equation} \label{MasGIT}
\mathsf{M}_{(k_+,k_-)}^{\CC^2, e}(\Sigma)\cong(\mathcal{V}^{\rm ss})/(\TT^2)^{\CC}.
\end{equation}
The right-hand side refers to the action $\rho^\CC$ of the complexification of the torus $\TT^2$ serving as structure group of the GSLM on the space $\mathcal V$ of solutions of the equations (\ref{vortlin1}) with topological charges $k_\pm$, and the semistable set $\mathcal{V}^{\rm ss} \subset \mathcal{V}$ consists of nonzero solutions.
More precisely: a $\TT^2$-connexion in the principal $\TT^2$ bundle $P\rightarrow \Sigma$ with degree ${\bf k}=(k_1,k_2)$ equips the associated rank two vector bundle $P^{\CC^2} \rightarrow \Sigma$ with a holomorphic structure, while the splitting $P^{\CC^2}\cong L_\tp\oplus L_\tm$ into degree $k_\tpm$ line bundles $L_\tpm$ yields a decomposition  $\phi=(\phi_\tp,\phi_\tm)$ of a holomorphic section into sections $\phi_\tpm$ of $L_\pm$ with induced holomorphic structures. 

We end up with a description of the space of solutions $\mathcal V$ entering (\ref{MasGIT}) as a product $\mathcal{V}=\mathcal{V}_\tp \times \mathcal{V}_\tm$ of two vector bundles
$$
\mathcal{V}_\pm \rightarrow  {\rm Pic}^{k_\tpm}(\Sigma) 
$$
over the Picard groups parametrising holomorphic line bundles over $(\Sigma,j_\Sigma)$ with degrees $k_\pm$, whose fibres 
$H^0(\Sigma, L_\tpm)$ are spaces of holomorphic sections with respect to holomorphic structures varying over the base.
Note that the dimension of these fibres is constant by virtue of the hypothesis (\ref{nposstable}), ensuring that there are no special divisors in degree $k_\tpm$ in $\Sigma$, and so it is determined by the theorem of Riemann--Roch to be
$$
{\rm rk}\, \mathcal{V}_\pm = \dim H^0(\Sigma, L_\tpm) =  k_\tpm-g+1.
$$

The  ``ss'' superscript in (\ref{MasGIT}) instructs us to remove the zero sections before taking the quotient by each factor $\CC^*$, and in this way we obtain the product of projectivisations
$$
\mathsf{M}_{(k_+,k_-)}^{\CC^2,{e}}(\Sigma)\cong \PP(\mathcal{V}_\tp) \times \PP(\mathcal{V}_\tm).
$$
The identification of $\PP(\mathcal{V}_\tpm)$ with ${\rm Sym}^{k_\tpm}(\Sigma)=\Sigma^{k_\tpm}/\mathfrak{S}_{k_\tpm}$, taking each point $[\phi_\tpm] \in \PP H^0(\Sigma, L_\tpm)$ on a given fibre to the effective divisor of zeros $(\phi_\tpm)=(\phi_\tpm)_0$,  is a construction
familiar from the geometry of algebraic curves (see~\cite{ACGH}, p.~18).
\end{proof}

\subsection{The gauged $\PP^1$ model as a limit at shrinking ${\rm U}(1)_1$} \label{sec:shrink}

To make contact with the gauged $\PP^1$ model we are primarily interested in, we will now specialise the parent model introduced in the previous section to the choices $\mathbf{Q}_+=(1,1)$, $\mathbf{Q}_-=(1,0)$, $(\tau_1,\tau_2)=(4,2-2\tau)$ and $(e_1,e_2)=(e,1)$. Thus we have
$$
k_\tp=k_1+k_2, \qquad k_\tm=k_1
$$
and the necessary conditions of Lemma~\ref{Lem:lin} become
$$
-(1+\tau)+\frac{2\pi k_\tm}{e^2\Vol(\Sigma)}\leq \frac{2\pi(k_+-k_-)}{\Vol(\Sigma)}
\leq 1-\tau.
$$
We will think of the parameter $\tau$ as being fixed, but keep the coupling $e$ as a free parameter --- in this way we are considering a one-parameter family of GLSMs, whose K\"ahler moduli spaces will be
denoted by $\M^{\C^2,e}_{(k_\tp,k_\tm)}$ for emphasis. The limit $e\rightarrow \infty$ corresponds to shrinking the first factor ${\rm U}(1)_1$ of $\TT^2$.

From Theorem \ref{thmmodulispace} and Theorem \ref{prop:modulispaceC2} we see that vortices in both the $\PP^1$ model and the GLSM (with $e$ sufficiently large) are in one-to-one correspondence with {\em pairs} of
effective divisors $D_\tpm$ on $\Sigma$. These divisors are unconstrained in the GLSM, but must be 
disjoint in the $\PP^1$ model. Hence we have a natural injective map
$$
\iota_e:\M^{\PP^1}_{(k_\tp,k_\tm)}\ra \M^{\CC^2, e}_{(k_\tp,k_\tm)},
$$
namely the map which sends a $\PP^1$ vortex with effective divisors $(D_\tp,D_\tm)$ of zeros and poles, respectively, to the GLSM vortex
(in the model with electric charge $e$) with the same divisors of zeros. Note that its range is a dense open subset of $\M^{\CC^2, e}_{(k_\tp,k_\tm)}$. Both these moduli spaces have $L^2$ metrics, which we denote
$g_{L^2}$ and $g_{L^2}^e$. Our interest in the particular parametric family
of GLSM's defined above is explained by the following conjecture.

\begin{conjecture}\label{blaski}
 The family of metrics $\iota_e^*g_{L^2}^e$ on $\M^{\PP^1}_{(k_\tp,k_\tm)}$
converges uniformly to $g_{L^2}$ as \mbox{$e\ra\infty$.}
\end{conjecture}

If true, this has strong consequences for the geometry of $\M^{\PP^1}_{(k_\tp,k_\tm)}$, as we will
see.
Since the range of $\iota_e$ is dense, the volume of $(\M^{\PP^1}_{(k_\tp,k_\tm)},\iota_e^*g_{L^2}^e)$
equals the volume of $(\M^{\CC^2,e}_{(k_\tp,k_\tm)},g_{L^2}^e)$, which is finite (since $\M^{\CC^2, e}_{(k_\tp,k_\tm)}$ is compact) and can be computed exactly via cohomological arguments. Taking the limit $e\ra\infty$ one
deduces that (subject to Conjecture \ref{blaski}) $(\M^{\PP^1}_{(k_\tp,k_\tm)},g_{L^2})$ also has
finite volume, and obtains an explicit formula for this. Physical consequences for the
thermodynamics of a vortex gas then follow.

Conjecture \ref{blaski} is motivated as follows\footnote{We note that a similar conjecture in the case of the {\em ungauged} $\P^1$ model and a GLSM with gauge group ${\rm U}(1)$ was made by Baptista \cite{BapL2}  and subsequently proved by Liu \cite{Liu} (albeit with a somewhat different notion of convergence).}. 
Recall that the GLSM is based on a principal $\TT^2$-bundle $P$ of degree $(k_1,k_2)=(k_\tm,k_\tp-k_\tm)$ and that $P^{\C^2}$ denotes the associated rank 2 vector bundle.
Denote by $P_2$ the principal ${\rm U}(1)$-bundle of degree $k_2$ associated with the second factor of $P$ and $P_2^{\PP^1}$ the $\PP^1$ bundle associated to this via the usual ${\rm U}(1)$-action on $\PP^1$. In a local trivialization of $P$ define the forgetful map
$$
T:((A_1,A_2),(\phi_\tp,\phi_\tm))\mapsto(A_2,[\phi_\tp:\phi_\tm])
$$
which takes a (locally defined) connexion and section on $P^{\C^2}$ and produces a (locally defined)
connexion and section on $P_2^{\PP^1}$. This map is, in fact, independent of the choice of trivialization, and hence globalizes to produce a map
$$
T:\AA(P)\times\Gamma(P^{\C^2})\ra \AA(P_2)\times\Gamma(P_2^{\PP^1})
$$
which is formally an $L^2$ Riemannian submersion (its differential is an $L^2$ isometry on the orthogonal complement to its kernel).

Now fix a disjoint pair of effective divisors $(D_\tp,D_\tm)$ and consider the (unique up to gauge) one-parameter
family of GLSM vortices $((A_1^e,A_2^e),(\phi_\tp^e,\phi_\tm^e))$ with $(\phi^e_\tpm) = D_\tpm$, for $e$ sufficiently large. Equation (\ref{vortlin1}) immediately implies that $T(A^e,\phi^e)$ satisfies the $\P^1$ vortex equation (\ref{vort1P1}), that is
$$
\ol\cd^{A^e_2}[\phi_\tp^e:\phi_\tm^e]=0.
$$
Furthermore, (\ref{vortlin2}) with $j=1$ suggests that
$$
|\phi_\tp^e|^2+|\phi_\tm^e|^2=4+O(e^{-2}).
$$ 
Substituting this in (\ref{vortlin2}) with $j=2$ yields
\ben
*F_{A_2}&=&1-\tau-\frac{2|\phi_\tp^e|^2}{|\phi_\tp^e|^2+|\phi_\tm^e|^2}+O(e^{-2})\\
&=&\frac{|\phi_\tm^e|^2-|\phi_\tp^e|^2}{|\phi_\tm^e|^2+|\phi_\tm^e|^2}-\tau+O(e^{-2})\\
&=& u_3-\tau+O(e^{-2})
\een
where we have used the usual identification $\PP^1\cong S^2$ given by $[X_\tp:X_\tm]\mapsto
X_\tp/X_\tm$ and inverse stereographic projection from $(0,0,-1)$. This suggests that $T(A^e,\phi^e)$ also
solves the second $\PP^1$ vortex equation (\ref{vort2P1}) approximately, up to errors of order $e^{-2}$, and hence that $T(A^e,\phi^e)$ is a good approximation, at large $e$, to the $\P^1$ vortex with
divisor pair $(D_\tp,D_\tm)$. Recalling that $T$ is formally Riemannian,  it follows that lengths of tangent vectors to $\M^{\P^1}_{(k_\tp,k_\tm)}$
should be well approximated by lengths of tangent vectors to $\M^{\C^2,e}_{(k_\tp,k_\tm)}$ corresponding to
the same infinitesimal motion of divisors, and hence that, at large $e$, $g_{L^2}$ should be
close to $\iota_e^* g_{L^2}^e$, the approximation becoming arbitrarily close as $e\ra\infty$. This, in more informal terms, is what Conjecture \ref{blaski} means.

\subsection{Integral formulae for the gauged $\PP^1$ model}

Since the spaces $\mathsf{M}_{(k_\tp,k_\tm)}^{\CC^2,e}(\Sigma)$ are compact complex manifolds and come equipped with an  $L^2$ metric which is K\"ahler~\cite{BapL2}, one may associate a K\"ahler class
$$[\omega{^e_{L^2}}] \in H^2\left(\mathsf{M}_{(k_\tp,k_\tm)}^{\CC^2,e}(\Sigma);\RR\right)$$ to them. 
This K\"ahler class can be expressed in terms of standard 2-cohomology classes of symmetric products of $\Sigma$, whose definition we recall.

If $g>0$, the intersection form $\sharp$ on $H_1(\Sigma;\Z)\cong\Z^{2g}$ extends bilinearly to a symplectic form on $H_1(\Sigma;\ZZ)\otimes \RR$, and we may choose generators $\hat\alpha_1,\hat\alpha_2,\ldots,\hat\alpha_{2g}$ for $H_1(\Sigma;\Z)$ which form a symplectic basis for this symplectic vector space,
meaning
$$
\left(\sharp(\hat\alpha_i,\hat\alpha_j)\right)_{i,j=1}^{2g}=\left(\begin{array}{cc} 0 & \I_g \\ -\I_g & 0 \end{array}\right).
$$
Let $\alpha_i$ be the Poincar\'e dual of $\hat\alpha_i$. Then the 
cohomology ring of $\Sigma$ is generated by $\alpha_1,\alpha_2,\ldots,\alpha_{2g}$, which satisfy the relations
$$
\alpha_i \alpha_j=0\text{ for }i\ne j\pm g\qquad\text{and}\qquad \alpha_{i}\alpha_{i+g}=-\alpha_{i+g}\alpha_i = \beta\text{ for } 1\le i \le g,
$$
where $\beta \in H^2(\Sigma;\ZZ)$ is the fundamental class. If $g=0$, then $\beta$ is the only generator of the cohomology ring. For $k\ge 1$, the K\"unneth
formula allows one to describe the cohomology of the Cartesian product $\Sigma^k$ in terms of tensor products of the pullbacks
$\alpha_{i,l}:=\pi_l^*\alpha_i$ and $\beta_l:=\pi_l^*\beta$ under the projections $\pi_l:\Sigma^k\rightarrow \Sigma$ for $1\le l \le k$ onto each factor, whereas a result of Macdonald establishes that the cohomology of the symmetric product ${\rm Sym}^k(\Sigma)=\Sigma^k/\mathfrak{S}_k$ in each degree $0\le j \le 2k$ consists of the $\mathfrak{S}_k$-invariant part
$$
H^j({\rm Sym}^k(\Sigma);\ZZ)=H^j(\Sigma^k;\ZZ)^{\mathfrak{S}_k};
$$
this follows from combining the assertions (4.1) and (12.1) in \cite{Mac}.
To construct presentations (see \cite{Mac}, \cite{BerTha}) for the cohomology ring of ${\rm Sym}^k(\Sigma)$, the following invariant elements are useful:
\begin{eqnarray*}
\eta & := & \sum_{l=1}^k \beta_l,\\
\xi_i & := &\sum_{l=1}^k \alpha_{i,l}\qquad \text{for }{1\le i \le 2g} \text{ if } g>0,\\
\sigma_i& := &\xi_i\xi_{i+g} \qquad \text{for }{1\le i \le g} \text{ if } g>0.
\end{eqnarray*}
In addition, the sum
$$
\theta:=\sum_{i=1}^g \sigma_i
$$
is quite natural: it corresponds to the pullback of the theta class
\begin{equation}\label{Theta}
\Theta=\sum_{i=1}^g[{\rm d}x_i \wedge {\rm d}x_{i+g}]
\end{equation}
on the Jacobian variety of $\Sigma$ (see Proposition (2.1) in \cite{BerTha})
under the Abel--Jacobi map 
$${\rm AJ}_k: {\rm Sym}^k(\Sigma)\rightarrow {\rm Pic}^k(\Sigma) \cong {\rm Jac}(\Sigma):=\frac{H^0(\Sigma,K_\Sigma)^*}{H_1(\Sigma;\ZZ)}\cong\frac{H_1(\Sigma;\ZZ)\otimes \RR}{H_1(\Sigma;\ZZ)}.$$
In (\ref{Theta}) we use $x_i$ to denote Cartesian coordinates on $H_1(\Sigma;\ZZ)\otimes \RR \cong \RR^{2g}$ associated to the symplectic basis
of $H_1(\Sigma;\ZZ)\otimes\R$ that we fixed at the beginning. In particular, one has that
\begin{equation} \label{thetag}
\theta^j=0 \quad \text{for }j>g.
\end{equation}

It is a well-known consequence (cf.\ proof of Corollary (4.3) in \cite{BerTha}) of the Poincar\'e formula (see p.~25 of \cite{ACGH}) that the classes $\eta$ and $\theta$ satisfy
\begin{equation} \label{binomialSym}
\int_{\Sym^{k}(\Sigma)} \eta^{k-j} \theta^j = \frac{g!}{(g-j)!} \qquad\text{ for } \quad 0\le j \le g. 
\end{equation}
This formula admits the following generalisation:

\begin{lemma}
For $0\le l \le g$ and $0\le j \le {\rm min} \{ g-l, k-l \}$, one has
\begin{equation} \label{genPoinc}
\int_{{\rm Sym}^k(\Sigma)} \eta^{k-j-l} \theta^j  \sigma_{i_1} \ldots \sigma_{i_l} = \frac{(g-l)!}{(g-j-l)!} 
\end{equation}
whenever $i_1,\ldots,i_l$ are any $l$ distinct indices from the set $\{1,2,\ldots,g\}$. 
\end{lemma}
Note that the integral (\ref{genPoinc}) is zero if any of the indices $i_h$ appear repeated, since the classes $\sigma_i$ are nilpotent by anticommutativity of the $\xi_i$.
\begin{proof}
We observe that the left-hand side of (\ref{genPoinc}) cannot depend on the choice of indices $i_1,\ldots,i_l$ for a fixed $l$, and then summing over all choices of ordered indices $i_1<\ldots <i_l$ we obtain
\begin{eqnarray*}
\int_{\Sym^{k}(\Sigma)} \eta^{k-j-l} \theta^{j} \sum_{i_1<\ldots <i_l}  \sigma_{i_1}\cdots \sigma_{i_l}& =&
\int_{\Sym^{k}(\Sigma)} \eta^{k-j-l} \theta^{j} \frac{1}{l!}\theta^l \\
&=& \frac{g!}{l!(g-j-l)!},
\end{eqnarray*}
where  the nilpotency of the $\sigma_i$ was used in the first step, and (\ref{binomialSym}) in the second step. Since there are $g \choose l$ different choices of ordered indices, the claim (\ref{genPoinc}) follows.
\end{proof}

At this point we turn to the moduli space $\mathsf{M}_{(k_\tp,k_\tm)}^{\CC^2,e}(\Sigma)$, which comes equipped with projections  
$$p_\pm: {\mathsf{M}_{(k_\tp,k_\tm)}^{\CC^2,e}(\Sigma)} \rightarrow {\rm Sym}^{k_\tpm} (\Sigma)$$
onto each factor of the product (\ref{productsyms}). We shall use sub- and superscripts $\pm$ to decorate the cohomology classes we have defined on the symmetric products to denote their pullbacks by $p_\pm$ to the moduli space. Thus 
$\eta_\tpm:=p_\pm^*c_1(\mathcal{O}_{\PP(\mathcal{V}_\tpm)}(1))$ are pullbacks of the first Chern classes of the tautological line bundles
over $\PP(\mathcal{V}_\tpm)$, whereas $\theta_\tpm:=({\rm AJ}_{k_\tpm}\! \circ p_\tpm)^* {\rm PD}[\Theta]$ are pullbacks of
the theta class in ${\rm Jac}(\Sigma)$ (independent of the choice of basepoints). We will also write
$$
\zeta_i:=\xi^\tp_{i}\xi^\tm_{i+g}-\xi^\tp_{i+g} \xi^\tm_{i},\qquad 1\le i\le g.
$$

We are now ready to state and prove the main result of this section.

\begin{theorem} \label{L2geomGLSM}
Under the conditions of Theorem~\ref{prop:modulispaceC2}, the K\"ahler class for the  $L^2$ metric on the vortex moduli space (\ref{productsyms}) of the GLSM with target $\CC^2$ is
given by
\begin{equation} \label{omegaL2}
[\omega_{L^2}^{e}] = \sum_{\sigma=\pm} (J_\sigma \eta_\sigma + K_\sigma \theta_\sigma) + 4 \pi^2 \sum_{i=1}^g \zeta_i,
\end{equation}
where
{\begin{eqnarray*}
J_\tp&:=& 2\pi (1-\tau) \Vols - 4\pi^2 ( k_\tp-k_\tm) ,\\
J_\tm&:=& 2\pi (1+\tau) \Vols - {4\pi^2}{e^{-2}}k_\tm+4\pi^2(k_\tp-k_\tm) ,\\
K_\tp&:=& 4 \pi^2,\\
K_\tm&:=& {4\pi^2}(1+{e^{-2}}) .
\end{eqnarray*}
This leads to
\begin{equation} \label{volC2}
\Vol \left(\mathsf{M}_{(k_\tp,k_\tm)}^{\CC^2, e}(\Sigma)\right) =  
 \sum_{\ell=0}^g \frac{g!(g-\ell)!}{(-1)^\ell \ell!} \prod_{\sigma=\pm} \sum_{j_\sigma=\ell}^g  \frac{(2\pi)^{2\ell} J_\sigma^{k_\sigma-j_\sigma}K_\sigma^{j_\sigma-\ell}}{(j_\sigma-\ell)!(g-j_\sigma)!(k_\sigma-j_\sigma)!}.
 \end{equation}
}
Defining further
\begin{eqnarray} \label{Chat}
S_\sigma (i_\tpm,j\tpm,k\tpm,\ell):= \frac{(k_\sigma+i_\sigma-2g) J_\sigma^{k_\sigma-i_\sigma-1} J_{-\sigma}^{k_{-\sigma} -j_\sigma} K_\sigma^{i_\sigma-\ell} K_{-\sigma}^{j_\sigma-\ell} } {(i_\sigma-\ell)! (j_\sigma-\ell)! (k_\sigma-i_\sigma-1)! (k_{-\sigma}-j_\sigma)! (g-i_\sigma)! (g-j_\sigma)!},
\end{eqnarray}
one also has
{\small
\begin{eqnarray*}
\int_{\mathsf{M}_{(k_\tp,k_\tm)}^{\CC^2,e}(\Sigma)} \!\!\!\!\!\!\!\!\!\!\!\!\!\!\!\!\! \Sc &=&4 \pi   (g!)^2 \sum_{\ell=0}^g  \frac{(-1)^\ell (k_\tp +k_\tm -2\ell-1)!}{\ell!(k_\tp+k_\tm-1)!} \sum_{\sigma=\pm} \sum_{i_\sigma=\ell}^{{\rm min}\{g,k_\sigma-1\}} 
\sum_{j_\sigma=0}^g S_\sigma (i_\tpm,j\tpm,k\tpm,\ell) .
\end{eqnarray*}
}
\end{theorem}

\begin{proof}

Our first claim is that the formula proposed by Baptista for the K\"ahler class $[\omega_{L^2}] \in H^2({\mathsf{M}_{(k_\tp,k_\tm)}^{\CC^2}(\Sigma)};\RR)$ in Theorem 4.4 of reference~\cite{BapL2} generalises to the parent GLSM with two coupling constants introduced in our Section~\ref{sec:GLSM} as follows:
\begin{equation} \label{formulaomegaL2}
[\omega_{L^2}] =   \sum_{j=1}^2 \left\{ \left( {\pi \tau_j} \Vols - \frac{4 \pi^2}{e_j^2}  k_j \right)\eta_j + \frac{4\pi^2}{e_j^2} \theta_j \right\}.
\end{equation}
Here, as in \cite{BapL2}, $\eta_j$ denote first Chern classes of the holomorphic line bundles $\mathcal{L}_j\rightarrow \M^{\CC^2,e}_{(k_\tp,k_\tm)}$ associated to
the principal $\TT^2$-bundle $\mathcal{V}^{\rm ss} \rightarrow \M^{\CC^2,e}_{(k_\tp,k_\tm)}$ introduced in the proof of Theorem~\ref{prop:modulispaceC2}
via the standard action of the subgroup ${\rm U}(1)_j \subset\TT^2 $ on $\CC$, whereas $\theta_j$ denote pullbacks of the theta classes in the Jacobian of $\Sigma$
under the maps $ \M^{\CC^2,e}_{k_\tp,k_\tm} \rightarrow {\rm Pic}^{k_j}(\Sigma)\cong {\rm Jac}(\Sigma)$ induced by the assignment of a line bundle $L_j\rightarrow \Sigma$  with holomorphic structure
${\bar\partial}^{A_j}$ to each pair $((A_1,A_2),(\phi_\tp,\phi_\tm))$. The derivation of (\ref{formulaomegaL2}) is entirely analogous to the one needed to
establish Baptista's formula, and will not be reproduced here.

Now we want to rearrange the sum over $j$ above as a sum over signs $\sigma=\pm$. We take advantage of the isomorphism
 $$
\Psi:
 {\rm Pic}^{k_\tp}(\Sigma) \times {\rm Pic}^{k_\tm}(\Sigma)
  \stackrel{\cong}{\longrightarrow} 
  {\rm Pic}^{k_1}(\Sigma) \times {\rm Pic}^{k_2}(\Sigma)
$$
given by
$$(L_\tp,L_\tm) \mapsto (L_\tm,L_\tp\otimes L^*_\tm)=:(L_1,L_2)$$ to reorganise the fibres of the vector bundle
$$
\mathcal{V} \rightarrow  {\rm Pic}^{k_1}(\Sigma) \times {\rm Pic}^{k_2}(\Sigma)
$$
considered by Baptista into a direct sum $\Psi^*\mathcal{V}$ whose two summands separately fibre as Picard vector bundles $\mathcal{V}_\tpm$ over each factor of the base. Under this isomorphism, the classes $\eta_j$ introduced by Baptista relate to the classes $\eta_\tpm=c_1(O_{\PP(\mathcal{V}_\tpm)}(1))$
defined above as
$$
\eta_1=\eta_\tm, \qquad \eta_2=\eta_\tp-\eta_\tm.
$$
The relation between the pull-backs of the theta classes in the two descriptions is slightly more complicated. We can write for $j=1,2$
$$
\theta_j=\sum_{i=1}^g \xi_{i}^{(j)}\xi_{i+g}^{(j)},
$$
where the $\xi_{i}^{(j)}$ (with $1\le i\le 2g$) are the natural generators for $H^1({\rm Pic}^{k_j}(\Sigma);\ZZ)$ obtained from the description ${\rm Pic}^{k_j}(\Sigma)\cong {\rm Jac}(\Sigma)$ as a complex torus; they relate to the  corresponding generators $\xi_i^\pm$ linearly as
$$
\xi_i^{(1)}=\xi_i^\tm\qquad \text{and} \qquad \xi_{i}^{(2)}=\xi_i^\tp-\xi_i^\tm, \qquad 1\le i\le 2g.   
$$
Hence we find
$$
\theta_1 = \theta_\tm \qquad \text{and} \qquad \theta_2= \sum_{i=1}^g(\xi_i^\tp - \xi_i^\tm)(\xi_{i+g}^\tp-\xi_{i+g}^\tm)= \theta_\tp + \theta_\tm-\sum_{i=1}^g\zeta_i .
$$

We observe that any monomial involving $\eta_\pm$ and an odd number of $\xi^\pm_i$s must vanish (see (5.4) in \cite{Mac} for a justification). Using this, it is not hard to check
that, under the assumption $k_\pm\ge 2g-1$,
\begin{equation} \label{liouville}
\frac{[\omega^e_{L^2}]^{k_\tp+k_\tm}}{(k_\tp+k_\tm)!} = \sum_{\ell=0}^{g} \frac{(2\pi)^{4\ell} (J_\tp \eta_\tp + J_\tm \eta_\tm+K_\tp\theta_\tp + K_\tm\theta_\tm)^{k_\tp+k_\tm-2\ell}}{(2\ell)!(k_\tp+k_\tm-2\ell)!}\left( \sum_{i=1}^g \zeta_i \right)^{2\ell},
\end{equation}
and furthermore, attending to the  nilpotency of the classes $\sigma_i$,
\begin{equation}\label{skewterm}
\left( \sum_{i=1}^{g} \zeta_i \right)^{2\ell} = (-1)^{\ell}(2\ell)! \sum_{i_1 <\cdots < i_\ell}\sigma^+_{i_1} \sigma^-_{i_1}\cdots \sigma^+_{i_\ell} \sigma^-_{i_\ell}.
\end{equation}

We now have from (\ref{liouville}) and (\ref{skewterm}), and noting that $\theta_s^j=0$ for $j>g$,
\begin{eqnarray*}
\int \frac{[\omega^e_{L^2}]^{k_\tp+k_\tm}}{(k_\tp+k_\tm)!} &=& \sum_{\ell=0}^g (-1)^{\ell} (2\pi)^{4\ell} \sum_{i_1<\cdots<i_\ell}\prod_{s=\pm}
\int_{{\rm Sym}^{k_s}(\Sigma)}\frac{(J_s\eta_s+K_s\theta_s)^{k_s-\ell}}{(k_s-\ell)!} \sigma^s_{i_1}\cdots \sigma^s_{i_\ell} \\
&=& \sum_{\ell=0}^g (-1)^{\ell} (2\pi)^{4\ell} \sum_{i_1<\cdots<i_\ell}\prod_{s=\pm} \sum_{j_s=0}^{{\rm min}\{g-\ell,k_s-\ell\}} \frac{J_s^{k_s-j_s-\ell}K_s^{j_s}}{j_s!(k_s-j_s-\ell)!}\frac{(g-\ell)!}{(g-j_s-\ell)!}\\
&=&  \sum_{\ell=0}^g (-1)^{\ell} {g\choose \ell}\prod_{s=\pm} \sum_{j_s=0}^{g-\ell} \frac{(2\pi)^{2\ell}J_s^{k_s-j_s-\ell}K_s^{j_s}}{j_s!(k_s-j_s-\ell)!} \frac{(g-\ell)!}{(g-j_s-\ell)!}
\end{eqnarray*}
after use of (\ref{genPoinc}) and (\ref{nposstable}), whence (\ref{volC2}) immediately follows.

To obtain the formula for the total scalar curvature, we take advantage of the splitting of the first Chern class of the product (\ref{productsyms}) as a sum over each factor:
$$
c_1({\rm T}(\Sym^{k_\tp}(\Sigma) \times \Sym^{k_\tm}(\Sigma))) = \sum_{s=\pm} c_1({\rm T}\Sym^{k_s}(\Sigma)).
$$
This reduces the calculation to the evaluation of integrals over the factors, as follows:
\begin{eqnarray*}
\int_{\mathsf{M}_{(k_\tp,k_\tm)}^{\CC^2, e}(\Sigma)} \!\!\!\!\!\!\!\!\!\!\!\!\! \Sc &=&
\frac{4\pi}{(k_\tp+k_\tm-1)!} \int_{\mathsf{M}_{(k_\tp,k_\tm)}^{\CC^2, e}(\Sigma)} \!\!\!\!\!\quad   c_1({\rm T}\mathsf{M}_{(k_\tp,k_\tm)}^{\CC^2, e}(\Sigma)) \wedge (\omega^e_{L^2})^{k_\tp+k_\tm-1} \\
&=& \frac{4 \pi}{(k_\tp+k_\tm-1)!} \sum_{\ell=0}^g  \sum_{s=\pm}{k_\tp + k_\tm -2\ell - 1 \choose k_s-\ell-1} \times \\
&& \times \int_{\Sym^{k_s}(\Sigma)} c_1({\rm T} \Sym^{k_s}(\Sigma))  \wedge (J_s \eta_s+K_s\theta_s)^{k_s-\ell-1} \sigma_{i_1}^s \cdots \sigma_{i_\ell}^s  \times \\
&& \times \int_{\Sym^{k_{\tm s}}(\Sigma)}(J_{\tm s} \eta_{\tm s} + K_{\tm s} \theta_{\tm s})^{k_{\tm s} -\ell}
\sigma_{i_1}^{-s} \cdots \sigma_{i_\ell}^{-s}.
\end{eqnarray*}
Using the formula for the first Chern class (see \cite{Mac}, \cite{BerTha})
$$
c_1({\rm T}\Sym^{k_s}(\Sigma)) = (k_s-g+1) \eta_s - \theta_s,
$$
and once again (\ref{genPoinc}) together with (\ref{thetag}), we arrive, after a somewhat tedious calculation, at the result stated.

\end{proof}

In light of Conjecture \ref{blaski} we expect that, in the limit $e\ra\infty$, the volume and
total scalar curvature of $\M^{\CC^2,e}_{(k_\tp,k_\tm)}(\Sigma)$ should converge to the
volume and scalar curvature of the moduli space of
$\PP^1$ vortices. That is:

\begin{conjecture} \label{volconj}
The total volume and the total scalar curvature of the moduli spaces $\mathsf{M}_{(k_\tp,k_\tm)}^{\PP^1}(\Sigma)$ are finite, and given by setting in the expressions for 
$$\Vol \left(\mathsf{M}_{(k_\tp,k_\tm)}^{\CC^2,e}(\Sigma)\right) \qquad \text{ and } \qquad
\int_{\mathsf{M}_{(k_\tp,k_\tm)}^{\CC^2,e}(\Sigma)} \!\!\!\!\!\!\!\!\!\!\! \Sc\qquad$$ obtained in Proposition~\ref{L2geomGLSM}, \swap{respectively,}{}
\begin{eqnarray*} \label{limCD}
J_\tpm &  := &  2\pi (1\mp \tau) \Vols \mp 4\pi^2 (k_\tp-k_\tm),  \\ 
K_\tpm &:= & 4\pi^2.
 \end{eqnarray*}
\end{conjecture}

Note that this conjecture is supported by our rigorous formula (\ref{volumeformula}), justified in Section~\ref{sec:vol11}, in the case
where $k_\tp=k_\tm=1, \Sigma=S^2_R$ and $\tau=0$.

\section{Thermodynamics of vortex-antivortex gases}  \label{sec:thermo}\news

We conclude our study by extracting some physical consequences of the general volume formulae for the moduli spaces ${\sf M}^{\PP^1}_{(k_\tp,k_\tm)}(\Sigma)$ that we have obtained.

In Section 4 of~\cite{ManNas}, the framework of Gibbs's classical statistical mechanics was used to derive an entropy formula as well as an equation of state for a (two-dimensional) gas of
vortices in a line bundle; see also \cite{Per} for a check of the volume formula in \cite{ManNas}, and  \cite{Man,RomGVB,EasRom} for related work in $g=0$. Here, we will take a similar approach to investigate the thermodynamics of a {\em gas mixture}
containing BPS vortices and antivortices in the gauged $\PP^1$ model at criticality, building on our volume formulae in the previous section.
As in~\cite{ManNas}, we shall assume that our gases contain a large number of particles ($k_\tp$ vortices and/or $k_\tm$ antivortices)
occupying a compact surface of large area $\Vols$. It is sensible to speak of a {\em thermodynamic limit} when these quantities become infinite while keeping
the particle densities
\begin{equation} \label{densities}
\nu_\tpm := \frac{k_\tpm}{\Vols} =: \frac{k_\tpm}{V}
\end{equation}
finite, and not simultaneously zero. Throughout this section, we shall use as in (\ref{densities}) the shorthand notation $V:= \Vols$. 

The starting point is the Hamiltonian governing the so-called {\em adiabatic approximation}~\cite{StuA} to the classical dynamics in the gauged sigma model in 2+1 dimensions, as proposed by Manton~\cite{ManSut,StuA}. This is a function $H$ on the phase space ${\rm T}^*\mathsf{M}_{(k_\tp,k_\tm)}^{\PP^1}(\Sigma)$ (equipped with its canonical symplectic structure $\omega_{\rm can}$) obtained by adding to the Hamiltonian for the geodesic flow of the  $L^2$ metric $g_{L^2}$, playing the role of kinetic energy, the constant-time energy bound in (\ref{Ebound}): 
$$H= \frac{1}{2}| \vartheta |^2_{g_{L^2}}+ 2\pi(1-\tau)k_\tp +  2\pi(1+\tau)k_\tm;$$
the latter gives a good estimate for the potential energy of the classical field theory dynamics from below, for small field momenta in configurations that evolve approximately (in the sense of the metric (\ref{L2innprod})) along vortex solutions.
We are using $\vartheta$ as notation for the
tautological 1-form with ${\rm d}\vartheta = -\omega_{\rm can}$. 

Let us denote the temperature by $T$, and also Planck's and Boltzmann's constants by $2\pi \hbar$ and $K$, respectively. Then the
partition function in classical statistical mechanics for the Hamiltonian system $({\rm T}^*\mathsf{M}_{(k_\tp,k_\tm)}^{\PP^1}(\Sigma), \omega_{\rm can}, H)$ is given by
\begin{eqnarray*}
\mathcal{Z}&=&\frac{1}{(2\pi\hbar)^{2(k_++k_-)}} \int_{{\rm T}^*\mathsf{M}_{(k_\tp,k_\tm)}^{\PP^1}(\Sigma)}   {\rm e}^{-\frac{H}{KT}} \frac{\omega_{\rm can}^{2(k_\tp + k_\tm)}}{(2(k_\tp + k_\tm))!} \\
&=& \left( \frac{KT}{2 \pi \hbar^2} \right)^{k_\tp+k_\tm}  \int_{\mathsf{M}_{(k_\tp,k_\tm)}^{\PP^1}(\Sigma)} {\rm e}^{- \frac{2\pi}{KT}((1-\tau)k_\tp + (1+\tau)k_\tm)} \frac{\omega_{L^2}^{k_\tp+k\tm}}{(k_\tp+k_\tm)!} \\
&=& \left( \frac{KT}{2 \pi \hbar^2} \right)^{k_\tp+k_\tm} {\rm e}^{- \frac{2\pi}{KT}((1-\tau)k_\tp + (1+\tau)k_\tm)} {\rm Vol} \left( \mathsf{M}_{(k_\tp,k_\tm)}^{\PP^1}(\Sigma) \right)  \\[10pt]
&=& \mathcal{Z}_\tp \mathcal{Z}_\tm
 \sum_{\ell=0}^g \frac{g!(g-\ell)!}{(-1)^\ell \ell!} \prod_{\sigma=\pm} \sum_{j_\sigma=\ell}^{g} 
  \frac{(2\pi)^{j_\sigma} [(1-\sigma\tau)V-2\pi(k_\sigma-k_{-\sigma})]^{k_\sigma-j_\sigma}}
 {(j_\sigma-\ell)!(k_\sigma-j_\sigma)!(g-j_\sigma)!},
\end{eqnarray*}
with
$$
\mathcal{Z}_\tpm:= \left( \frac{KT}{ \hbar^2} \right)^{k_\tpm} {\rm e}^{- \frac{2\pi}{KT}(1\tmp\tau)k_\tpm}.
$$
The first step  above involves a Gaussian integration along the cotangent fibres, which was justified in \cite{Man}, whereas the last step makes use of the volume formula in Conjecture~\ref{volconj}.

The Helmholtz free energy of the system can be computed as 
\begin{eqnarray*}
F&=&-KT \log \mathcal{Z}\\
&=&-KT (k_\tp + k_\tm) \log \frac{K T}{\hbar^2} + 2 \pi \sum_{\sigma = \pm}  (1-\sigma \tau)k_\sigma \\
&&\qquad -KT \log \left[\sum_{\ell=0}^g \frac{g!(g-\ell)!}{(-1)^\ell \ell!} \prod_{\sigma=\pm} \sum_{j_\sigma=\ell}^{g} 
  \frac{(2\pi)^{j_\sigma} [(1-\sigma\tau)V-2\pi(k_\sigma-k_{-\sigma})]^{k_\sigma-j_\sigma}}
 {(j_\sigma-\ell)!(k_\sigma-j_\sigma)!(g-j_\sigma)!}\right].
\end{eqnarray*}
We observe that the sum over $\ell$ can be rearranged as
\begin{equation} \label{lastterm}
\sum_{\ell=0}^g  (-1)^\ell {g \choose \ell}
\prod_{\sigma=\pm}  \frac{(2\pi V)^{k_\sigma}}{k_\sigma!}\sum_{j_\sigma=0}^{g-\ell}   {g -\ell \choose j_\sigma} 
\left( \frac{1-\sigma\tau}{2\pi} - \nu_\sigma + \nu_{-\sigma} \right)^{k_\sigma-j_\sigma-\ell} 
\frac{k_\sigma! 	\, {V}^{-j_\sigma-\ell}}{(k_\sigma-j_\sigma-\ell)!},
\end{equation}
where the last fraction enjoys the asymptotics
\begin{equation} \label{asymptnu}
\frac{k_\sigma! \, {V}^{-j_\sigma-\ell}}{(k_\sigma-j_\sigma-\ell)!} = \left( \frac{k_\sigma}{V}\right)^{j_\sigma+\ell} + O\left( \frac{1}{V}\right) =  
\nu_\sigma^{j_\sigma+\ell }+  O\left( \frac{1}{V}\right)
\qquad \text{as } V \rightarrow \infty
\end{equation} 
in the thermodynamic limit defined above. Truncating (\ref{asymptnu}) to the leading term,  we estimate (\ref{lastterm}) in this limit by
$$
\sum_{\ell=0}^g (-1)^\ell {g \choose \ell} 
\prod_{\sigma=\pm} \frac{(2\pi V)^{k_\sigma} \nu_\sigma^{\ell}}{k_\sigma!}  \left( \frac{1-\sigma\tau}{2\pi} - \nu_\sigma + \nu_{-\sigma} \right)^{k_\sigma-g} \left( \frac{1-\sigma\tau}{2\pi}  + \nu_{-\sigma} \right)^{g-\ell} =
$$
$$
= \left( \frac{1-\tau^2}{(2\pi)^2} + \frac{1-\tau}{2\pi}\nu_\tp + \frac{1+\tau}{2\pi}\nu_\tm \right)^g
\prod_{\sigma=\pm} \frac{(2\pi V)^{k_\sigma} }{k_\sigma!}  \left( \frac{1-\sigma\tau}{2\pi} - \nu_\sigma + \nu_{-\sigma} \right)^{k_\sigma-g}.
$$
To proceed, we employ Stirling's approximation for large $k_\sigma$
$$
\log (k_\sigma!) \approx k_\sigma \log k_\sigma - k_\sigma,
$$
and are led to the estimate
\begin{eqnarray}
F &\approx &2 \pi \sum_{\sigma = \pm}  (1-\sigma \tau)k_\sigma 
- KT \sum_{\sigma=\pm} k_\sigma \log \frac{ {\rm e} K T}{\hbar^2 \nu_\sigma} \nonumber 
- KT  \sum_{\sigma=\pm} k_\sigma \log  \left[ 1-\sigma\tau - 2\pi (\nu_\sigma-\nu_{\tm\sigma} ) \right] \\
&&-KTg  \log \frac{1-\tau^2 +2\pi[(1-\tau) \nu_\tp + (1+\tau)\nu_\tm]}{1-\tau^2-4\pi \tau(\nu_\tp-\nu_\tm) - (2\pi)^2(\nu_\tp - \nu_\tm)^2}.   \label{freeenergy}
\end{eqnarray}

We can use the asymptotic version (\ref{freeenergy}) of $F$ to obtain (an approximation to) the entropy of the system as
\begin{eqnarray} \label{firstentropy}
{S} = -\frac{\partial F}{\partial {T}} & \approx&
K\sum_{\sigma=\pm} k_\sigma \log \frac{{\rm e}^2 KT}{\hbar^2 \nu_\sigma} + 
K  \sum_{\sigma=\pm} (k_\sigma-g)  \log  \left[ 1-\sigma\tau - 2\pi (\nu_\sigma-\nu_{\tm\sigma} ) \right]  \nonumber \\
&&+ Kg  \log  \left[ {1-\tau^2 +2\pi[(1-\tau) \nu_\tp + (1+\tau)\nu_\tm]} \right].
\end{eqnarray}
Likewise, we calculate the pressure of the system (no principal bundles within this section!) to be
\begin{eqnarray} \label{thepressure}
{P}=-\frac{\partial F}{\partial V} & \approx&
 KT \sum_{\sigma=\pm} \nu_\sigma+KT  \sum_{\sigma=\pm} \left(\nu_\sigma-\frac{g}{V}\right) \frac{ 2\pi (\nu_\sigma - \nu_{\tm \sigma})}{1-\sigma\tau - 2\pi (\nu_\sigma-\nu_{\tm\sigma} )} \nonumber\\
 &&+\frac{gKT}{V}  \frac{2\pi[(1-\tau) \nu_\tp + (1+\tau)\nu_\tm]}{1-\tau^2 +2\pi[(1-\tau) \nu_\tp + (1+\tau)\nu_\tm]}.
\end{eqnarray}
This result would be expected to deliver an equation of state for the gas mixture; however, it does not yet incorporate the scaling properties that one usually demands on such an equation (which should make it amenable to a virial expansion~\cite{Cal}, for instance). To remedy this, we observe that in our thermodynamic limit 
one may discard the terms proportional to the ratio $\frac{g}{V}$, which tends to zero. By doing so, we
derive from (\ref{thepressure}) a simplified equation of state of the form 
\begin{eqnarray} \label{eqstate}
P &\approx& KT  \sum_{\sigma=\pm} \nu_\sigma +KT  \sum_{\sigma=\pm}  \frac{2 \pi \nu_\sigma(\nu_\sigma - \nu_{\tm \sigma})}
{1-\sigma\tau - 2\pi (\nu_\sigma-\nu_{\tm\sigma} )} \nonumber\\
&=&  \sum_{\sigma=\pm}  \frac{KT(1-\sigma\tau)\nu_\sigma}{1-\sigma\tau - 2\pi (\nu_\sigma-\nu_{\tm\sigma} )} \\
&=& \sum_{\sigma=\pm}  \frac{KTk_\sigma}{V - \frac{2\pi}{1-\sigma\tau}(k_\sigma-k_{\tm\sigma} )} \nonumber.
\end{eqnarray}
This admits the following interpretation: the total pressure $P$ is the sum of two partial pressures $P_\tpm$, each being a deformation to
the pressure $KTk_\tpm/V$ of an ideal gas taking into account a correction of the area $V$ resulting from particle interactions. The correction subtracts $k_\tpm$ times
a multiple of an effective vortex size $\frac{2\pi}{1\mp \tau}$ from the actual area $V$, as appropriate for rigid bodies obeying a Clausius equation of state, but also {\em adds}
to it $k_\tmp$ times the same characteristic size --- suggesting that the two species form bound states, counteracting the effect of rigid body motion through effective depletion of particles of a given species.

The result (\ref{eqstate}) has the attractive feature of leading to a complete virial expansion for two species
\begin{eqnarray}
\frac{P}{KT}&=&  \sum_{n=0}^\infty  \sum_{\sigma=\pm} \nu_\sigma \left( \frac{2\pi}{1-\sigma\tau} (\nu_\sigma-\nu_{\tm \sigma})\right)^n \nonumber \\
&=& \nu_\tp + \nu_\tm+\frac{2\pi}{1-\tau}\nu_\tp^2+\frac{2\pi}{1+\tau}\nu_\tm^2- \frac{4\pi}{1-\tau^2}\nu_\tp \nu_\tm  \nonumber \\
&&+\frac{4\pi^2}{(1-\tau)^2}\nu_\tp^3+\frac{4\pi^2}{(1+\tau)^2}\nu_\tm^3 -\frac{16\pi^2(1+\tau^2)}{(1-\tau^2)^2}(\nu_\tp^2 \nu_\tm+\nu_\tp \nu_\tm^2) + \cdots \label{virial}
\end{eqnarray}
Note that this expansion does converge under the assumption (\ref{Bradlow}).
Replicating what was already known for the Abelian Higgs model at criticality, considered in \cite{ManNas}, we thus see that the equation of state (\ref{eqstate}) incorporating the full thermodynamic limit, as well as its virial coefficients (\ref{virial}), 
are independent of the topology (i.e.~the genus $g$) of $\Sigma$.

For fixed $\tau$,  the virial coefficients corresponding to each $(\pm)$-vortex species alone, i.e.\ the coefficients of $\nu_\tpm^n$ in (\ref{virial}), 
are independent of $T$ and generated by a geometric series --- exactly as they would if they were determined by two separate Clausius' equations of state; this is also in agreement with the Abelian Higgs model (see \cite{ManSut}, p.~239). The novelty is the presence of mixed terms, involving powers of both $\nu_\tp$ and $\nu_\tm$, all of them with {\em negative} rather than positive virial coefficients.

If we also choose to discard (as negligible in the thermodynamic limit) the terms involving the genus $g$ in the entropy formula (\ref{firstentropy}), it simplifies to
\begin{equation} \label{newentropy}
S\approx 
K\sum_{\sigma=\pm} k_\sigma \log \frac{{\rm e}^2 KTV}{\hbar^2 k_\sigma} +  K  \sum_{\sigma=\pm} k_\sigma \log  \left[ 1-\sigma\tau - \frac{2\pi (k_\sigma-k_{\tm\sigma})}{V} \right].
\end{equation}
This plainer version of $S$ now has the advantage of being homogeneous of first order in the two particle species, as a standard entropy is expected to be.
We can recognize the first term in (\ref{newentropy}) as the entropy of a mixture of ideal gases of two species (labelled by $\sigma=\pm$), but our model yields an extra term
representing a departure from ideal-gas behaviour.

We will now calculate the entropy of mixing $\Delta S^{\rm mix}$ for the system of vortex gases. 
At a given temperature $T$,  this is defined as the difference between the entropy of the mixture (with densities $\nu_\tpm$ in a volume $V$) and the entropy of a system consisting of the two gas species in separated containers with volumes $V_\tpm$ adding up to $V$ at the same pressure (for ideal gases, this is equivalent to requiring the particle density in each container to equal the total particle density in the original mixture; see  \cite{Cal}, p.~69).
We first compute the partial volumes $V_\tpm$ in the definition from the requirements
$V_\tp+V_\tm=V$ and (making use of the equation of state (\ref{eqstate}))
$$
 \frac{(1-\tau)\frac{k_\tp}{V_\tp}}{1-\tau - 2\pi \frac{k_\tp}{V_\tp}} =  \frac{(1+\tau)\frac{k_\tm}{V_\tm}}{1+\tau - 2\pi \frac{k_\tm}{V_\tm}};
$$
these lead to
$$
V_\tpm=\frac{V\pm \frac{4\pi\tau}{1-\tau^2} k_\tmp}{1+\frac{k_\tmp}{k_\tpm}}.
$$
It should be noted that the thought experiment of isolating the two gases requires (according to~(\ref{Bradlow})) that the separate bounds 
$$
\frac{k_\tpm}{V_\tpm} \le  \frac{1 \mp \tau}{2\pi}
$$
be satisfied, which together amount to a lower bound on the total volume:
$$
\frac{V}{2\pi}\ge \frac{k_\tp}{1-\tau} + \frac{k_\tm}{1+\tau}.
$$
The entropy of mixing is obtained by subtracting from (\ref{newentropy}) the sum
$$
K\sum_{\sigma=\pm}  k_\sigma \left[  \log \frac{{\rm e}^2 KTV_\sigma}{\hbar^2 k_\sigma} +   \log \left(  1-\sigma\tau - \frac{2\pi k_\sigma}{V_\sigma} \right)\right],
$$
and this yields
$$
\Delta S^{\rm mix} = K\sum_{\sigma=\pm} k_\sigma \log \left[ \left( 1+\frac{\nu_{\tm \sigma}}{ \nu_\sigma}\right) \frac{1-\sigma\tau-2\pi(\nu_\sigma-\nu_{\tm \sigma})}{1-\sigma\tau -2\pi [\nu_\sigma + (\frac{1-\sigma\tau}{1+\sigma\tau}) \nu_{\tm \sigma}]} \right].
$$
This quantity is manifestly positive under all assumptions made on the parameters. It should be compared with the classical Gibbs entropy of mixing
$$
\Delta S^{\rm mix}_{\rm ideal} = K\sum_{\sigma=\pm} k_\sigma \log \left( 1+\frac{\nu_{\tm \sigma}}{ \nu_\sigma}\right)
$$
for a system of two ideal gases.

\subsection*{Acknowledgements} 

This project was started as part of the activities of a Junior Trimester Program on ``Mathematical Physics'' hosted at the Hausdorff Research Institute for Mathematics (HIM), University of Bonn, in 2012. We would like to thank HIM for support, our guest Jo\~ao Baptista for remarks about the shrinking limit described in Section~\ref{sec:shrink}, and Marcel B\"okstedt (Aarhus) for many discussions  related to our project. NR is thankful to the School of Mathematics at the University of Leeds,  IH\' ES, the Mainz Institute for Theoretical Physics, as well as the Galileo Galilei Institute for Theoretical Physics in Arcetri for hospitality while this paper was in preparation. This work was supported by the UK Engineering and Physical Sciences Research Council under grant number
EP/P024688/1.

\bibliographystyle{numsty}

\end{document}